\documentclass[11pt,colorlinks]{amsart}
\pdfoutput=1

\usepackage{hyperref}
\usepackage[obeyFinal,textsize=scriptsize,shadow,loadshadowlibrary]{todonotes}
\usepackage[shortlabels]{enumitem}
\usepackage{amsmath}
\usepackage{amssymb}
\usepackage{amsthm}
\usepackage{booktabs}
\usepackage[nameinlink]{cleveref}
\usepackage{derivative}
\usepackage{graphicx}
\usepackage{mathdots}
\usepackage{mathrsfs}
\usepackage{mathtools}
\usepackage{microtype}

\usepackage{asymptote}
\usepackage{tikz-cd}
\usetikzlibrary{decorations.pathmorphing}

\allowdisplaybreaks

\usepackage[backend=biber,backref=true,style=alphabetic]{biblatex}

\usepackage{thmtools}
\declaretheorem[name=Theorem,numberwithin=section]{theorem}
\declaretheorem[name=Lemma,sibling=theorem]{lemma}
\declaretheorem[name=Proposition,sibling=theorem]{proposition}
\declaretheorem[name=Corollary,sibling=theorem]{corollary}
\declaretheorem[name=Conjecture,sibling=theorem]{conjecture}
\declaretheorem[name=Hypothesis,sibling=theorem]{hypothesis}
\declaretheorem[name=Assumption,sibling=theorem]{assume}
\declaretheorem[name=Definition,sibling=theorem,style=definition]{definition}
\declaretheorem[name=Example,sibling=theorem,style=definition]{example}
\declaretheorem[name=Remark,sibling=theorem,style=definition]{remark}

\providecommand{\ol}{\overline}

\providecommand{\eps}{\varepsilon}
\providecommand{\half}{\frac{1}{2}}
\providecommand{\inv}{^{-1}}
\providecommand{\CC}{\mathbb C}
\providecommand{\FF}{\mathbb F}

\providecommand{\QQ}{\mathbb Q}

\providecommand{\ZZ}{\mathbb Z}
\providecommand{\ts}{\textsuperscript}

\providecommand{\ii}{\item}
\newcommand{\surjto}{\twoheadrightarrow}

\DeclareMathOperator{\BC}{BC}
\DeclareMathOperator{\End}{End}
\DeclareMathOperator{\Gal}{Gal}
\DeclareMathOperator{\GK}{GK}
\DeclareMathOperator{\GL}{GL}
\DeclareMathOperator{\Hom}{Hom}
\DeclareMathOperator{\Int}{Int}
\DeclareMathOperator{\Lie}{Lie}
\DeclareMathOperator{\Mat}{Mat}
\DeclareMathOperator{\Norm}{N}
\DeclareMathOperator{\Nm}{Nm}
\DeclareMathOperator{\Orb}{Orb}

\DeclareMathOperator{\PU}{PU}

\DeclareMathOperator{\SO}{SO}
\DeclareMathOperator{\Sat}{Sat}
\DeclareMathOperator{\Serre}{ST}
\DeclareMathOperator{\Spec}{Spec}
\DeclareMathOperator{\Spf}{Spf}
\DeclareMathOperator{\SU}{SU}
\DeclareMathOperator{\Sym}{Sym}
\DeclareMathOperator{\Tr}{Tr}
\DeclareMathOperator{\U}{U}
\DeclareMathOperator{\Vol}{Vol}

\DeclareMathOperator{\diag}{diag}
\DeclareMathOperator{\id}{id}
\DeclareMathOperator{\rproj}{proj}
\DeclareMathOperator{\tr}{tr}

\newcommand{\cc}{\mathbf c}
\newcommand{\nn}{\mathbf n}
\newcommand{\uu}{\mathbf u}
\newcommand{\vv}{\mathbf v}
\newcommand{\DD}{\mathcal D}
\newcommand{\HH}{\mathcal H}
\newcommand{\MM}{\mathcal M}
\newcommand{\EE}{\mathbb E}
\newcommand{\TT}{\mathbb T}
\newcommand{\VV}{\mathbb V}
\newcommand{\XX}{\mathbb X}

\newcommand{\RZ}{\mathcal N} 
\newcommand{\ZD}{\mathcal Z} 
\newcommand{\ZO}[1]{\mathcal Z^{\dagger}_{\SO(#1)}}
\newcommand{\DT}{\mathcal D^{\tr = 0}}
\newcommand{\ODT}{\OO_\DD^{\tr = 0}}

\newcommand{\guv}{{(\gamma, \uu, \vv^\top)}}
\newcommand{\oneV}{\mathbf{1}_{\OO_F^n \times (\OO_F^n)^\vee}}
\newcommand{\rs}{_{\mathrm{rs}}}
\newcommand{\TTr}{\mathcal{T}_{\mathbf{1}_{K, \le r}}}
\newcommand{\fr}{\mathbf{1}_{K, \le r}}

\newcommand{\OO}{O} 
\newcommand{\Sheaf}{\mathcal{O}} 

\newcommand{\jiao}{\mathop{\otimes}^{\mathbf{L}}} 

\usepackage{listings}
\hypersetup{pdfauthor={Evan Chen},pdftitle={Semi-Lie arithmetic fundamental lemma for the full spherical Hecke algebra}}

\usepackage{amsaddr}
\title[Semi-Lie Hecke AFL]
{Semi-Lie arithmetic fundamental lemma for the full spherical Hecke algebra}

\author{Evan Chen}
\date{\today}
\address{Department of Mathematics, Massachusetts Institute of Technology}
\email{evanchen@alum.mit.edu}

\begin{document}

\begin{abstract}
  As an analog to the Jacquet-Rallis fundamental lemma that appears in the
  relative trace formula approach to the Gan-Gross-Prasad conjectures,
  the arithmetic fundamental lemma was proposed by Wei Zhang and used in an approach
  to the arithmetic Gan-Gross-Prasad conjectures.
  The Jacquet-Rallis fundamental lemma was recently generalized by Spencer Leslie
  to a statement holding for the full spherical Hecke algebra.
  In the same spirit, Li, Rapoport, and Zhang
  have recently formulated a conjectural generalization of the arithmetic
  fundamental lemma to the full spherical Hecke algebra.
  This paper formulates another analogous conjecture for the semi-Lie version
  of the arithmetic fundamental lemma proposed by Yifeng Liu.
  Then this paper produces explicit formulas for particular cases
  of the weighted orbital integrals in the two conjectures mentioned above,
  and proves the first non-trivial case of the conjecture.
\end{abstract}

\maketitle

\tableofcontents
\newpage

\section{Introduction}
Throughout this whole paper, $p > 2$ is a prime,
$F$ is a finite extension of $\QQ_p$,
and $E/F$ is an unramified quadratic field extension.

\subsection{Brief history and motivation for the arithmetic fundamental lemma}
The primary motivation for this paper arises from
the study of conjectured variants of the arithmetic fundamental lemma
for spherical Hecke algebras proposed in \cite{ref:AFLspherical}.
This section briefly provides an overview of the historical context
that led to the formulation of these conjectures.
This history is also summarized in \Cref{fig:history}.

Because this subsection is meant for motivation only, in this survey we do not give
complete definitions or statements, being content to outline a brief gist.
A more detailed account can be found in \cite{ref:survey}.

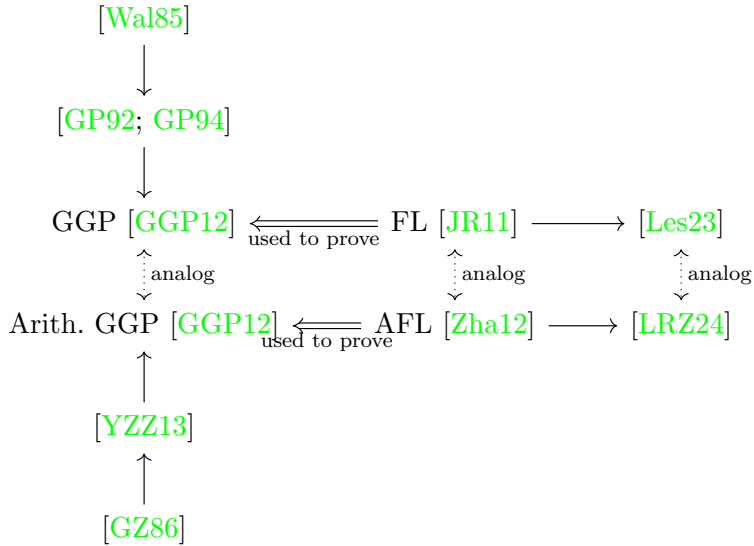
\begin{figure}[ht]
  \centering
  \begin{tikzcd}
      \text{\cite{ref:waldspurger}} \ar[d] \\
      \text{\cite{ref:GP1,ref:GP2}} \ar[d] \\
    \text{GGP \cite{ref:GGP}}
      \ar[d, dotted, leftrightarrow, "\text{analog}"]
      & \ar[l, Rightarrow, "\text{used to prove}"] \text{FL \cite{ref:JR}}
        \ar[d, dotted, leftrightarrow, "\text{analog}"] \ar[r]
      & \text{\cite{ref:leslie}}
        \ar[d, dotted, leftrightarrow, "\text{analog}"] \\
    \text{Arith.\ GGP \cite{ref:GGP}}
      & \ar[l, Rightarrow, "\text{used to prove}"] \text{AFL \cite{ref:AFL}} \ar[r]
      & \text{\cite{ref:AFLspherical}} \\
    \text{\cite{ref:GZshimura}} \ar[u] \\
    \text{\cite{ref:gross_zagier}} \ar[u]
  \end{tikzcd}
  \caption{The history behind the fundamental lemma and its arithmetic counterpart.
    Unlabeled arrows denote generalizations.}
  \label{fig:history}
\end{figure}

\subsubsection{The GGP conjectures, and the fundamental lemma of Jacquet-Rallis}
In modern arithmetic geometry, a common theme is that there are deep connections
between geometric data with the values of related $L$-functions.

This story begins with a result of
Waldspurger \cite{ref:waldspurger} which showed a formula
relating the nonvanishing of an automorphic period integral
to the central value of the same $L$-functions.
Later, a conjecture that generalizes Waldspurger's formula
was proposed by Gross-Prasad in \cite{ref:GP1,ref:GP2}.
This was further generalized to a series of conjectures
now known as the Gan-Gross-Prasad (GGP) conjectures,
which were proposed in 2012 in \cite{ref:GGP};
they generalize the Gross-Prasad conjecture to different classical groups.
Specifically, the GGP conjecture predict the nonvanishing of a period integral
based on the values of the $L$-function of a certain cuspidal automorphic representation.

In 2011, Jacquet-Rallis \cite{ref:JR} proposed an approach to the Gross-Prasad conjectures
for unitary groups via a relative trace formula (RTF).
The idea is to compare an RTF for the general linear group to one for a unitary group.
This approach relies on a so-called \emph{fundamental lemma},
which links values of certain orbital integrals
over two reductive groups over a non-Archimedean local field.

Let's be a bit more precise about what this fundamental lemma says.
Let $\VV_n^+$ denote the split $E/F$-Hermitian space of dimension $n$ (unique up to isomorphism),
fix a unit vector $w_0$ in it,
and let $(\VV_n^+)^\flat$ be the orthogonal complement of the span of $w_0$.
Let $(G')^\flat \coloneqq \GL_{n-1}(E)$, $G' \coloneqq \GL_n(E)$,
$G^\flat \coloneqq \U((\VV_n^+)^\flat)(F)$ and $G \coloneqq \U(\VV_n^+)(F)$.
For certain
\[ \gamma \in (G')^\flat \times G', \qquad g \in G^\flat \times G \]
the Jacquet-Rallis fundamental lemma proposes a relation between two orbital integrals.
Specifically, it supplies a relation between
\begin{itemize}
\item the orbital integral of $\gamma$ with respect to
  the indicator function $\mathbf{1}_{(K')^\flat \times K'}$
  of the natural hyperspecial compact subgroup
  \[ (K')^\flat \times K' \subset (G')^\flat \times G' = \GL_{n-1}(E) \times \GL_n(E); \]
  and
\item the orbital integral of $g$ with respect to
  the indicator function $\mathbf{1}_{K^\flat \times K}$
  of the natural hyperspecial compact subgroup
  \[ K^\flat \times K \subset G^\flat \times G = \U((\VV_n^+)^\flat)(F) \times \U(\VV_n^+)(F). \]
\end{itemize}
In other words, it states that
\begin{equation}
  \Orb(\gamma, \mathbf{1}_{(K')^\flat \times K'}) = \omega(\gamma) \Orb(g, \mathbf{1}_{K^\flat \times K})
  \label{eq:old_FL}
\end{equation}
where $\omega(\gamma)$ is a suitable \emph{transfer factor}.
The fundamental lemma has since been proved completely;
a local proof was given by Beuzart-Plessis \cite{ref:BeuzartPlessis}
while a global proof was given for large characteristic by W.\ Zhang \cite{ref:Wei2021}.

\subsubsection{The arithmetic GGP conjectures, and the arithmetic fundamental lemma}
At around the same time Waldspurger's formula was published,
Gross-Zagier \cite{ref:gross_zagier} proved a formula
relating the height of Heegner points
on certain modular curves to the derivative at $s=1$ of certain $L$-functions.
The Gross-Zagier formula was then generalized over several decades,
culminating in \cite{ref:GZshimura} where the formula is established
for Shimura curves over arbitrary totally real fields.

An arithmetic analogue of the original Gan-Gross-Prasad conjectures,
which we henceforth refer to as \emph{arithmetic GGP} \cite{ref:GGP},
can then be formulated, further generalizing Gross-Zagier's formula.
Here the modular curves in Gross-Zagier
are replaced with higher dimensional Shimura varieties.
Rather than the period integrals considered previously,
one instead takes intersection numbers of cycles on some Shimura varieties.
Specifically, if one considers the Shimura variety associated to a classical group,
the arithmetic GGP conjecture predicts a relation between intersection numbers
on this Simura variety with the central derivative of automorphic $L$-functions.

By analogy to the work Jacquet-Rallis \cite{ref:JR},
the arithmetic GGP conjectures should have a corresponding
\emph{arithmetic fundamental lemma} (henceforth AFL),
which was proposed by W.\ Zhang \cite[Conjecture 2.9]{ref:AFL}.
The arithmetic fundamental lemma then relates the derivative
of the weighted orbital integral with respect to the indicator function
$\mathbf{1}_{(K')^\flat \times K'} \in \HH((G')^\flat \times G, (K')^\flat \times K')$, that is
\[ \left. \pdv{}{s} \right\rvert_{s=0} \Orb(\gamma, \mathbf{1}_{(K')^\flat \times K'}, s) \]
for $\gamma \in (G')^\flat \times G'$,
to arithmetic intersection numbers on a certain Rapoport-Zink formal moduli space.
The AFL in \cite{ref:AFL} has since been proven over $p$-adic fields for any large
prime $p$ in Mihatsch-Zhang \cite{ref:MZ2021}, W.\ Zhang \cite{ref:Wei2021},
and then for any odd prime $p$ by Z.\ Zhang \cite{ref:Zhiyu}.

\subsubsection{The semi-Lie version of the AFL proposed by Liu}
There is another different version of the AFL, proposed by Yifeng Liu in
\cite[Conjecture 1.12]{ref:liuFJ},
which is often referred to as the \emph{semi-Lie version} of the AFL.
Its statement has been shown to be equivalent to AFL,
see \cite[Remark 1.13]{ref:liuFJ} and is thus now proven.
In contrast, the original AFL proposed by Zhang in \cite[Conjecture 2.9]{ref:AFL}
is sometimes referred to as the \emph{group version}.

A more detailed account of this equivalence is described in \cite[\S1.4]{ref:liuFJ}.

\subsubsection{Generalizations of FL and AFL to the full spherical Hecke algebra}
Recently it was shown by Leslie \cite{ref:leslie} that in fact
\eqref{eq:old_FL} holds in greater generality where the indicator function
$\mathbf{1}_{K^\flat \times K}$ can be replaced by any element in the spherical
Hecke algebra $\varphi \in \HH((G')^\flat \times G', (K')^\flat \times K')$.
In that case, $\mathbf{1}_{(K')^\flat \times K}$ needs to be replaced
by the corresponding element $\varphi'$ under a certain base change homomorphism
\begin{align*}
  \HH((G')^\flat \times G', (K')^\flat \times K') &\to \HH(G^\flat \times G, K^\flat \times K) \\
  \varphi' &\mapsto \varphi
\end{align*}

In that case, the identity \eqref{eq:old_FL} still hold as
\begin{equation}
  \Orb(\gamma, \varphi') = \omega(\gamma) \Orb(g, \varphi).
  \label{eq:eq:leslie_FL}
\end{equation}
To complete the analogy illustrated in \Cref{fig:history},
there should thus be a generalization of the AFL in which
$\mathbf{1}_{(K')^\flat \times K}$ is replaced by any element of the Hecke algebra
$\HH((G')^\flat \times G, (K')^\flat \times K)$.
This conjectural generalization
(for the group version of the AFL) is \cite{ref:AFLspherical}
and we discuss it momentarily;
while the analogous conjectural generalization for the semi-Lie version of the AFL
is our \Cref{conj:semi_lie_spherical} also discussed in the next section.

\subsection{Formulation of AFL conjectures to the full spherical Hecke algebra}
\subsubsection{The inhomogeneous version of the arithmetic fundamental lemma for spherical Hecke algebras proposed by Li-Rapoport-Zhang}
In contrast to the vague motivational cheerleading in the previous section,
starting in this section we will give more precise statements,
even though we will necessarily need to reference definitions appearing in later sections.

Retain the notation $G' \coloneqq \GL_n(E)$, and $G \coloneqq \U(\VV_n^+)(F)$,
with $K' \subset G'$ and $K \subset G$ the natural hyperspecial compact subgroups.
Also, let $q$ denote the residue characteristic of $F$.
Moreover, define the symmetric space
\[ S_n(F) \coloneqq \left\{ g \in \GL_n(E) \mid g \bar{g} = \id_n \right\}. \]
Finally, let $\VV_n^-$ be the non-split Hermitian space of dimension $n$
(unique up to isomorphism),
while $\VV_n^+$ continues to denote the split one (again unique up to isomorphism).

For concreteness, we focus on the inhomogeneous version
of the arithmetic fundamental lemma, which is \cite[Conjecture 6.2.1]{ref:AFLspherical}.
This allows us to deal with just $G'$ instead of $(G')^\flat \times G'$, etc.,
so that the Hecke algebra $\HH((G')^\flat \times G', (K')^\flat \times K')$
can be replaced by the simpler one $\HH(\GL_n(E)) \coloneqq \HH(G', K')$.
Similarly, $\HH(G^\flat \times G, K^\flat \times K)$
can be replaced by the simpler $\HH(\U(\VV_n^+)) \coloneqq \HH(G, K)$.

The AFL conjecture provides a bridge between a geometric left-hand side
(given by an intersection number)
and an analytic right-hand side (given by a weighted orbital integral).
Stating it requires several pieces of data.
We only mention these pieces by name here, with definitions given later:
\begin{itemize}
  \ii On the geometric side, we have an intersection number.
  It uses the following ingredients.
  \begin{itemize}
    \ii We choose a regular semisimple element $g \in \U(\VV_n^-)\rs$.
    (The notation $\U(\VV_n^-)\rs$ denotes the regular semisimple elements of $\U(\VV_n^-)$, etc.
    The notion of regular semisimple is defined in \Cref{def:regular}.)
    \ii We choose a function $f \in \HH(\U(\VV_n^+))$ from the spherical Hecke algebra,
    defined in \Cref{ch:background}.
    \ii We define a certain \emph{intersection number} $\Int(g, f)$
    in \Cref{def:intersection_number_inhomog}.
    These intersection numbers take place in a Rapoport-Zink space
    described in \Cref{ch:geo}.
  \end{itemize}

  \ii On the analytic side, we have a weighted orbital integral.
  It uses the following ingredients.
  \begin{itemize}
    \ii We choose a regular semisimple element $\gamma \in S_n(F)\rs$.
    \ii We choose a test function $\phi$ which comes
    from a certain $\HH(\GL_n(E))$-module that we will denote $\HH(S_n(F))$.
    This module $\HH(S_n(F))$ is defined in \Cref{ch:background}.
    \ii The weighted orbital integral $\Orb(\gamma, \phi, s)$
    is itself defined in \Cref{def:orbital0}.
    (It is connected to an unweighted orbital integral on the unitary group
    according to \Cref{thm:rel_fundamental_lemma}.)
    We abbreviate
    \[ \partial \Orb(\gamma, \phi)
      \coloneqq \left. \pdv{}{s} \right\rvert_{s=0} \Orb(\gamma, \phi, s). \]
    \ii There is also an extra transfer factor $\omega \in \{\pm1\}$
    which we define in \Cref{ch:transf}.
  \end{itemize}

  \ii We need a way to connect the inputs between the two parts of our conjecture.
  Specifically, $f$ and $\phi$ need to be linked, and $g$ and $\gamma$ need to be linked.
  This is done as follows.
  \begin{itemize}
    \ii Once the regular semisimple element $g \in \U(\VV_n^-)\rs$ is chosen,
    we require $\gamma \in S_n(F)\rs$ to be a \emph{matching} element.
    This matching is defined in \Cref{def:matching_inhomog}.
    (Alternatively, one could imagine picking $\gamma \in S_n(F)\rs$ first
    and finding corresponding $g$.
    It turns out $\gamma$ will match an element of $\U(\VV_n^\pm)\rs$ in general,
    and the conjecture is only formulated in the case where $g \in \U(\VV_n^-)$.)

    \ii Once $f \in \HH(\U(\VV_n^+))$ is chosen, we select
    \[ \phi = (\BC_{S_n}^{\eta^{n-1}})^{-1}(f) \]
    to be the image of $f$ under a \emph{base change}.
    This base change is defined and then calculated explicitly for $n = 3$ in \Cref{ch:satake}.
  \end{itemize}
\end{itemize}
With all our protagonists now having names and references,
we can now state the conjecture proposed in \cite{ref:AFLspherical}.

\begin{conjecture}
  [Inhomogeneous version of the AFL for the full spherical Hecke algebra,
  {\cite[Conjecture 6.2.1]{ref:AFLspherical}}]
  \label{conj:inhomog_spherical}
  Let $f \in \HH(\U(\VV_n^+))$ be any element of the Hecke algebra, and let
  \[ \phi = (\BC_{S_n}^{\eta^{n-1}})^{-1}(f) \in \HH(S_n(F)) \]
  be its image under base change as defined in \Cref{ch:satake}.
  Then for matching (as defined in \Cref{def:matching_inhomog}) regular semisimple elements
  \[ g \in \U(\VV_n^-)\rs \longleftrightarrow \gamma \in S_n(F)\rs \]
  we have
  \begin{equation}
    \Int\left( (1,g), \mathbf{1}_{K^\flat} \otimes f \right)
    = -\frac{\omega(\gamma)}{\log q} \partial \Orb(\gamma, \phi)
    \label{eq:inhomog}
  \end{equation}
  where the weighted orbital integral $\Orb(\dots)$ is defined in \Cref{def:orbital0},
  the transfer factor $\omega$ is defined in \Cref{ch:transf},
  and the intersection number $\Int(\dots)$ is defined in \Cref{ch:geo}.
\end{conjecture}

At present, the (inhomogeneous) AFL is the case where $f = \mathbf{1}_K$,
and is thus proven.
Note that in the case of interest where $\gamma \in S_n(F)\rs$
matches an element of $\U(\VV_n^-)\rs$ (rather than $\U(\VV_n^+)\rs$),
the actual value of $\Orb(\gamma, \phi, s)$ at $s = 0$ is always zero
by \Cref{thm:rel_fundamental_lemma};
so the conjecture instead looks at the first derivative at $s = 0$.

The generalized conjecture is also proved in full for
$n = 2$ in \cite[Theorem 1.0.1]{ref:AFLspherical}
(in that reference, our $n$ denotes the $n+1$ in \emph{loc.\ cit.}).
The part of the calculation involving the weighted orbital integral has two parts:
\begin{itemize}
  \ii The calculation makes $\BC_{S_{n}}^{\eta^{n-1}}$
  completely explicit in a natural basis for $n = 2$.
  The result is \cite[Lemma 7.1.1]{ref:AFLspherical}.

  \ii The calculation makes explicit the value of the weighted orbital integral
  \[ \Orb(\gamma, \phi, s) \]
  for any $\gamma \in S_n(F)\rs$ and $\phi \in \HH(S_n(F))$,
  in terms of invariants of $\gamma$ and a decomposition of $\phi$ in a natural basis.
  The result is \cite[Proposition 7.3.2]{ref:AFLspherical}.
\end{itemize}
Combining these two (hence obtaining the right-hand side of \eqref{eq:inhomog})
with a calculation of intersection numbers in \cite[Corollary 7.4.3]{ref:AFLspherical}
(which is the left-hand side of \eqref{eq:inhomog})
shows that \Cref{conj:inhomog_spherical} holds for $n = 2$,
cf.\ \cite[Theorem 7.5.1]{ref:AFLspherical}.

\subsubsection{A proposed arithmetic fundamental lemma for spherical Hecke algebras in the semi-Lie case}
The primary focus of this paper is an analogous conjecture to \Cref{conj:inhomog_spherical}
for the semi-Lie version (also called the Fourier-Jacobi case).
It serves to complete the analogy given in \Cref{tab:semi_lie_analogy}.

\begin{table}[ht]
  \centering
  \begin{tabular}{lll}
    \toprule
    Version & AFL for $\mathbf{1}$ (now proven) & Full spherical Hecke \\
    \midrule
    Group & \cite[Conjecture 2.9]{ref:AFL} & \cite[Conjecture 6.2.1]{ref:AFLspherical} \\
    Semi-Lie & \cite[Conjecture 1.12]{ref:liuFJ} & \Cref{conj:semi_lie_spherical} \\
    \bottomrule
  \end{tabular}
  \caption{Table showing the analogy between the proposed
    \Cref{conj:semi_lie_spherical} and the existing conjectures.}
  \label{tab:semi_lie_analogy}
\end{table}

In this variation, as in \cite{ref:liuFJ},
rather than matching $g \in \U(\VV_n^-)\rs$ to $\gamma \in S_n(F)\rs$,
we consider an augmented space larger than $\U(\VV_n^-)$ and $S_n(F)$.
Specifically, one considers a matching between tuples of regular semisimple elements
\[ (g, u) \in (\U(\VV_n^-) \times \VV_n^-)\rs
  \longleftrightarrow (\gamma, \uu, \vv^\top) \in (S_n(F) \times V'_n(F))\rs \]
where $V'_n(F) = F^n \times (F^n)^\vee$ (defined in \Cref{def:matching_semi_lie})
consists of pairs of column vectors and row vectors of length $n$,
and the space $\VV_n^-$ is defined in \Cref{def:VV_n_nonsplit}.
The notion of \emph{matching} is defined in \Cref{def:matching_semi_lie} as well.

Meanwhile, we still use the same test functions $f$ and $\phi$,
as we did for \cite[Conjecture 6.2.1]{ref:AFLspherical}.
The derivative of interest will be denoted
\[ \partial \Orb(\guv, \phi) \coloneqq
  \left. \pdv{}{s} \right\rvert_{s=0}
  \Orb((\gamma, \uu, \vv^\top), \phi \otimes \oneV, s). \]
Finally, we also update the definition of intersection number
to accommodate the new term $u$ in \Cref{def:intersection_number_semi_lie_spherical}.
\begin{conjecture}
  [Semi-Lie version of the AFL for the full spherical Hecke algebra]
  Let $f \in \HH(\U(\VV_n^+))$ be any element of the Hecke algebra, and let
  \[ \phi = (\BC_{S_n}^{\eta^{n-1}})^{-1}(f) \in \HH(S_n(F)) \]
  be its image under base change defined in \Cref{ch:satake}.
  Then for matching (as defined in \Cref{def:matching_semi_lie}) regular semisimple elements
  \[ (g, u) \in (\U(\VV_n^-) \times \VV_n^-)\rs \longleftrightarrow
    \guv \in (S_n(F) \times V'_n(F))\rs \]
  we have
  \begin{equation}
    \Int\left( (g,u), f \right) = -\frac{\omega\guv}{\log q} \partial \Orb(\guv, \phi \otimes \oneV)
  \end{equation}
  where the orbital integral $\Orb(\dots)$ is defined in \Cref{def:orbitalFJ},
  the transfer factor is defined in \Cref{ch:transf},
  and the intersection number $\Int(\dots)$ is defined in \Cref{ch:geo}.
  \label{conj:semi_lie_spherical}
\end{conjecture}
Note that in this version the new orbital integral $\Orb(\guv, \phi \otimes \oneV, s)$
is defined similarly.
However, as far as we know, no analog of \Cref{thm:rel_fundamental_lemma}
(linking it to an unweighted orbital integral on the unitary side) appears in the literature.
Thus we record the corresponding statement as \Cref{conj:rel_fundamental_lemma_semilie}.
Like before, \Cref{conj:rel_fundamental_lemma_semilie}
predicts that $\Orb(\guv, \phi \otimes \oneV, 0) = 0$ in the case of interest,
which in this case can be checked independently.

\begin{remark}
  For $n = 1$, the Hecke algebra $\HH(S_n(F))$ is trivial
  and therefore \Cref{conj:semi_lie_spherical}
  becomes a special case of the known result \cite{ref:liuFJ}.
  Therefore $n=2$ is the first case of \Cref{conj:semi_lie_spherical} worth examining.
\end{remark}

\subsubsection{A proposed large image conjecture}
\label{sec:intro_large_kernel}

We let $S_n(F)\rs^-$ (resp.\ $(S_n(F) \times V'_n(F))\rs^-$)
denote the subsets of $S_n(F)\rs$ (resp.\ $(S_n(F) \times V'_n(F))\rs$)
consisting of regular semisimple elements that also match with an element of
$\U(\VV_n^-)\rs$ (resp.\ $(\U(\VV_n^-) \times \VV_n^-)\rs$),
i.e.\ those for which \Cref{conj:inhomog_spherical}
and \Cref{conj:semi_lie_spherical} apply.

In \cite{ref:AFLspherical}, it was observed that for $n = 2$
there was a rather large space of $\phi \in \HH(S_2(F))$ such that
\[ \partial \Orb \left(\gamma, \phi \right) = 0 \]
held identically across all $\gamma \in S_2(F)\rs^-$.
If we consider the map
\begin{align*}
  \partial\Orb \colon \HH(S_2(F)) &\to \mathcal{C}^\infty(S_2(F)\rs^-) \\
  \phi &\mapsto \left( \gamma \mapsto \partial \Orb \left(\gamma, \phi \right) \right)
\end{align*}
then \cite[Theorem 8.2.3]{ref:AFLspherical} in fact
shows this map has a kernel of codimension $2$
(equivalently, the image of the map was only two-dimensional).
They thus stated \cite[Conjecture 1.0.2]{ref:AFLspherical} which asserts that for $n \ge 2$,
a similarly defined map (albeit for more than one Hecke algebra) has a large kernel.

It is therefore natural to ask whether a similar large kernel result
could hold for the analogous orbital integral in the semi-Lie case.
We instead propose the following \emph{large image} conjecture,
which we have proved for $n = 2$.
\begin{conjecture}
  [Large image conjecture for $(S_n(F) \times V'_n(F))\rs^-$]
  \label{conj:kernel_semi_lie}
  Let $n \ge 2$.
  The map
  \begin{align*}
    \partial\Orb \colon \HH(S_n(F)) &\to \mathcal{C}^\infty\left( (S_n(F) \times V'_n(F))^- \right) \\
    \phi &\mapsto \left( \guv \mapsto \partial \Orb \left(\guv, \phi \right) \right)
  \end{align*}
  is injective.
\end{conjecture}
In fact for $n=2$ we have a more precise result
showing that the injectivity essentially comes from $v(\uu\vv^\top)$ alone in this case.
See \Cref{thm:semi_lie_ker_trivial} and \Cref{thm:semi_lie_ker_huge} momentarily.

\subsection{Results}

\subsubsection{Formulas for the orbital side of the semi-Lie AFL conjecture for $n=2$}
\label{sec:semi_lie_intro_formulas}
The main case of interest in this thesis is the new conjectured AFL
for the spherical Hecke algebra in the semi-Lie situation in the
specific case $n = 2$ where one can provide evidence for the conjecture.
On the orbital side, the various ingredients can be described concretely
in the following way:
\begin{itemize}
  \ii The Hecke algebra $\HH(S_2(F))$ has a natural basis of
  indicator functions $\mathbf{1}_{K', \le r}$ for each $r \ge 0$;
  see \Cref{ch:orbitalFJ0} for a definition.

  \ii Suppose $\guv \in (S_2(F) \times V'_2(F))\rs^-$.
  Then under the action $\GL_2(F)$ we may assume $\guv$ is of the form
  \[
    \guv = \left( \begin{pmatrix} a & b \\ c & d \end{pmatrix},
      \begin{pmatrix} 0 \\ 1 \end{pmatrix},
      \begin{pmatrix} 0 & e \end{pmatrix} \right)
    \in (S_2(F) \times V_2'(F))\rs^-
  \]
  (that is, we can find a $\GL_2(F)$-orbit representative of this form).
  The parameters $a$, $b$, $c$, $d$ need to satisfy certain dependencies
  for the matching to hold;
  these requirements are documented in \Cref{lem:semi_lie_params}.
\end{itemize}

Then we were able to derive the following fully explicit formula
in terms of the representative detailed in \Cref{lem:semi_lie_params}.
See \Cref{sec:proof_semi_lie_formula} for some concrete examples
and illustrations of \Cref{thm:semi_lie_formula}.
\begin{restatable}[Explicit orbital integral on $S_2(F) \times V'_2(F)$]{theorem}{semilieformula}
  \label{thm:semi_lie_formula}
  Let
  \[
    \guv = \left( \begin{pmatrix} a & b \\ c & d \end{pmatrix},
      \begin{pmatrix} 0 \\ 1 \end{pmatrix},
      \begin{pmatrix} 0 & e \end{pmatrix} \right)
    \in (S_2(F) \times V_2'(F))\rs^-
  \]
  satisfy the requirements in \Cref{lem:semi_lie_params}.
  Let $r \ge 0$.

  If $v(e) < 0$ or $v(b) + v(c) < -2r$, then
  \[ \Orb(\guv, \mathbf{1}_{K'_{S, \le r}} \otimes \oneV, s) = 0 \]
  holds identically for all $s \in \CC$.

  Otherwise define
  \[ \nn_\guv(k) \coloneqq \min\left( \left\lfloor \tfrac{k + (v(b)+r)}{2} \right\rfloor,
    \left\lfloor \tfrac{(2v(e)+v(c)+r)-k}{2} \right\rfloor, N \right) \]
  where
  \[ N \coloneqq \min \left(
      v(e), \tfrac{v(b)+v(c)-1}{2} + r,
      v(d-a) + r \right). \]
  Also, if $v(d-a) < v(e) - r$ and $v(b) + v(c) > 2v(d-a)$, then additionally define
  \begin{align*}
    \cc_\guv(k) &= \min\big( k - (2v(d-a)-v(b)+r), \\
      &\qquad (2v(e)+v(c)-2v(d-a)-r)-k, v(e)-v(d-a)-r \big).
  \end{align*}
  Otherwise define $\cc_\guv(k) = 0$.
  Then we have
  \begin{align*}
    \Orb(\guv, \mathbf{1}_{K'_{S, \le r}}, s)
    &= \sum_{k = -(v(b)+r)}^{2v(e)+v(c)+r} (-1)^k
    \left( 1 + q + q^2 + \dots + q^{\nn_\guv(k)} \right) (q^s)^k \\
    &+ \sum_{k = 2v(d-a)-v(b)+r}^{2v(e)+v(c)-2v(d-a)-r} (-1)^k \cc_\guv(k) q^{v(d-a) + r} (q^s)^k.
  \end{align*}
\end{restatable}

Differentiating this yields the following result:
\begin{restatable}[Derivative at $s=0$ for $S_2(F) \times V'_2(F)$]{corollary}{semiliederiv}
  \label{cor:semi_lie_derivative_single}
  Retain the setting of \Cref{thm:semi_lie_formula}.
  Also define $\varkappa \coloneqq v(e) - (v(d-a)+r)$.
  If both $\varkappa \ge 0$ and $v(b) + v(c) > 2v(d-a)$, then we have the formula
  \begin{align*}
    \frac{(-1)^{v(c)+r}}{\log q}
    &\partial \Orb(\guv, \mathbf{1}_{K'_{S, \le r}}) \\
    &= \sum_{j=0}^N \left( \frac{2v(e)+v(b)+v(c)+1}{2} + r - 2j \right) \cdot q^j \\
    & - q^{v(d-a)+r} \cdot
    \begin{cases}
      \frac{\varkappa}{2} & \text{if }\varkappa \equiv 0 \pmod 2 \\
      \left( v(e)+\frac{v(b)+v(c)}{2}-2v(d-a)-r \right) - \frac{\varkappa}{2}
      & \text{if }\varkappa \equiv 1 \pmod 2.
    \end{cases}
  \end{align*}
  Otherwise we instead have the formula
  \[
    \frac{(-1)^{v(c)+r}}{\log q}
    \partial \Orb(\guv, \mathbf{1}_{K'_{S, \le r}}) \\
    = \sum_{j=0}^N \left( \frac{2v(e)+v(b)+v(c)+1}{2} + r - 2j \right) \cdot q^j.
  \]
\end{restatable}

The formula simplifies even further if one considers instead
$\mathbf{1}_{K'_{S, \le r}} + \mathbf{1}_{K'_{S, \le (r-1)}}$;
and indeed we will see that this particular combination
comes up naturally in \Cref{ch:finale} with a special role
as the base change mentioned in \cite[Lemma 7.1.1]{ref:AFLspherical}.

\begin{restatable}[The special case $\partial \Orb(\guv, \mathbf{1}_{K'_{S, \le r}} + \mathbf{1}_{K'_{S, \le (r-1)}})$]{corollary}{semiliecombo}
  \label{cor:semi_lie_combo}
  Retain the setting of \Cref{thm:semi_lie_formula}.
  Also define $\varkappa \coloneqq v(e) - (v(d-a)+r)$.
  For $r \ge 1$ define
  \begin{align*}
    C &\coloneqq
    \begin{cases}
      \frac{\varkappa-1}{2}
        & \text{if } \varkappa > 0 \text{ is odd}
          \text{ and } v(b) + v(c) > 2v(d-a)  \\
      \frac{\varkappa+v(b)+v(c)-2v(d-a)-1}{2}
        & \text{if } \varkappa \ge 0 \text{ is even}
          \text{ and } v(b) + v(c) > 2v(d-a)  \\
      v(e) - N
        & \text{if } v(e) \ge \frac{v(b)+v(c)-1}{2} + r
        \text{ and } 2v(d-a) > v(b) + v(c) \\
      0 & \text{otherwise}
    \end{cases} \\
    C' &\coloneqq
    \begin{cases}
      C + 1 & \text{if } \varkappa \ge 0 \text{ and } v(b)+v(c) > 2v(d-a) \\
      0 & \text{otherwise}.
    \end{cases}
  \end{align*}
  Then
  \begin{align*}
    \frac{(-1)^{v(c)+r}}{\log q} &
    \partial\Orb(\guv, \mathbf{1}_{K'_{S, \le r}} + \mathbf{1}_{K'_{S, \le (r-1)}}) \\
    &= (q^N + q^{N-1} + \dots + 1) + C q^N + C' q^{N-1}
  \end{align*}
\end{restatable}

\begin{example}[Examples of \Cref{cor:semi_lie_combo}]
  We show some examples of \Cref{cor:semi_lie_combo}:
  \begin{itemize}
  \ii When $r=5$, $v(b) = -20$, $v(c) = 37$, $v(e) = 35$ and $v(d-a) > \frac{v(b)+v(c)}{2} = 8.5$
  the derivative in \Cref{cor:semi_lie_combo} equals
  \[ \log q \cdot (23q^{13} + q^{12} + q^{11} + q^{10} + q^9 + \dots + q + 1). \]
  \ii When $r = 6$, $v(b) = 10$, $v(c) = 5$, $v(e) = 7$, $v(d-a) > v(e)-r = 1$,
  the derivative in \Cref{cor:semi_lie_combo} equals
  \[ -\log q \cdot (q^7 + q^6 + q^5 + \dots + q + 1). \]
  \ii When $r = 8$, $v(b) = -101$, $v(c) = 1000$, $v(e) = 29$, $v(d-a) = 11$,
  the derivative in \Cref{cor:semi_lie_combo} equals
  \[ \log q \cdot (444 q^{19} + 445q^{18} + q^{17} + q^{16} + q^{15} + \dots + q + 1). \]
  \end{itemize}
\end{example}

\subsubsection{Kernel and image results for the semi-Lie orbital integral when $n=2$}
As we mentioned our earlier conjecture \Cref{conj:kernel_semi_lie}
is true for $n = 2$.
More precisely, we have the following two theorems.

\begin{restatable}[$\partial\Orb$ is injective even for fixed $\gamma \in S_2(F)$]{theorem}{semiliekertriv}
  \label{thm:semi_lie_ker_trivial}
  Fix any $\guv \in (S_2(F) \times V'_2(F))\rs^-$.
  Then there doesn't exist any nonzero function $\phi \in \HH(S_2(F))$ such that
  \[ \partial \Orb \left( (\gamma, \uu, \varpi^i \vv^\top), \phi, \right) = 0 \]
  holds for every integer $i$.
  Thus \Cref{conj:kernel_semi_lie} holds for $n = 2$.
\end{restatable}

In particular, for $n=2$ the map
\begin{align*}
  \partial\Orb \colon \HH(S_2(F)) &\to \mathcal{C}^\infty\left( (S_2(F) \times V'_2(F))^- \right) \\
  \phi &\mapsto \left( \guv \mapsto \partial \Orb \left(\guv, \phi \right) \right)
\end{align*}
is indeed injective.

\begin{restatable}[The kernel of $\partial\Orb$ is large for fixed $(\uu, \vv^\top) \in V'_2(F)$]{theorem}{semiliekerlarge}
  \label{thm:semi_lie_ker_huge}
  Let $N \ge 0$ be an integer.
  Consider all $\guv \in (S_2(F) \times V'_2(F))\rs^-$ for which $v(\uu\vv^\top) \le N$.
  Then the space of $\phi \in \HH(S_2(F))$ for which
  \[ \partial \Orb \left( \guv, \phi \right) = 0 \]
  holds for all such $\guv$ is a $\QQ$-vector subspace of $\HH(S_2(F))$
  whose codimension is at most $N + 2$.

  Moreover, this subspace of $\HH(S_2(F))$
  is not contained in any maximal ideal of $\HH(S_2(F))$
  when $\HH(S_2(F))$ is viewed as a ring under the isomorphism of \Cref{ch:satake}.
\end{restatable}

\begin{remark}
  [On formalizing large kernels]
  In each case the Hecke algebra is isomorphic as a $\QQ$-algebra to $\QQ[T]$
  for a single variable $T = Y+Y^{-1}$.
  So actually it's mildly surprising that we get a result on finite codimension.
  In general, a finite codimension vector subspace of $\QQ[T]$ could
  be contained in a maximal ideal,
  such as the codimension one subspace $T \mathbb Q[T] \subset \QQ[T]$.
  Conversely, a \emph{finite} dimension vector subspace might still
  not be contained in any maximal ideal,
  such as the one-dimensional space $\QQ \subseteq \QQ[T]$.

  Thus neither finite codimension nor generating all of $\QQ[T]$ as an ideal imply each other.
  It might be interesting to consider other different ways
  of formalizing the notion of ``large kernel''.
\end{remark}

\subsubsection{The geometric side of the semi-Lie AFL conjecture for $n=2$}
To provide evidence for \Cref{conj:semi_lie_spherical}, we prove that:
\begin{theorem}
  [Semi-Lie AFL for the full Hecke algebra for $n=2$]
  \label{thm:semi_lie_n_equals_2}
  Our generalized AFL conjecture, \Cref{conj:semi_lie_spherical}, holds for $n = 2$.
\end{theorem}

The proof of \Cref{thm:semi_lie_n_equals_2} is built up gradually
throughout the paper, culminating in \Cref{ch:finale}.

We comment briefly on the strategy of the proof.
The proof is made possible because the intersection numbers for $n=2$
are easier to work with for a few reasons.
\begin{itemize}
\ii First, one can identify $\VV_2^-$ with an $E/F$-quaternion division algebra,
equipped with a compatible Hermitian form defined via quaternion multiplication.
This makes it possible to describe $\U(\VV_2^-)$ concretely as transformations obtained
via left multiplication by an element of $E$ and right multiplication by a quaternion.
\ii Secondly, it becomes possible to replace the so-called Rapoport-Zink spaces $\RZ_2$
used in the definition of the intersection number with a Lubin-Tate space $\MM_2$.
Thus the problem of computing the intersection number
$\Int\left( (g,u), f \right) \log q$
can be reduced to calculating the intersection of certain special
Kudla-Rapoport divisors on the space $\MM_2$.

However, on the Lubin-Tate space $\MM_2$,
there is a result known as the Gross-Keating formula \cite{ref:GK}
which allows one to make this intersection number fully explicit.
One can then match the resulting equation to the formulas described in
\Cref{cor:semi_lie_derivative_single}
and verify that, under the base change \Cref{ch:satake}
and the matching condition described in \Cref{ch:rs_matching},
the two obtained formulas are identical.
\end{itemize}

\subsection{Roadmap}
The rest of the paper is organized as follows.

\subsubsection{Background information}
The paper begins with some general background information.
\begin{itemize}
  \ii In \Cref{ch:background} we provide some preliminary background
  on the spaces appearing in the overall paper and the Hecke algebras
  $\HH(\U(\VV_n^+))$ and $\HH(S_n(F))$.

  \ii Further background is stated in \Cref{ch:rs_matching},
  where we describe the matching of regular semisimple elements
  so that we may speak of $S_n(F)\rs^-$ and $(S_n(F) \times V'_n(F))\rs^-$.

  \ii In \Cref{ch:satake} we provide reminders on the Satake transform.
\end{itemize}

\subsubsection{Introduction and calculation of the orbital integrals}
In \Cref{ch:orbital0} we introduce the weighted orbital integral
for the group version of the AFL for full spherical Hecke algebra briefly.
This is only for comparison with the definition in later sections
and will not be used again afterwards.

In \Cref{ch:orbitalFJ0} we introduce the weighted orbital integral
for the semi-Lie version of the AFL for full spherical Hecke algebra,
and state the parameters to be used in \Cref{lem:semi_lie_params}.
The main calculation is carried out in \Cref{ch:orbitalFJ1,ch:orbitalFJ2},
which is used to prove \Cref{thm:semi_lie_formula} and its corollaries.

\subsubsection{Large kernel and image}
Having completely computed the orbital integrals in these cases,
we take a side trip in \Cref{ch:ker} to prove the large image results asserted.
We establish \Cref{conj:kernel_semi_lie} for $n = 2$ by proving the asserted
\Cref{thm:semi_lie_ker_trivial} and \Cref{thm:semi_lie_ker_huge}
for the orbital integral on $(S_2(F) \times V'_2(F))\rs^-$.

\subsubsection{Geometric side}
We then turn our attention to the other parts of the two versions of the AFL.
In addition to stating the relevant definitions,
the subsequent chapters aim to prove \Cref{thm:semi_lie_n_equals_2}.
\begin{itemize}
  \ii In \Cref{ch:transf} we briefly define the transfer factors $\omega \in \{\pm1\}$.
  \ii In \Cref{ch:geo}, we describe the Rapoport-Zink spaces $\RZ_n$
  that the geometric side is based on, and define the intersection numbers
  $\Int((1,g), \mathbf{1}_{K^\flat} \otimes f)$ and $\Int((g,u), f)$
  that appear in the two version of the generalized AFL.
  The main ingredients are the Hecke operator introduced in \cite{ref:AFLspherical}
  and the KR-divisor $\ZD(u)$ introduced in \cite{ref:KR}.
  \ii In \Cref{ch:jiao}, we specialize to the situation $n = 2$
  for the intersection numbers in the semi-Lie AFL only.
  The Rapoport-Zink space $\RZ_2$ become replaced with Lubin-Tate space $\MM_2$,
  and we introduce the Gross-Keating formula
  that will be our primary tool for the calculation.
\end{itemize}
Finally, in \Cref{ch:finale} we tie everything together and establish
\Cref{thm:semi_lie_n_equals_2}.

\section{Acknowledgments}
This paper is a subset of the author's PhD thesis \cite{ref:evans_thesis}
which was completed under the supervision of Wei Zhang.
(The thesis \cite{ref:evans_thesis} contains some additional calculations
related to the $n=3$ case of \Cref{conj:inhomog_spherical}
which were not included in this paper, but can be found in the published thesis.
Some Sage source files are also provided in the published thesis.)
I thank my advisor Wei Zhang for suggesting this project,
for his infinite patience and kindness throughout my entire time during graduate school,
and his seemingly instantaneous response times at all hours of the day
to the many, many stupid questions I asked.

I thank Ryan C.~Chen for his interest in this project,
reading a draft of the thesis, and for suggesting \cite{ref:CLZ}
which helped to remove the hypotheses from the main result of this paper.
I also thank Mark Sellke for proofreading a draft of my PhD thesis
and finding several corrections.

I thank the organizers and attendees of the fall 2024
learning seminar on arithmetic inner product formula
for inviting me to speak about my work-in-progress, and their helpful comments on it
leading directly to improvements to this paper.

This work was partially supported by NSF GRFP under grant numbers 1745302 and 2141064.

\section{General background}
\label{ch:background}

\subsection{Notation}
We provide a glossary of notation that will be used in this paper.
As mentioned in the introduction, $p > 2$ is a prime,
$F$ is a finite extension of $\QQ_p$,
and $E/F$ is an unramified quadratic field extension.

\begin{itemize}
  \ii For any $a \in E$, we let $\bar a$ denote the image of $a$
  under the nontrivial automorphism of $\Gal(E/F)$.
  (Hence $a = \bar a$ exactly when $a \in F$.)
  \ii Fix $\eps \in \OO_F^\times$ such that $E = F[\sqrt{\eps}]$.
  \ii Denote by $\varpi$ a uniformizer of $\OO_F$, such that $\bar \varpi = \varpi$.
  \ii Let $q \coloneqq |\OO_F/\varpi|$ be the residue characteristic.
  (Hence $|\OO_E / \varpi| = q^2$.)
  \ii Let $v$ be the associated valuation for $\varpi$.
  \ii Let $\eta$ be the quadratic character attached to $E/F$ by class field theory,
  so that $\eta(x) = -1^{v(x)}$.
  \ii $\VV_n^+$ denotes a split $E/F$-Hermitian space of dimension $n$ (unique up to isomorphism).
  \ii Let $\beta$ denote the $n \times n$ antidiagonal matrix
  \[ \beta \coloneqq \begin{pmatrix} && 1 \\ & \iddots \\ 1 \end{pmatrix} \]
  and pick the basis of $\VV_n^+$ such that the Hermitian form on $\VV_n^+$ is given by
  \[ \VV_n^+ \times \VV_n^+ \to E \qquad (x,y) \mapsto x^\ast \beta y. \]
  \ii Set
  \[ \U(\VV_n^+) = \{ g \in \GL_n(\OO_E) \mid g^\ast \beta g = \beta\} \]
  the unitary group over $\VV_n^+$.
  Note that $\beta$ is \emph{antidiagonal}, in contrast to the convention $\beta = \id_n$
  that is often used for unitary matrices with entries in $\CC$.
  The natural hyperspecial maximal compact subgroup
  is \[ \U(\VV_n^+) \cap \GL_n(\OO_E). \]
  In some parts of the paper we abbreviate $G = \U(\VV_n^+)$ and $K = \U(\VV_n^+) \cap \GL_n(\OO_E)$
  following the convention in \cite{ref:AFLspherical}.
  \ii Let $K' \coloneqq \GL_n(\OO_E)$ denote the hyperspecial maximal compact subgroup of $G' \coloneqq \GL_n(E)$.
  \ii Let $\VV_n^-$ denote the non-split $E/F$-Hermitian space of dimension $n$
  (unique up to isomorphism), and $\U(\VV_n^-)$ the corresponding unitary group.
  This space will be realized in \Cref{ch:geo}.
\end{itemize}

\subsection{Intersection of disks in an ultrametric space}
The following two lemmas will be useful for both versions of the orbital integral.
It is a slight rephrasing of \cite[Lemma 4.4]{ref:AFL}.

\begin{lemma}
  [One-disk volume lemma]
  \label{lem:volume}
  Let $\xi \in \OO_E^\times$, $\rho \in \ZZ$, and $n \ge \max(\rho, 1)$ an integer.
  Then
  \begin{align*}
    &\Vol\left( \left\{ x \in E \mid v(1-x \bar x) = n,
      \; v(x-\xi) \ge \rho \right\} \right) \\
    &=
    \begin{cases}
      0 & \text{if } v(1-\xi\bar\xi) < \rho \\
      q^{-n}(1-q^{-2}) & \text{if } \rho \le 0 \\
      q^{-(n+\rho)}(1-q\inv) & \text{if } v(1-\xi\bar\xi) \ge \rho \ge 1.
    \end{cases}
  \end{align*}
\end{lemma}
\begin{proof}
  When $\rho \le 0$, the condition $v(x - \xi) \ge \rho$ is vacuously true,
  so we just are computing
  \[ \Vol\left( \left\{ x \in E \mid v(1-x \bar x) = n \right\} \right). \]
  However, in general for any Schwartz function $\psi$ on $E$ we have an identity
  \begin{equation}
    \int_E \psi(x) \odif x
    = \frac{1}{1-q\inv} \int_{y \in F} \int_{t \in \OO_E^\times}
    \psi\left( x_t \cdot \frac{y}{\bar y} \right) \odif y \odif t
    \label{eq:schwartz}
  \end{equation}
  where $x_t$ is any choice of element $x_t \in E$ such that $t = x_t \bar{x_t}$
  (see the proof of \cite[Lemma 4.4]{ref:AFL}).
  Note the measures here are additive despite $t \in \OO_E^\times$.
  So if one selects
  \[ \psi(x) = \mathbf{1}_{v(1 - \Norm(x)) = n} \]
  then we get
  \begin{align*}
    \Vol\left( \left\{ x \in E \mid v(1-x \bar x) = n \right\} \right)
    &= \frac{1}{1-q\inv} \int_{t \in F} \int_{y \in \OO_E^\times}
    \mathbf{1}_{v\left( 1 - \Norm\left(x_t \cdot \frac{y}{\bar y} \right) \right) = n} \odif y \odif t \\
    &= \frac{1}{1-q\inv} \int_{t \in F} \int_{y \in \OO_E^\times}
    \mathbf{1}_{v\left( 1 - t \right) = n} \odif y \odif t \\
    &= \frac{\Vol_E(\OO_E^\times)}{1-q\inv} \cdot \Vol_F(1 + \varpi^{-n} \OO_F^\times) \\
    &= \frac{1 - \frac{1}{q^2}}{1-q\inv}
      \cdot \left( q^{-n} \cdot \left( 1 - \frac 1q \right) \right) \\
    &= q^{-n} (1-q^{-2}).
  \end{align*}
  The case $\rho > 0$ is proved in \cite[Lemma 4.4]{ref:AFL}
  using the same method of using \eqref{eq:schwartz}.
\end{proof}

We also comment on the well-known fact that in an ultrametric space,
any two disks are either disjoint or one is contained in the other.
See \Cref{fig:no_mastercard}.
\begin{lemma}
  [No ultrametric MasterCard logo]
  Choose $\xi_1, \xi_2 \in E$ and $\rho_1 \geq \rho_2$.
  Consider the two disks:
  \begin{align*}
    D_1 &= \left\{ x \in E \mid v(x-\xi_1) \ge \rho_1 \right\} \\
    D_2 &= \left\{ x \in E \mid v(x-\xi_2) \ge \rho_2 \right\}.
  \end{align*}
  Then, if $v(\xi_1-\xi_2) \geq \rho_2$, we have $D_1 \subseteq D_2$.
  If not, instead $D_1 \cap D_2 = \varnothing$.
  \label{lem:no_mastercard}
\end{lemma}
\begin{proof}
  Because $E$ is an ultrametric space and $\Vol(D_1) \leq \Vol(D_2)$,
  we either have $D_1 \subseteq D_2$ or $D_1 \cap D_2 = \varnothing$.
  The latter condition checks which case we are in by testing if $\xi_1 \in D_2$,
  since $\xi_1 \in D_1$.
\end{proof}

\begin{figure}
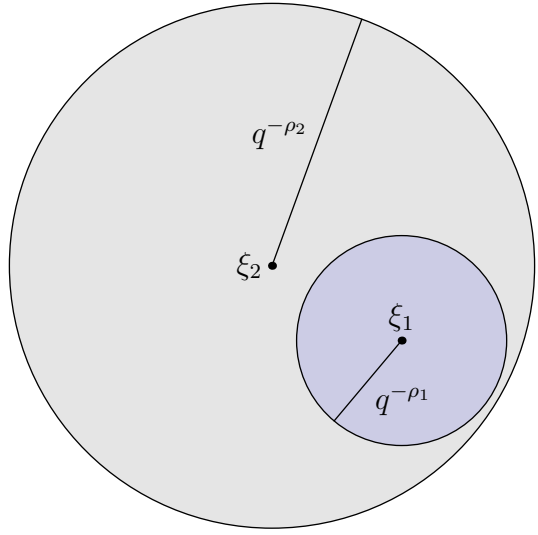

\centering
\begin{asy}
  defaultpen(fontsize(12pt));
  size(7cm);
  pair O1 = 0.57*dir(-30);
  pair O2 = (0,0);
  real r1 = 0.4;
  real r2 = 1;
  filldraw(circle(O2, r2), rgb(0.9, 0.9, 0.9));
  filldraw(circle(O1, r1), rgb(0.8, 0.8, 0.9));
  pair P1 = O1+r1*dir(230);
  pair P2 = O2+r2*dir( 70);
  draw(O1--P1);
  draw(O2--P2);
  dot("$\xi_1$", O1, dir(90));
  dot("$\xi_2$", O2, dir(180));
  label("$q^{-\rho_1}$", midpoint(O1--P1), dir(P1-O1)*dir(90));
  label("$q^{-\rho_2}$", midpoint(O2--P2), dir(P2-O2)*dir(90));
\end{asy}
\caption{Figure corresponding to \Cref{lem:no_mastercard}.}
\label{fig:no_mastercard}
\end{figure}

We package both of these results together in this lemma that will be used repeatedly.
\begin{lemma}
  [Two-disk volume lemma]
  \label{lem:quadruple_ineq}
  Let $\xi_1, \xi_2 \in \OO_E^\times$ and let $\rho_1 \ge \rho_2$ be integers.
  Also let $n \ge \max(\rho_1, 1)$ be an integer.
  Then the set of points $x \in E$ satisfying all of the equations
  \begin{align*}
    v(x - \xi_1) &\ge \rho_1 \\
    v(x - \xi_2) &\ge \rho_2 \\
    v(1 - x \ol x) &= n
  \end{align*}
  has positive volume if and only if
  \[ v(1 - \xi_1 \bar{\xi_1}) \ge \rho_1, \qquad \rho_2 \le v(\xi_1 - \xi_2). \]
  In that case, the volume is equal to
  \[
    \begin{cases}
      q^{-(n+\rho_1)}(1-q\inv) & \text{if } \rho_1 \ge 1 \\
      q^{-n}(1-q^{-2}) & \text{if } \rho_1 \le 0.
    \end{cases}
  \]
\end{lemma}
In the situation where $\xi_i \notin \OO_E^\times$,
the quantity $v(x-\xi_i) = \min(0, v(\xi_i))$
becomes independent of the value of $x$,
and so \Cref{lem:quadruple_ineq} becomes unnecessary
(\Cref{lem:volume} will suffice).
We will deal with this situation when it arises;
it turns out this will only occur when $v(b) \neq 0$.

\subsection{The spaces $S_n(F)$ and $S_n(F) \times V'_n(F)$}
For the analytic side of the two AFL conjectures we investigate,
the following two spaces will be used as inputs to our weighted orbital integrals.
\begin{definition}
  [$S_n(F)$; {\cite[(4.10)]{ref:highdim2024}}]
  We define the symmetric space
  \[ S_n(F) \coloneqq \left\{ g \in \GL_n(E) \mid g \bar g = \id_n \right\}. \]
  It has a natural left action of $\GL_n(E)$ by
  \begin{align*}
    \GL_n(E) \times S_n(F) &\to S_n(F) \\
    g \cdot \gamma &\coloneqq g \gamma \bar g^{-1}.
  \end{align*}
\end{definition}

\begin{definition}
  [$V'_n(F)$; {\cite[(4.11)]{ref:highdim2024}}]
  We set
  \[ V'_n(F) \coloneqq F^n \times (F^n)^\vee \]
  where $-^\vee$ denotes the $F$-dual space, i.e., $(F^n)^\vee = \Hom_F(F^n, F)$.
  Then we may also consider the augmented space
  \[ S_n(F) \times V'_n(F). \]
  If we identify $F^n$ with column vectors of length $n-1$ and $(F^n)^\vee$
  with row vectors of length $n$ then we have a left action of $\GL_n(F)$ by
  \begin{align*}
    \GL_n(F) \times (S_n(F) \times V'_n)(F)
    &\to S_n(F) \times V'_n(F) \\
    h \cdot (\gamma, \uu, \vv^\top)
    &\coloneqq (h \gamma h^{-1}, h\uu, \vv^\top h^{-1}).
  \end{align*}
  Note that according to the embedding
  \begin{align*}
    S_n(F) \times V'_n(F)
    &\hookrightarrow \Mat_{n+1}(E) \\
    (\gamma, \uu, \vv^\top)
    &\mapsto \begin{pmatrix} \gamma & \uu \\ \vv^\top & 0 \end{pmatrix}
  \end{align*}
  we can consider elements of $S_n(F) \times V'_n(F)$ as elements of $\Mat_{n+1}(E)$ too.
  In that case the action of $h \in \GL_{n+1}(F)$
  coincides with $h \cdot g \mapsto hg\bar{h}^{-1}$ as well.
\end{definition}

\begin{definition}
  [$K'_S$]
  For brevity, let
   \[ K'_S \coloneqq S_n(F) \cap \GL_n(\OO_F). \]
\end{definition}

\subsection{Definition of Hecke algebra}
We reminder the reader the definition of a Hecke algebra.
For this subsection, $G$ will denote \emph{any}
unimodular locally compact topological group,
and $K$ any closed subgroup of $G$.

\begin{definition}
  [$\HH(G,K)$]
  The \emph{Hecke algebra}
  \[ \HH(G, K) \coloneqq \QQ[K \backslash G \slash K] \]
  is defined as the space of compactly supported $K$-bi-invariant
  locally constant functions on $G$.
  (The adjective \emph{spherical} Hecke algebra refers to the special case
  where $K$ is a maximal compact subgroup of $G$,
  which is the main case of interest for us.)

  Given two such functions $f_1$ and $f_2$ in $\HH(G,K)$,
  one can consider define the convolution
  \[ (f_1 \ast f_2)(g) \coloneqq \int_G f_1(g\inv x) f_2(x) \; \odif x \]
  which makes $\HH(G, K)$ into a $\QQ$-algebra,
  whose identity element is $\mathbf{1}_K$.
\end{definition}

In other sources, this may be denoted $\HH_K(G)$ or even just $\HH_K$.
(So what is written in $\HH_{K'^\flat \otimes K'}$ in other places
will be written as $\HH(G'^\flat \otimes G', K'^\flat \otimes K')$ here).

In the case where $G$ is a reductive Lie group and
$K$ is the maximal compact subgroup
(or more generally whenever $(G,K)$ is a Gelfand pair),
this Hecke algebra is actually commutative.

\subsection{The specific Hecke algebras $\HH(\GL_n(E))$ and $\HH(\U(\VV_n^+))$ and the module $\HH(S_n(F))$}
For our purposes, we define shorthands for two specific Hecke algebras
that will come up consistently:
\begin{align*}
  \HH(\GL_n(E)) &\coloneqq \HH(\GL_n(E), \GL_n(\OO_E)) \\
  \HH(\U(\VV_n^+)) &\coloneqq \HH(\U(\VV_n^+), \U(\VV_n^+) \cap \GL_n(\OO_E)).
\end{align*}
Note that $\GL_n(\OO_E)$ and $\U(\VV_n^+) \cap \GL_n(\OO_E)$
are the natural hyperspecial maximal compact subgroups of $\GL_n(E)$ and $\U(\VV_n^+$),
respectively.

Now the symmetric space $S_n(F)$ is not a group,
so it does not make sense to define the same thing here.
Nevertheless, we introduce
\[ \HH(S_n(F)) \coloneqq \mathcal{C}_{\mathrm{c}}^{\infty}(S_n(F))^{K'} \]
as the set of smooth compactly supported functions on $S_n(F)$
which are invariant under the action of $K' \subseteq G'$;
this is an $\HH(\GL_n(E))$-module,
where the action of $f \in \HH(\GL_n(E))$ on $\phi \in \HH(S_n(F))$ is given by
\[ (f \cdot \phi)(\gamma) \coloneqq \int_G f(g) \phi(g \cdot \gamma) \odif g \]
for $\gamma \in S_n(F)$.
This does \textbf{not} have a multiplication structure at the moment, \emph{a priori}.
However, we will later (in \Cref{ch:satake}) give an isomorphism from $\HH(S_n(F))$
to $\HH(\U(\VV_n^+))$ as $\QQ$-vector spaces;
since the latter is a $\QQ$-algebra,
this induces a multiplication structure on $\HH(S_n(F))$
and consequently we may speak of $\HH(S_n(F))$ as a ring under this isomorphism.

Throughout this paper, to be consistent with the notation, we denote
\begin{itemize}
  \ii elements of $\HH(\U(\VV_n^+))$ using $f$ or $f_i$ or similar
    (i.e.\ lowercase Roman letters);
  \ii elements of $\HH(\GL_n(E))$ by $f'$ or $f'_i$ or similar
    (i.e.\ lowercase Roman letters with apostrophes);
  \ii elements of $\HH(S_n(F))$ by $\phi$ or $\phi_i$
    (i.e.\ lowercase Greek letters).
\end{itemize}

\section{Regular semi-simplicity and matching}
\label{ch:rs_matching}

\subsection{Regular semi-simple elements}
We first recall the notion of regularity
that first appeared in \cite[\S6]{ref:multoneconj}.

\begin{definition}
  [Regular semisimple in $\Mat_n(E)$]
  \label{def:rs}
  Consider a $n \times n$ matrix
  \[ \begin{pmatrix} A & \uu \\ \vv^\top & d \end{pmatrix} \in \Mat_n(E) \]
  where $A$ is an $(n-1) \times (n-1)$ matrix.
  Then we say this matrix is \emph{regular semi-simple} if
  \[ \left< \uu, A\uu, \dots, A^{n-2}\uu \right> \]
  and \[ \left< \vv^\top, \vv^\top A, \dots, \vv^\top A^{n-2} \right> \]
  are each a basis of $E^{n-1}$.
  Equivalently, the matrix
  \[ \left[ \vv^\top A^{i+j-2} \uu \right]_{i,j=1}^{n-1} \]
  should be nonsingular.
\end{definition}

\begin{remark}
  [Equivalent definition of regular semisimple]
  In \cite[Theorem 6.1]{ref:multoneconj}, this definition is shown to be equivalent to
  requiring that, under the action of conjugation by $\GL_{n-1}(E)$:
  \begin{itemize}
  \ii the matrix has trivial stabilizer; and
  \ii the $\GL_{n-1}(\ol E)$-orbit is a Zariski-closed subset of $\GL_n(\ol E)$.
  \end{itemize}
  Here $\ol E$ is as usual an algebraic closure of $E$.
\end{remark}

\begin{remark}
  [Invariants under $\GL_{n-1}(E)$ conjugation; {\cite[Proposition 6.2]{ref:multoneconj}}]
  It turns out we can detect whether two regular semisimple elements
  \[
    \begin{pmatrix} A_1 & \uu_1 \\ \vv_1^\top & d_1 \end{pmatrix},
    \begin{pmatrix} A_2 & \uu_2 \\ \vv_2^\top & d_2 \end{pmatrix}
    \in \Mat_{n}(E)
  \]
  are conjugate by an element of $\GL_{n-1}(E)$.
  This happens if and only if the following conditions all hold:
  \begin{itemize}
    \ii The matrices $A_1$ and $A_2$ have the same characteristic polynomial;
    \ii We have \[ \vv_1^\top A_1^i \uu_1 = \vv_2^\top A_2^i \uu_2 \]
    for every $i = 0, 1, \dots, n-2$; and
    \ii We have $d_1 = d_2$.
  \end{itemize}
  Thus, this gives a set of invariants that completely classify the orbits
  under the action of $\GL_{n-1}(E)$.

  Put another way, the invariants of
  \[ \begin{pmatrix} A & \uu \\ \vv^\top & d \end{pmatrix} \in \Mat_n(E) \]
  are the (monic) characteristic polynomial of $A$
  (which has $n-1$ coefficients besides the leading coefficient),
  the values of $\vv^\top A^i \uu$ for $i = 0, \dots, n-2$
  and the number $d$, for a total of $2n-1$ numbers.
  \label{rem:invariants}
\end{remark}

We can now speak of regular-simplicity in each of the four
particular cases relevant to this paper.
\begin{definition}
  [Regular semisimple]
  In the group version of the AFL:
  \begin{itemize}
    \ii We say $\gamma \in S_n(F)$ is regular semisimple
    if its image under the inclusion $S_n(F) \subseteq \Mat_n(E)$ is regular semisimple.
    We write $\gamma \in S_n(F)\rs$.

    \ii For $g \in \U(\VV_n^\pm)$,
    we say $g$ is regular semisimple
    if its image under the inclusion $\U(\VV_n^\pm) \subseteq \Mat_n(E)$ is regular semisimple.
    We write $g \in \U(\VV_n^\pm)\rs$.
  \end{itemize}
  In the semi-Lie version of the AFL:
  \begin{itemize}
    \ii We say $(\gamma, \uu, \vv^\top) \in S_n(F) \times V'_n(F)$
    is regular semisimple if its image under the embedding
    \begin{equation}
      \begin{aligned}
        S_n(F) \times V'_n(F) &\hookrightarrow \Mat_{n+1}(E) \\
        (\gamma, \uu, \vv^\top) &\mapsto \begin{pmatrix} \gamma & \uu \\ \vv^\top & 0 \end{pmatrix}
      \end{aligned}
      \label{eq:embed_FJ_analytic}
    \end{equation}
    is regular semisimple.
    In other words, we require that
    both of the sets
    $\left( \uu, \gamma \uu \dots, \gamma^{n-1}\uu \right)$
    and
    $\left( \vv^\top, \vv^\top \gamma, \dots, \vv^\top \gamma^{n-1} \right)$
    are bases of $E^n$.
    In this case we write $(\gamma, \uu, \vv^\top) \in (S_n(F) \times V'_n(F))\rs$.

    \ii For $(g,u) \in \U(\VV_n^\pm) \times \VV_n^\pm$ we say $(g, u)$
    is regular semisimple if its image under the embedding
    \begin{equation}
      \begin{aligned}
        \U(\VV_n^\pm) \times \VV_n^\pm &\hookrightarrow \Mat_{n+1}(E) \\
        (g, u) &\mapsto \begin{pmatrix} g & u \\ u^\ast & 0 \end{pmatrix}
      \end{aligned}
      \label{eq:embed_FJ_geometric}
    \end{equation}
    is regular semisimple.
    Here $u^\ast$ is the conjugate transpose.

    This is equivalent to the set $\left(  u, gu, \dots, g^{n-1}u \right)$
    being linearly independent (i.e.\ form a basis of $\VV_n^\pm$);
    in this case the independence of $\left( u^\ast, u^\ast g, \dots, u^\ast g^{n-1} \right)$
    is redundant, so it's enough to verify one condition.
    We write $(g,u) \in (\U(\VV_n^\pm) \times \VV_n^\pm)\rs$.
  \end{itemize}
  \label{def:regular}
\end{definition}

\subsection{Matching in the group version of the inhomogeneous AFL}
We now describe the matching condition used in the group version of AFL.
\begin{definition}
  [Matching $S_n(F)\rs \longleftrightarrow \U(\VV_n^\pm)\rs$;
  {\cite[p.\ 202]{ref:AFL}}]
  We say $\gamma \in S_n(F)\rs$ matches the element $g \in \U(\VV_n^\pm)\rs$
  if $g$ is conjugate to $\gamma$ by an element of $\GL_{n-1}(E)$.
  By \Cref{rem:invariants}, this is an assertion that
  the invariants for $\gamma$ and $g$ coincide.
  \label{def:matching_inhomog}
\end{definition}
In that case, we have the following result.
\begin{proposition}
  [{\cite[Lemma 2.3]{ref:AFL}}; see also {\cite[(3.3.2)]{ref:AFLspherical}}]
  \label{prop:valuation_delta_matching_group}
  This definition of matching gives
  a bijection of regular semisimple orbits
  \[ [S_n(F)]\rs \xrightarrow{\sim} [\U(\VV_n^+)]\rs \amalg [\U(\VV_n^-)]\rs. \]
  Moreover, we can detect whether $\gamma \in S_n(F)\rs$ matches an orbit of
  $\U(\VV_n^+)\rs$ or $\U(\VV_n^-)\rs$ as follows.
  Suppose we write $\gamma$ in the format of \Cref{def:rs} and consider
  \[ \Delta \coloneqq \det \left[ \vv^\top A^{i+j-2} \uu \right]_{i,j=1}^{n-1} \neq 0. \]
  Then
  \begin{itemize}
    \ii $\gamma$ matches an orbit in $\U(\VV_n^+)\rs$ if $v(\Delta)$ is even;
    \ii $\gamma$ matches an orbit in $\U(\VV_n^-)\rs$ if $v(\Delta)$ is odd.
  \end{itemize}
\end{proposition}
In this paper, \Cref{conj:inhomog_spherical}
requires that $\gamma$ should match an element of $\U(\VV_n^-)\rs$
and consequently we will usually only be interested in the latter case.
We write the following abbreviation:
\begin{definition}
  [$S_n(F)\rs^\pm$]
  We let \[ S_n(F)\rs^- \subset S_n(F)\rs \]
  denote the subset of elements in $S_n(F)\rs$ that match
  with an element in $\U(\VV_n^-)\rs$.
  Define $S_n(F)\rs^+$ similarly.
  Hence $S_n(F)\rs = S_n(F)\rs^- \amalg S_n(F)\rs^+$.
\end{definition}

\subsection{Matching in the semi-Lie version of the AFL}
For the semi-Lie version matching is defined analogously:
\begin{definition}
  [Matching $(S_n(F) \times V'_n(F))\rs \longleftrightarrow (\U(\VV_n^\pm) \times \VV_n^\pm)\rs$;
  {\cite[\S1.3]{ref:liuFJ}}]
  We say $(\gamma, \uu, \vv^\top) \in (S_n(F) \times V'_n(F))\rs$
  matches the element $(g, u) \in (\U(\VV_n^\pm) \times \VV_n^\pm)\rs$ if
  their images under the embeddings \eqref{eq:embed_FJ_analytic}
  and \eqref{eq:embed_FJ_geometric} are conjugate by an element of $\GL_n(E)$.

  Unwrapping this with \Cref{rem:invariants},
  an equivalent definition is to require both of the following conditions:
  \begin{itemize}
    \ii As elements of $\GL_n(E)$,
    both $g$ and $\gamma$ have the same characteristic polynomial.
    \ii We have $\vv^\top \gamma^i \uu = \left< g^i u, u \right>$ for all $0 \le i \le n-1$,
    where $\left< -,- \right>$ is the Hermitian form on $\VV_n^\pm$.
  \end{itemize}
  \label{def:matching_semi_lie}
\end{definition}
We have the following analogous criteria for matching.
\begin{proposition}
  [\cite{ref:liuFJ}]
  \label{prop:valuation_delta_matching_semilie}
  This definition of matching gives a bijection of regular semisimple orbits
  \[ [S_n(F) \times V'_n(F)]\rs \xrightarrow{\sim} [\U(\VV_n^+) \times \VV_n^+]\rs \amalg [\U(\VV_n^-) \times \VV_n^-]\rs. \]
  Moreover, we can detect whether $\guv \in S_n(F)\rs$ matches an orbit of
  $(\U(\VV_n^+) \times \VV_n^+)\rs$ or $(\U(\VV_n^-) \times \VV_n^-)\rs$ as follows:
  consider the determinant
  \[ \Delta \coloneqq \det \left[ \vv^\top \gamma^{i+j-2} \uu \right]_{i,j=1}^n \neq 0. \]
  Then
  \begin{itemize}
    \ii $\gamma$ matches an orbit in $(\U(\VV_n^+) \times \VV_n^+)\rs$ if $v(\Delta)$ is even;
    \ii $\gamma$ matches an orbit in $(\U(\VV_n^-) \times \VV_n^-)\rs$ if $v(\Delta)$ is odd.
  \end{itemize}
\end{proposition}
In this paper, \Cref{conj:semi_lie_spherical}
requires that $\guv$ should match an element of $(\U(\VV_n^-) \times \VV_n^-)\rs$
and consequently we will usually only be interested in the latter case.
Accordingly we write the following shorthand:
\begin{definition}
  [$(S_n(F) \times V'_n(F))\rs^\pm$]
  We let \[ (S_n(F) \times V'_n(F))\rs^- \subset (S_n(F) \times V'_n(F))\rs  \]
  denote the subset of those elements in $(S_n(F) \times V'_n(F))\rs$ that match
  with an element of $(\U(\VV_n^-) \times \VV_n^-)\rs$.
  Define $(S_n(F) \times V'_n(F))\rs^+$ analogously.
  Hence $(S_n(F) \times V'_n(F))\rs = (S_n(F) \times V'_n(F))\rs^- \amalg (S_n(F) \times V'_n(F))\rs^+$.
\end{definition}

\section{Base change}
\label{ch:satake}

This section introduces necessary background material on the base change
\[ \BC_{S_n}^{\eta^{n-1}} \colon \HH(S_n(F)) \to \HH(\U(\VV_n^+)). \]

Throughout this section we let $\Sym(n)$ denote the symmetric group in $n$ variables
with order $n!$
(since $S_n(F) \subseteq \GL_n(E)$ is already reserved for the symmetric space).

\subsection{Background on the Satake transformation in transformation}
We recall a general form of the Satake transformation, which will be used later.

For this subsection, $G$ will denote an arbitrary connected reductive group
over some non-Archimedean local field $F$.
We will not distinguish between $G$ and $G(F)$ when there is no confusion.

To simplify things, we will assume $G$ is unramified;
but we do \emph{not} assume $G$ is split.
Introduce the following notation:
\begin{itemize}
  \ii Let $K$ be a hyperspecial maximal compact subgroup of $G$
  (it exists because $G$ is unramified).
  \ii Let $A$ denote a maximal $F$-split torus in $G$.
  All the maximal $F$-split tori in $G$ are conjugate; let $A$ denote one of them.
  \ii Let $M$ be the centralizer of $A$; this is itself a maximal torus in $G$.
  \ii Let $\prescript{\circ}{} M \coloneqq M(F) \cap K$
  be the maximal compact subgroup of $M$.
  \ii Let $P$ denote a minimal $F$-parabolic containing $A$.
  \ii Let $\delta$ denotes the modulus character of $P$.
  It can be describes as follows.
  Let $\varpi$ denote a uniformizer for $F$ and $q$ the residue characteristic.
  Then if $\rho$ is the Weyl vector and $\mu$ is a positive cocharacter, then
  \[ \delta(\mu(\varpi)) = q^{- \left< \mu, \rho\right>}. \]
  \ii Let $N$ denote the unipotent radical of $P$.
  \ii Let $W$ be the relative Weyl group for the pair $(G,A)$,
  which acts on $\HH(M, \prescript{\circ}{} M)$.
\end{itemize}
We can now state the Satake isomorphism.
\begin{definition}
  [Satake transform]
  The \emph{Satake transform} is a canonical isomorphism of Hecke algebras
  \[ \Sat \colon \HH(G, K) \to \HH(M, \prescript{\circ}{} M)^W \]
  which is given by defining
  \[ (\Sat(f))(t) \coloneqq \delta(t)^\half \int_N f(nt) \odif n  \]
  for each $t \in M$.
\end{definition}
We are going to apply this momentarily in two situations:
once when $G$ is the general linear group (which is split),
and once when $G$ is a unitary group.

\subsection{The Satake transformation for the particular Hecke algebras $\HH(\GL_n(E))$ and $\HH(\U(\VV_n^+))$}
To take the Satake transform of $\HH(\U(\VV_n^+))$, we define the following abbreviations.
\begin{itemize}
  \ii Let $T$ denote the split diagonal torus of $\GL_n$.
  \ii Let
  \[ N' \coloneqq \left\{ \begin{pmatrix}
      1 & \ast & \dots & \ast \\
        & 1 & \dots & \ast \\
        &   & \ddots & \vdots \\
        &   &   & 1 \end{pmatrix}\right\} \subseteq \GL_n(E) \]
  denote the unipotent upper-triangular matrices.
\end{itemize}
Similarly for $\HH(\U(\VV_n^+))$:
\begin{itemize}
  \ii Set $m \coloneqq \left\lfloor n/2 \right\rfloor$ for brevity.
  \ii Let
  \[ A \coloneqq \left\{
    \diag(x_1, \dots, x_m, 1_{n-2m}, x_m\inv, \dots, x_1\inv) \right\} \]
  so that $A(F)$ is a maximal $F$-split torus of $\U(\VV_n^+)$.
  \ii Let $N \coloneqq N' \cap G$ denote the unipotent upper triangular matrices
  which are also unitary.
  \ii For brevity, let $W_m \coloneqq (\ZZ/2\ZZ)^m \rtimes \Sym(m)$
  be the relative Weyl group of $(G,A)$.
\end{itemize}

We can now introduce the Satake transform for our two
\emph{bona fide} Hecke algebras, using the data in Table~\ref{tab:satakestuff}.

\begin{table}[ht]
  \centering
  \begin{tabular}{lll}
    \toprule
    Group & $G' = \GL_n(E)$ & $G = \U(\VV_n^+)$ \\ \midrule
    Local field & $E$ & $F$ \\\hline
    Hyperspecial compact & $K' = \GL_n(\OO_E)$ & $K = G \cap \GL_n(\OO_E)$ \\\hline
    Max'l split torus & $T(E)$ & $A(F)$ \\\hline
    Centralizer of split torus & $T(\OO_E)$ & $A(\OO_F)$ \\\hline
    Parabolic (Borel) & Upper tri in $G'$ & Upper tri in $G$ \\\hline
    Unipotent rad.\ of parabolic & $N'$ (unipot.\ upper tri) & $N$ (unipot.\ upper tri) \\\hline
    Relative Weyl group & $\Sym(n)$ & $W_m = (\ZZ/2\ZZ)^m \rtimes \Sym(m)$ \\
    \bottomrule
  \end{tabular}
  \caption{Data needed to run the Satake transformation.}
  \label{tab:satakestuff}
\end{table}

Hence, the Satake transformations obtained can be viewed as
\begin{align*}
  \Sat &\colon \HH(\GL_n(E)) \xrightarrow{\sim} \QQ[T(E) / T(\OO_E)]^{\Sym(n)} \\
  \Sat &\colon \HH(\U(\VV_n^+))\xrightarrow{\sim} \QQ[A(F) / A(\OO_F)]^{W_m}
\end{align*}
(In both cases, the modular character $\delta^{1/2}$ gives rational values,
so it is okay to work over $\QQ$.)

To make this further concrete, we remark that the cocharacter groups
involved are free abelian groups with known bases.
This identification lets us rewrite the right-hand sides above as concrete polynomials.
Specifically, we identify
\[ \QQ[T(E) / T(\OO_E)]^{\Sym(n)}
  \xrightarrow{\sim} \QQ[X_1^\pm, \dots, X_n^\pm]^{\Sym(n)} \]
by identifying $X_i$ with the
cocharacter corresponding to injection into the $i$\ts{th} factor.
Similarly, we identify
\[ \QQ[A(F) / A(\OO_F)]^{W_m}
  \xrightarrow{\sim} \QQ[Y_1^{\pm}, \dots, Y_m^{\pm}]^{W_m} \]
by identifying $Y_i + Y_i^{-1}$
with the cocharacter corresponding to
\[ x \mapsto \diag(1, \dots, x, \dots, x\inv, \dots, 1) \]
where $x$ is in the $i$\ts{th} position and $x\inv$ is in the $(n-i)$\ts{th} position,
and all other positions are $1$.
Here $\QQ[Y_1^{\pm}, \dots, Y_m^{\pm}]^{W_m}$
denotes the ring of symmetric polynomials in $Y_i + Y_i^{-1}$.

So, henceforth, we will consider
\begin{align*}
  \Sat &\colon \HH(\GL_n(E)) \xrightarrow{\sim} \QQ[X_1^\pm, \dots, X_n^\pm]^{\Sym(n)} \\
  \Sat &\colon \HH(\U(\VV_n^+)) \xrightarrow{\sim} \QQ[Y_1^{\pm}, \dots, Y_m^{\pm}]^{W_m}.
\end{align*}

\subsection{Relation of Satake transformation to base change}
Let
\[ \BC \colon \HH(\GL_n(E)) \to \HH(\U(\VV_n^+)) \]
denote the stable base change morphism from $\GL_n(E)$ to the unitary group $\U$.
The relevance of the Satake transformation is that
(see e.g.\ \cite[Proposition 3.4]{ref:leslie})
it gives a way to make this $\BC$ completely explicit:
we have a commutative diagram
\begin{center}
\begin{tikzcd}
  \HH(\GL_n(E))  \ar[r, "\sim"', "\Sat"] \ar[d, "\BC"]
    & \QQ[X_1^\pm, \dots, X_n^\pm]^{\Sym(n)} \ar[d, "\BC"] \\
  \HH(\U(\VV_n^+)) \ar[r, "\sim"', "\Sat"]
    & \QQ[Y_1^\pm, \dots, Y_m^\pm]^{W_m}
\end{tikzcd}
\end{center}
Here the right arrow is also denoted $\BC$ following \cite{ref:AFLspherical}
(although it is denoted $\nu$ in \cite{ref:leslie}).
This gives a way in which we can concretely calculate the map $\BC$
in some situations.

\subsection{The map $\BC^{\eta^{n-1}}_{S_n}$}
Before we can define the map $\BC^{\eta^{n-1}}_{S_n}$
we need one more piece of notation.
Consider the following map.
\begin{definition}
  [$\rproj$]
  Denote by $\rproj \colon \GL_n(E) \surjto S_n(F)$ the projection defined by
  \[ \rproj(g) \coloneqq g \bar{g}\inv. \]
\end{definition}
Then $\rproj$ induces a map
\begin{align*}
  \rproj_\ast \colon \HH(\GL_n(E)) &\to \HH(S_n(F)) \\
  \rproj_\ast(f')\left( g\bar{g}\inv \right) &= \int_{\GL_n(F)} f'(gh) \odif h
\end{align*}
by integration on the fibers.
A similar twisted version by $\eta$
\begin{align*}
  \rproj_\ast^\eta \colon \HH(\GL_n(E)) &\to \HH(S_n(F)) \\
  \rproj_\ast^\eta(f')\left( g\bar{g}\inv \right) &= \int_{\GL_n(F)} f'(gh) \eta(gh) \odif h
\end{align*}
is defined analogously,
where as before $\eta(g) = (-1)^{v(\det g)}$ in a slight abuse of notation.

Then Leslie \cite{ref:leslie} shows the following result.
\begin{theorem}
  [{\cite[Theorem 3.2 and Proposition 3.4]{ref:leslie}}]
  Both maps $\rproj_\ast$ and $\rproj_\ast^\eta$ induce isomorphisms
  \begin{align*}
    \BC_{S_n} \colon \HH(S_n(F)) &\xrightarrow{\sim} \HH(\GL_n(E)) \\
    \BC^{\eta^{n-1}}_{S_n} \colon \HH(S_n(F)) &\xrightarrow{\sim} \HH(\GL_n(E))
  \end{align*}
  such that
  \begin{align*}
    \BC &= \BC_{S_n} \circ \rproj_\ast \\
    \BC &= \BC^{\eta^{n-1}}_{S_n} \circ \rproj_\ast^{\eta^{n-1}}.
  \end{align*}
\end{theorem}
We take these isomorphisms promised by this theorem
as the definition of $\BC_{S_n}$ and $\BC^{\eta^{n-1}}_{S_n}$ in our conjectures
(noting when $n$ is odd they coincide, as $\eta^{n-1} = 1$).

When combined with the Satake information we have, we get the following diagram.
\begin{center}
\begin{tikzcd}
  \HH(\GL_n(E)) \ar[dd, "\rproj_\ast^{\eta^{n-1}}"', bend right = 60] \ar[r, "\sim"', "\Sat"] \ar[d, "\BC"]
    & \QQ[X_1^\pm, \dots, X_n^\pm]^{\Sym(n)} \ar[d, "\BC"] \\
  \HH(\U(\VV_n^+)) \ar[r, "\sim"', "\Sat"]
    & \QQ[Y_1^\pm, \dots, Y_m^\pm]^{W_m} \\
    \HH(S_n(F)) \ar[u, "\sim", "\BC_{S_n}^{\eta^{n-1}}"']
\end{tikzcd}
\end{center}

\subsection{Example: Calculation of $\BC_{S_n}$ when $n = 3$}
As an example, we make the base change completely explicit
in the special case $n = 3$, where $m = \left\lfloor n/2 \right\rfloor = 1$.
(In this case $\BC_{S_n}^{\eta^{n-1}} = \BC_{S_n}$ as $\eta^2 = 1$.)
The completed result is \Cref{prop:BC_S3}.
This can be considered a digression and this section will not be needed later.

This calculation parallels the $n = 2$ case that was done in \cite[Lemma 7.1.1]{ref:AFLspherical}.
When it is not more difficult, some of the results will be stated for all $n$,
rather than $n = 3$ specifically.

\subsubsection{Overview}
Throughout this subsection, we use the shorthand
\[ \varpi^{(n_1, n_2, n_3)} \coloneqq \diag(\varpi^{n_1}, \varpi^{n_2}, \varpi^{n_3}). \]
As a $\QQ$-module, the spaces $\HH(\U(\VV_n^+))$ and $\HH(S_n(F))$
have a canonical basis of indicator functions indexed by $\ZZ$:
\begin{itemize}
  \ii $\HH(S_n(F))$ has $\QQ$-module basis $\mathbf{1}_{K'_{S,j}}$ for $j \ge 0$.
  \ii $\HH(\U(\VV_n^+))$ has a $\QQ$-module basis given by the indicator functions
  \[ \mathbf{1}_{\varpi^{-r} \Mat(\OO_E) \cap \U(\VV_n^+)} \]
  for $r \ge 0$.
\end{itemize}
On the other hand, the natural $\QQ$-module basis for $\HH(\GL_n(E))$, namely
\[ \mathbf{1}_{K'\varpi^{(n_1, n_2, n_3)}K'} \]
is given by triples of integers $n_1 \ge n_2 \ge n_3 \ge 0$, and is much larger.
So explicit calculations for the $\rproj_\ast$ or the Satake transforms viewed in
$\CC[X_1, X_2, X_3]^{\Sym(n)}$ are nontrivial if one works with the entire basis.

Hence the overall strategy, to reduce the amount of work we have to do,
is to focus on only the $\ZZ$-indexed elements
\[
  \mathbf{1}_{\Mat_3(\OO_E), v\circ\det=r}
  = \sum_{\substack{n_1 \ge n_2 \ge n_3 \\ n_1 + n_2 + n_3 = r}}
  \mathbf{1}_{K'\varpi^{(n_1, n_2, n_3)}K'} \in \HH(\GL_n(E))
\]
for $r \ge 0$.
This aggregated indicator function is easier to compute,
because given an explicit matrix it is somewhat easier
to evaluate \[ \mathbf{1}_{\Mat_3(\OO_E), v\circ\det=r} \]
at it (one only needs to check it has $\OO_E$ entries
and that the determinant has valuation $r$,
rather than determining the exact coset $K'\varpi^{(n_1, n_2, n_3)}K'$).

\subsubsection{Satake transform of the determinant characteristic function on the top arrow}
This is the easiest calculation, and we do it for all $n$ rather than just $n = 3$.
\begin{proposition}
  [Satake transform for $v \circ \det = r$]
  For every integer $r \ge 0$, we have
  \[ \Sat(\mathbf{1}_{\Mat_n(\OO_E), v\circ\det=r})
    = q^{(n-1)r} \sum_{e_1 \dots + e_n = r} X_1^{e_1} \dots X_n^{e_n}. \]
\end{proposition}
\begin{proof}
  We evaluate the coefficient $X_1^{e_1} \dots X_n^{e_n}$.
  Choose a cocharacter $\mu$,
  and suppose $\mu(\varpi) = \varpi^{(e_1, \dots, e_n)}$ with $n_1 \ge n_2 \ge n_3$.
  Let $q_E = q^2$ be the residue characteristic of $E$.
  Take the upper triangular matrices as our Borel subgroup as usual,
  so the unipotent radical of this Borel subgroup
  are the unipotent upper triangulars $N'$ which we describe as
  \[ N' \coloneqq \left\{
      \begin{pmatrix}
      1 & y_{12} & y_{13} & \dots & y_{1n} \\
        & 1 & y_{23} & \dots & y_{2n} \\
        &   & 1 & \dots & y_{3n} \\
        &   &   & \ddots & \vdots  \\
        &   &   &   & 1
      \end{pmatrix}
    \mid y_{12}, \dots, y_{(n-1)n}\in E \right\} \]
  and with additive Haar measure is $\odif{y_{12}, y_{23} \dotso, y_{(n-1)n}}$.
  Recall also the Weyl vector for $\GL_n(E)$ is just
  \[ \rho_{\GL_n(E)} = \left< \frac{n-1}{2}, \frac{n-3}{2}, \dots, -\frac{n-1}{2} \right>. \]
  Compute
  \begin{align*}
    &\Sat(\mathbf{1}_{\Mat_n(\OO_E), v\circ\det=r})(\mu(\varpi)) \\
    &= \delta(\mu(\varpi))^\half \int_{n' \in N'}
      \mathbf{1}_{\Mat_n(\OO_E, v\circ \det = r)} (\mu(\varpi) n') \odif{n'} \\
    &= q_E^{-\left< \mu, \rho\right>}
    \underbrace{\int_{y_{12} \in E} \int_{y_{13} \in E} \dotso \int_{y_{(n-1)n} \in E}}_{\binom n2 \text{ integrals}} \\
    &\qquad
      \mathbf{1}_{\Mat_3(\OO_E), v \circ \det = r}
      \left( \begin{pmatrix}
        \varpi^{e_1} & \varpi^{e_1} y_{12} & \varpi^{e_1} y_{13} & \dots & \varpi^{e_1} y_{1n} \\
        & \varpi^{e_2} & \varpi^{e_2} y_{23} & \dots & \varpi^{e_2} y_{2n} \\
        &   & \varpi^{e_3} & \dots & \varpi^{e_3} y_{3n} \\
        &   &   & \ddots & \vdots  \\
        &   &   &   & \varpi^{e_n}
        \end{pmatrix} \right) \\
    &\qquad \odif{y_{12}, y_{23} \dotso, y_{(n-1)n}} \\
    &= q_E^{-\left(\frac{n-1}{2}e_1 + \frac{n-3}{2}e_2 + \dots + -\frac{n-1}{2} e_n \right)}
    \mathbf{1}_{e_1 + \dots + e_n = r}
    \underbrace{\int_{y_{12} \in E} \int_{y_{13} \in E} \dotso \int_{y_{(n-1)n} \in E}}_{\binom n2 \text{ integrals}} \\
    &\qquad \prod_{1 \le i < j \le n} \mathbf{1}_{\OO_E}(\varpi^{e_i} y_{ij}) \odif{y_{ij}} \\
    &= q_E^{-\left(\frac{n-1}{2}e_1 + \frac{n-3}{2}e_2 + \dots + -\frac{n-1}{2} e_n \right)}
    \mathbf{1}_{e_1 + \dots + e_n = r} \prod_{1 \le i < j \le n} q_E^{e_i} \\
    &= q_E^{-\left(\frac{n-1}{2}e_1 + \frac{n-3}{2}e_2 + \dots + -\frac{n-1}{2} e_n \right)}
    \mathbf{1}_{e_1 + \dots + e_n = r} \prod_{1 \le i \le n} q_E^{(n-i)e_i} \\
    &= \mathbf{1}_{e_1 + \dots + e_n = r} \prod_{1 \le i \le n}^n q_E^{\frac{n-1}{2} e_i} \\
    &= q_E^{\frac{n-1}{2} r} \mathbf{1}_{e_1 + \dots + e_n = r} \\
    &= \begin{cases}
      q^{\frac{n-1}{2} r} & \text{if } e_1 + \dots + e_n = r \\
      0 & \text{otherwise}.
    \end{cases}
  \end{align*}
  This gives the sum claimed earlier.
\end{proof}

\subsubsection{Satake transform of the indicator on the bottom arrow}
\begin{proposition}
  [Satake transform for $\varpi^{-r} \Mat_3(\OO_E) \cap \U(\VV_3^+)$]
  For each $r \ge 0$ we have
  \[ \Sat\left(\mathbf{1}_{\varpi^{-r} \Mat_3(\OO_E) \cap \U(\VV_3^+)}\right)
    = \sum_{i=0}^r q^{2r - \mathbf{1}_{r \equiv i \bmod 2}} Y_1^{\pm i} \]
  where we adopt the shorthand
  \[
    Y_1^{\pm i} \coloneqq
    \begin{cases}
      Y_1^i + Y_1^{-i} & i > 0 \\
      1 & i = 0 .
    \end{cases}
  \]
\end{proposition}
\begin{proof}
  We first need to describe \[ N = N' \cap \U(\VV_3^+) \] a little more carefully.
  For $n \in N'$ we have
  \[
    n^\ast \beta n
    =
    \begin{pmatrix} 1 \\ \bar{y_1} & 1 \\ \bar{y_2} & \bar{y_3} & 1 \end{pmatrix}
    \beta
    \begin{pmatrix}
      1 & y_1 & y_2 \\
        & 1 & y_3 \\
        & & 1
    \end{pmatrix}
    = \begin{pmatrix}
      & & 1 \\
      & 1 & y_3 + \bar{y_1} \\
      1 & y_1 + \bar{y_3} & y_2 + \bar{y_2} + y_3 \bar{y_3}
    \end{pmatrix}.
  \]
  So $n \in N$ if and only if the above matrix equals $\beta$, which means
  \[ 0 = y_3 + \bar{y_1} = y_2 + \bar{y_2} + y_3 \bar{y_3}. \]
  Then we can re-parametrize by $z_1, z_2, z_3 \in F$ according to
  \begin{align*}
    y_3 &= z_1 + z_2 \sqrt\eps \\
    y_2 &= -\frac{z_1^2 + z_2^2 \eps}{2} + z_3\sqrt\eps \\
    y_1 &= -z_1 + z_2\sqrt\eps.
  \end{align*}
  Back to the original task.
  For each $i \ge 0$ we can evaluate the Satake transform at the element
  $\nu(\varpi) = \diag(\varpi^i, 1, \varpi^{-i})$, for the cocharacter $\nu$
  corresponding to $Y_1^i + Y_1^{-i}$:
  \begin{align*}
    &\Sat\left( \mathbf{1}_{\varpi^{-r} \Mat_3(\OO_E) \cap \U(\VV_3^+)}\right)
      \left( \nu(\varpi)  \right) \\
    &= \delta(\nu(\varpi))^\half \int_{n \in N}
      \mathbf{1}_{\varpi^{-r} \Mat_3(\OO_E) \cap \U(\VV_3^+)}
      \left( \nu(\varpi) n' \right) \odif n \\
    &= \delta(\nu(\varpi))^\half \int_{n \in N}
      \mathbf{1}_{{\varpi^{-r}} \Mat_3(\OO_E) \cap \U(\VV_3^+)}
      \left( \begin{pmatrix} \varpi^i & \varpi^i y_1 & \varpi^i y_2 \\
               & 1 & y_3 \\
               & & \varpi^{-i} \end{pmatrix} \right) \odif n
  \end{align*}
  The matrix itself is always in $\U(\VV_3^+)$, because it's the product of two unitary matrices.
  So the indicator needs to check whether all the entries have valuation at least $-r$.
  If we switch characterization to the coordinates $z_1$, $z_2$, $z_3$ we described earlier,
  we see that the conditions are
  \begin{align*}
    i &\le r, \\
    v(z_1) &\ge -r, \\
    v(z_2) &\ge -r,\\
    v(z_3) &\ge -(r+i),\\
    v(z_1^2 + z_2^2 \eps) &\ge -(r+i).
  \end{align*}
  Assume $i \le r$ henceforth.
  The condition for $z_1$ and $z_2$ then really says
  \[ \min(v(z_1), v(z_2)) \ge -\left\lfloor \frac{r+i}{2} \right\rfloor. \]
  So the integral factors as a triple integral
  \[
    \int_{z_1 \in F}
    \int_{z_2 \in F}
    \int_{z_3 \in F}
    \mathbf{1}_{\varpi^{-\left\lfloor \frac{r+i}{2} \right\rfloor} \OO_F}(z_1)
    \mathbf{1}_{\varpi^{-\left\lfloor \frac{r+i}{2} \right\rfloor} \OO_F}(z_2)
    \mathbf{1}_{\varpi^{-(r+i)} \OO_F}(z_3)
    \odif{z_1,z_2,z_3}
  \]
  which is equal to
  \[ q^{2\left\lfloor \frac{r+i}{2} \right\rfloor+r+i}. \]
  Meanwhile, $\delta(\nu(\varpi))^{\half} = q^{-2i}$.
  In summary,
  \[
    \Sat\left( \mathbf{1}_{\varpi^{-r} \Mat_3(\OO_E) \cap \U(\VV_3^+)}\right) \left( \nu(\varpi) \right)
    =
    \begin{cases}
      q^{2\left\lfloor \frac{r+i}{2} \right\rfloor - i + r} & i \le r \\
      0 & i > r
    \end{cases}
  \]
  Finally, since
  \[ 2\left\lfloor \frac{r+i}{2} \right\rfloor - i + r
    = \begin{cases}
      2r & r+i \text{ is even} \\
      2r-1 & r+i \text{ is odd}
    \end{cases}
  \]
  we get the formula claimed.
\end{proof}

\subsubsection{Integration over fiber}
\begin{proposition}
  [Integration over fiber]
  For every integer $r \ge 0$, we have
  \begin{align*}
    &\rproj_\ast(\mathbf{1}_{\Mat_3(\OO_E), v\circ\det=r}) \\
    &= \sum_{j=0}^r \left(
      \sum_{i=0}^{2(r-j)} \min \left( 1 + \left\lfloor \frac i2 \right\rfloor,
        1 + \left\lfloor \frac{2(r-j)-i}{2} \right\rfloor \right) q^i \right)
        \mathbf{1}_{K'_{S,j}}.
  \end{align*}
\end{proposition}
\begin{proof}
  The coefficient of $\mathbf{1}_{K'_{S,j}}$ will be equal to
  the evaluation of the integral at any $g$ such that $g\bar{g} \in K'_{S,j}$.
  Fixing $j \ge 0$, we are going to take the choice
  \[
    g = \begin{pmatrix}
      1 &   & \varpi^{-j} \sqrt{\eps} \\
      & 1 \\
      &   & 1
    \end{pmatrix}.
  \]
  We need to check this choice of $g$ indeed satisfies $g\bar{g}\inv \in K'_{S,j}$.
  This follows as
  \[ \bar{g} = \begin{pmatrix} 1 &   & -\varpi^{-j} \sqrt{\eps} \\ & 1 \\ &   & 1 \end{pmatrix}
    \implies \bar{g}\inv = \begin{pmatrix} 1 &   & \varpi^{-j} \sqrt{\eps} \\ & 1 \\ &   & 1 \end{pmatrix}
  \]
  and therefore
  \[
    g\bar{g}\inv = \begin{pmatrix}
      1 &   & 2\varpi^{-j} \sqrt{\eps} \\
      & 1 \\
      &   & 1
    \end{pmatrix} \in K'_{S,j}
  \]
  as needed.

  Having chosen the representative $g$, we aim to calculate the right-hand side of
  \[
    \rproj_\ast(\mathbf{1}_{\Mat_3(\OO_E), v\circ\det=r})(g\bar{g})
    = \int_{h \in \GL_3(F)} \mathbf{1}_{\Mat_3(\OO_E), v\circ\det=r}(gh) \odif h.
  \]
  We take (non-Archimedean) Iwasawa decomposition of $h \in \GL_3(F)$ to rewrite it as
  \[
    h =
    \begin{pmatrix} x_1 \\ & x_2 \\ && x_3 \end{pmatrix}
    \begin{pmatrix} 1 & y_1 & y_2 \\ & 1 & y_3 \\ & & 1 \end{pmatrix}
    k
  \]
  for $k \in \GL_3(\OO_F) \subseteq K'$, which does not affect the indicator function.
  Here $x_1, x_2, x_3 \in F^\times$ and $y_1, y_2, y_3 \in F$.
  In that case, note that
  \begin{align*}
    gh
    &=
    \begin{pmatrix}
      1 &   & \varpi^{-j}\sqrt\eps \\
      & 1 \\
      &   & 1
    \end{pmatrix}
    \begin{pmatrix} x_1 \\ & x_2 \\ && x_3 \end{pmatrix}
    \begin{pmatrix} 1 & y_1 & y_2 \\ & 1 & y_3 \\ & & 1 \end{pmatrix} k \\
    &=
    \begin{pmatrix}
      1 &   & \varpi^{-j}\sqrt\eps \\
      & 1 \\
      &   & 1
    \end{pmatrix}
    \begin{pmatrix} x_1 & x_1 y_1 & x_1 y_2 \\ & x_2 & x_2 y_3 \\ & & x_3 \end{pmatrix} k \\
    &=
    \begin{pmatrix}
      x_1 & x_1 y_1 & x_1 y_2 + x_3 \varpi^{-j} \sqrt\eps \\
      & x_2 & x_2 y_3 \\
      & & x_3
    \end{pmatrix}
    k.
  \end{align*}
  Hence, we can rewrite the $\rproj_\ast(\mathbf{1}_{\Mat_3(\OO_E), v\circ\det=r})$
  as a six-fold integral
  \begin{align*}
    &\rproj_\ast(\mathbf{1}_{\Mat_3(\OO_E), v\circ\det=r}) \\
    &= \int_{x_1 \in F^\times} \int_{x_2 \in F^\times} \int_{x_3 \in F^\times}
    \int_{y_1 \in F} \int_{y_2 \in F} \int_{y_3 \in F} \\
    &\quad \mathbf{1}_{\Mat_3(\OO_E), v\circ\det=r} \left(
    \begin{pmatrix}
      x_1 & x_1 y_1 & x_1 y_2 + x_3 \varpi^{-j} \sqrt\eps \\
      & x_2 & x_2 y_3 \\
      & & x_3
    \end{pmatrix}
    \right) \\
    &\quad \odif[{\times,\times,\times}]{x_1,x_2,x_3,y_1,y_2,y_3}.
  \end{align*}
  Apparently the indicator function only depends on the valuations,
  so accordingly we rewrite the six-fold integral as a discrete sum over the valuations
  $\alpha_i \coloneqq v(x_i)$.
  Then the conditions are that
  \begin{align*}
    &\alpha_1 \ge 0, \quad \alpha_2 \ge 0, \quad \alpha_3 \ge j \\
    &v(y_1) \ge - \alpha_1, \quad v(y_2) \ge - \alpha_1, \quad v(y_3) \ge -\alpha_2.
  \end{align*}
  We have $\Vol(\varpi^{\alpha_i} \OO_F^\times) = 1$
  and $\Vol(\varpi^{-\alpha_i} \OO_F) = q^{\alpha_i}$.
  Hence the integral can be rewritten as the discrete sum
  \begin{align*}
    \sum_{\substack{\alpha_1 + \alpha_2 + \alpha_3 = r \\ \alpha_1 \ge 0 \\ \alpha_2 \ge 0 \\ \alpha_3 \ge j}}
    q^{\alpha_1} \cdot q^{\alpha_1} \cdot q^{\alpha_2}
    &= \sum_{\substack{\alpha_1 + \alpha_2 \le r-j \\ \alpha_1 \ge 0 \\ \alpha_2 \ge 0}}
    q^{2\alpha_1+\alpha_2} \\
    &= \sum_{i=0}^{2(r-j)}
    \min \left( 1 + \left\lfloor \frac i2 \right\rfloor,
      1 + \left\lfloor \frac{2(r-j)-i}{2} \right\rfloor
    \right) q^i
  \end{align*}
  as desired.
\end{proof}

\subsubsection{Base change from $\HH(\U(\VV_3^+))$ to $\HH(S_3(F))$}
We first need to determine an element of $\HH(\U(\VV_n^+))$
which is in the pre-image of
\[ \mathbf{1}_{\varpi^{-r} \Mat_3(\OO_E) \cap \U(\VV_3^+)} \]
under $\BC \colon \HH(\GL_3(E)) \to \HH(\U(\VV_3^+))$.

For convenience, we define the shorthand
\[
  \HH(\GL_3(E)) \ni
  f'_r \coloneqq \begin{cases}
    \mathbf{1}_{\Mat_3(\OO_E), v \circ \det = r} & r \ge 0 \\
    0 & r < 0
  \end{cases}
\]
for every integer $r$.
We start with the following intermediate calculation.
\begin{align*}
  &\BC\left( \Sat \left( f'_r - q^2 f'_{r-1} \right) \right) \\
  &= \BC \left(
    q^{2r} \sum_{n_1+n_2+n_3=r} X_1^{n_1} X_2^{n_2} X_3^{n_3}
    - q^2 \cdot q^{2(r-1)} \sum_{n_1+n_2+n_3=(r-1)} X_1^{n_1} X_2^{n_2} X_3^{n_3} \right) \\
  &= q^{2r} \left( \sum_{n_1+n_2+n_3=r} Y_1^{n_1-n_3} - \sum_{n_1+n_2+n_3=(r-1)} Y_1^{n_1-n_3} \right) \\
  &= q^{2r} \left( \sum_{n_1+n_3=r} Y_1^{n_1-n_3} \right) \\
  &= q^{2r} \left( Y_1^{r} + Y_1^{r-2} + \dots + Y_1^{-r} \right).
  \intertext{Replacing $r$ with $r-1$ gives}
  &\BC\left( \Sat \left(f'_{r-1} - q^2 f'_{r-2} \right) \right) \\
  &= q^{2r-2} \left( Y_1^{r-1} + Y_1^{r-3} + \dots + Y_1^{-(r-1)} \right).
  \intertext{Adding the former equation to $q$ times the latter gives}
  &\BC\left( \Sat\left( f'_r + (q-q^2) f'_{r-1} - q^3 f'_{r-2} \right) \right) \\
  &= q^{2r} \left( Y_1^r + Y_1^{r-2} + \dots + Y_1^{-r} \right)
  + q^{2r-1} \left( Y_1^{r-1} + Y_1^{r-3} + \dots + Y_1^{-(r-1)} \right) \\
  &= \Sat(\mathbf{1}_{\varpi^{-r} \Mat_3(\OO_E) \cap \U(\VV_3^+)}).
\end{align*}
This shows that
\[ \BC(f'_r + (q-q^2) f'_{r-1} - q^3 f'_{r-2}) =
  \mathbf{1}_{\varpi^{-r} \Mat_3(\OO_E) \cap \U(\VV_3^+)} \]
so indeed $f'_r + (q-q^2) f'_{r-1} - q^3 f'_{r-2}$
lies in the desired pre-image of the map $\BC \colon \HH(\GL_3(E)) \to \HH(\U(\VV_3^+))$.

On the other hand, it is easy to check that
\begin{align*}
  &\rproj_\ast(f'_r -q^2 f'_{r-1}) \\
  &= \sum_{j=0}^r \Bigg[
      \sum_{i=0}^{2(r-j)} \min \left( 1 + \left\lfloor \frac i2 \right\rfloor,
      1 + \left\lfloor \frac{2(r-j)-i}{2} \right\rfloor \right) q^i \\
  &\qquad - \sum_{i=0}^{2(r-1-j)} \min \left( 1 + \left\lfloor \frac i2 \right\rfloor,
    1 + \left\lfloor \frac{2((r-1)-j)-i}{2} \right\rfloor \right) q^{i+2}
  \Bigg] \mathbf{1}_{K'_{S,j}} \\
  &= \sum_{j=0}^r \left[ 1+q+q^2+\dots+q^{r-j} \right] \mathbf{1}_{K'_{S,j}} \\
  \intertext{so}
  &\rproj_\ast(f'_r -q^2 f'_{r-1} + q \left( f'_{r-1} - q^2 f'_{r-3} \right)) \\
  &= \sum_{j=0}^r \left[ (1+q+q^2+\dots+q^{r-j})+(q+q^2+\dots+q^{r-j}) \right] \mathbf{1}_{K'_{S,j}} \\
  &= \sum_{j=0}^r \left[ 1 + 2q + 2q^2 + \dots + 2q^{r-j} \right] \mathbf{1}_{K'_{S,j}}.
\end{align*}
To summarize, the completed commutative diagram can be written in full as
\begin{center}
\begin{tikzcd}
  \begin{tabular}{c} $f'_r + (q-q^2) f'_{r-1}$ \\ $- q^3 f'_{r-2} \in \HH(\GL_3(E))$ \end{tabular}
    \ar[dd, "\rproj_\ast"', mapsto, bend right = 50]
    \ar[r, "\Sat", mapsto] \ar[d, "\BC", mapsto]
    & \dots \in \QQ[X_1^\pm, X_2^\pm, X_3^\pm]^{\Sym(3)} \ar[d, "\BC", mapsto] \\
  \begin{tabular}{c} $\mathbf{1}_{\varpi^{-r} \Mat_3(\OO_E) \cap \U(\VV_3^+)}$ \\ $\in \HH(\U(\VV_3^+))$ \end{tabular}
    \ar[r, "\Sat", mapsto]
    & \begin{tabular}{l}
      $q^{2r} \left( Y_1^{\pm r} + \dotsb + Y_1^{\mp r} \right)$ \\
      $+ q^{2r-1} \left( Y_1^{\pm(r-1)} + \dotsb + Y_1^{\mp(r-1)} \right)$ \\
      $\in \QQ[Y_1^\pm]^{W_1}$
      \end{tabular} \\
  \begin{tabular}{c}
    $\sum_{j=0}^r \big[ 1 + 2q + 2q^2$ \\
    $+ \dots + 2q^{r-j} \big] \mathbf{1}_{K'_{S,j}}$ \\
    $\in \HH(S_3(F))$
  \end{tabular} \ar[u, "\sim", "\BC_{S_3}"', mapsto]
\end{tikzcd}
\end{center}

Thus, we arrive at the following:
\begin{proposition}
  [Base change $\BC_{S_3}$]
  \label{prop:BC_S3}
  For $n = 3$, we have
  \begin{align*}
    \BC_{S_3} \left( \sum_{j=0}^r \left[ 1 + 2q + 2q^2 + \dots + 2q^{r-j} \right]
    \mathbf{1}_{K'_{S,j}} \right)
    &= \mathbf{1}_{\varpi^{-r} \Mat_3(\OO_E) \cap \U(\VV_3^+)} \\
    \BC_{S_3} \left( \mathbf{1}_{K'_{S,r}}
    + \sum_{j=0}^{r-1} 2q^{r-j} \mathbf{1}_{K'_{S,j}} \right)
    &= \mathbf{1}_{K\varpi^{(r,0,-r)}K}
  \end{align*}
  for every integer $r \ge 0$.
\end{proposition}
\begin{proof}
  The first equation is the one we just proved.
  The second one follows by noting that
  \[
    \mathbf{1}_{K\varpi^{(r,0,-r)}K}
    = \mathbf{1}_{\varpi^{-r} \Mat_3(\OO_E) \cap \U(\VV_3^+)}
    - \mathbf{1}_{\varpi^{-(r-1)} \Mat_3(\OO_E) \cap \U(\VV_3^+)}
  \]
  so one merely subtracts the left-hand sides evaluated at $r$ and $r-1$ for $r \ge 1$
  to get
  \begin{align*}
    &\phantom= \sum_{j=0}^r \left[ 1 + 2q + 2q^2 + \dots + 2q^{r-j} \right] \mathbf{1}_{K'_{S,j}} \\
    &\qquad - \sum_{j=0}^{r-1} \left[ 1 + 2q + 2q^2 + \dots + 2q^{(r-1)-j} \right] \mathbf{1}_{K'_{S,j}} \\
    &= \mathbf{1}_{K'_{S,r}} +
      \sum_{j=0}^{r-1} \left[ 1 + 2q + 2q^2 + \dots + 2q^{r-j} \right] \mathbf{1}_{K'_{S,j}} \\
    &\qquad - \sum_{j=0}^{r-1} \left[ 1 + 2q + 2q^2 + \dots + 2q^{(r-1)-j} \right] \mathbf{1}_{K'_{S,j}} \\
    &= \mathbf{1}_{K'_{S,r}} + \sum_{j=0}^{r-1} \left[ 2q^{r-j}\mathbf{1}_{K'_{S,j}} \right].
  \end{align*}
  as claimed.
\end{proof}

\section{Definition of the weighted orbital integral for $S_n(F)$}
\label{ch:orbital0}
We briefly mention the definition of the weighted orbital integral
that appears in \Cref{conj:inhomog_spherical}.
However, we will not reuse this definition later on.
This definition is only used for comparison with later definitions
and to provide context for \Cref{conj:inhomog_spherical}.

Let $H = \GL_{n-1}(F)$.
Then $H$ has a natural embedding into $\GL_n(E)$ by
\[ h \mapsto \begin{pmatrix} h & 0 \\ 0 & 1 \end{pmatrix} \]
which endows it with an action $S_n(F)$.
Then our weighted orbital integral is defined as follows.
\begin{definition}
  [{\cite[Equation (3.2.3)]{ref:AFLspherical}}]
  \label{def:orbital0}
  For brevity let $\eta(h) \coloneqq \eta(\det h)$ for $h \in H$.
  For $\gamma \in S_n(F)$, $\phi \in \HH(S_n(F))$, and $s \in \CC$,
  we define the \emph{weighted orbital integral} by
  \[ \Orb(\gamma, \phi, s) \coloneqq
    \int_{h \in H} \phi(h\inv \gamma h) \eta(h)
    \left\lvert \det(h) \right\rvert_F^{-s} \odif h. \]
\end{definition}
\begin{definition}
  [The abbreviation $\partial \Orb(\gamma, \phi)$]
  From now on we will abbreviate
  \[ \partial \Orb(\gamma, \phi)
    \coloneqq \left. \pdv{}{s} \right\rvert_{s=0} \Orb(\gamma, \phi, s). \]
\end{definition}

We remark that this weighted orbital integral is related to
an (unweighted) orbital integral on the unitary side
by the so-called relative fundamental lemma.
Specifically, for $g \in \U(\VV_n^+)$ and $f \in \HH(\U(\VV_n^+))$,
we define the (unweighted) orbital integral by
\[ \Orb^{\U(\VV_n^+)}(g, f) \coloneqq \int_{\U(\VV_n^+)} f(x^{-1}gx) \odif x. \]
Then the following result is true.
\begin{theorem}
  [Relative fundamental lemma; {\cite[Theorem 1.1]{ref:leslie}}]
  \label{thm:rel_fundamental_lemma}
  Let $\phi \in \HH(S_n(F))$ and $\gamma \in S_n(F)\rs$.
  \[ \omega(\gamma) \Orb(\phi, \gamma, 0)
    = \begin{cases}
      0 & \text{if }\gamma \in S_n(F)\rs^- \\
      \Orb^{\U(\VV_n^+)}(g, \BC^{\eta^{n-1}}_{S_n}(\phi)) & \text{if } \gamma \in S_n(F)\rs^+
    \end{cases}
  \]
  where the transfer factor $\omega$ is defined in \Cref{ch:geo}.
\end{theorem}

\section{Synopsis of the weighted orbital integral
  $\Orb((\gamma, \uu, \vv^\top), \phi \otimes \oneV, s)$
  for $(\gamma, \uu, \vv^\top) \in (S_2(F) \times V'_2(F))\rs$
  and $\phi \in \HH(S_2(F))$}
\label{ch:orbitalFJ0}

Throughout this section, $H = \GL_n(F)$ (rather than $H = \GL_{n-1}(F)$)
and $K' = \GL_n(\OO_F)$.
For the concrete calculation, we are mostly interested in the case $n = 2$.
The goal of this chapter is to define the orbital integral in
\Cref{thm:semi_lie_formula} and give a precise statement
of the parameters used to state the formula,
as well as the dependencies between the parameters that needs to hold
in order for the matching to work.

\subsection{Definition}
For our conjecture, it will be enough to define the weighted orbital integral
in the case where our function is of the form
\[ \phi \otimes \oneV \]
where $\phi \in \HH(S_n(F))$ is the left component, and
the right component is the indicator function defined in the obvious way:
\begin{align*}
  \mathbf{1}_{\OO_F^n \times (\OO_F^n)^\vee} \colon V'_n(F) &\to \{0,1\} \\
  (\uu, \vv^\top) &\mapsto
  \begin{cases}
    1 & \uu \text{ and } \vv^\top \text{ have } \OO_F \text{-entries} \\
    0 & \text{otherwise}.
  \end{cases}
\end{align*}

Then, unsurprisingly from the definition of our action as
\[ h \cdot \guv = (h\gamma h\inv, h\uu, \vv^\top h\inv) \]
we analogously define the weighted orbital integral as follows.
\begin{definition}
  [{\cite[\S1.3]{ref:liuFJ}}]
  \label{def:orbitalFJ}
  For brevity let $\eta(h) \coloneqq \eta(\det h)$ for $h \in H$.
  For $\guv \in S_n(F) \times V'_n(F)$,
  $\phi \in \HH(S_n(F))$, and $s \in \CC$,
  we define the weighted orbital integral by
  \begin{align*}
    & \Orb(\guv, \phi \otimes \oneV, s) \\
    &\coloneqq \int_{h \in H} \phi(h\inv \gamma h) \oneV(h \uu, \vv^\top h^{-1})
    \eta(h) \left\lvert \det(h) \right\rvert_F^{-s} \odif h.
  \end{align*}
\end{definition}
\begin{definition}
  [The abbreviation $\partial \Orb(\guv, \phi)$]
  Moving forward we abbreviate
  \[ \partial \Orb(\guv, \phi) \coloneqq
    \left. \pdv{}{s} \right\rvert_{s=0}
    \Orb((\gamma, \uu, \vv^\top), \phi \otimes \oneV, s). \]
\end{definition}

As before it seems this weighted orbital integral should be related to an ordinary one.
To define it, fix a self-dual lattice $\Lambda_n$ in $\VV_n^+$ of full rank.
First, if $(g,u) \in \U(\VV_n^+) \times \VV_n^+$ and $f \in \HH(\U(\VV_n^+))$,
then we define an orbital integral for $\U(\VV_n^+) \times \VV_n^+$ by
\begin{equation}
  \Orb^{\U(\VV_n^+) \times \VV_n^+}\left( (g,u), f \otimes \mathbf{1}_{\Lambda_n} \right)
  \coloneqq \int_{\U(\VV_n^+)} f(x\inv g x) \mathbf{1}_{\Lambda_n}(x^{-1} u) \odif x.
  \label{eq:unweighted_orbital_semi_lie}
\end{equation}
Then in the spirit of \cite[Conjecture 1.9]{ref:liuFJ}
and \Cref{thm:rel_fundamental_lemma}, we propose the following.
\begin{conjecture}
  [Relative fundamental lemma in the semi-Lie case]
  \label{conj:rel_fundamental_lemma_semilie}
  Let $\phi \in \HH(S_n(F))$ and $\guv \in (S_n(F) \times V'_n(F))\rs$.
  Then
  \begin{align*}
    &\omega\guv \Orb(\phi \otimes \oneV, \guv, 0) \\
    &=
    \begin{cases}
      0 & \text{if } \guv \in (S_n(F) \times V'_n(F))\rs^- \\
      \Orb^{\U(\VV_n^+) \times \VV_n^+}((g,u), \BC^{\eta^{n-1}}_{S_n}(\phi) \otimes \mathbf{1}_{\Lambda_n})
      & \text{if } \guv \in (S_n(F) \times V'_n(F))\rs^+
    \end{cases}
  \end{align*}
  where the transfer factor $\omega$ is defined in \Cref{ch:geo}.
\end{conjecture}
Wei Zhang suggests that this conjecture can be proven by similar means
to \Cref{thm:rel_fundamental_lemma},
but since it is not necessary for this paper we do not pursue this proof here.

\subsection{Basis for the indicator functions in $\HH(S_2(F))$}
\label{sec:hecke_basis_FJ}
From now on assume $n = 2$.
Set
\[ S_2(F) \coloneqq \left\{ g \in \GL_2(E) \mid g \bar{g} = \id_2 \right\}. \]

We again have a Cartan decomposition indexed by a single integer $r \ge 0$:
\begin{lemma}
  [Cartan decomposition of $S_2(F)$]
  For each integer $r \ge 0$ let
  \[ K'_{S,r} \coloneqq \GL_2(\OO_E) \cdot
    \begin{pmatrix} 0 & \varpi^r \\ \varpi^{-r} & 0 \end{pmatrix} \]
  denote the orbit of
  $\begin{pmatrix} 0 & \varpi^r \\ \varpi^{-r} & 0 \end{pmatrix}$
  under the left action of $\GL_2(\OO_E)$.
  Then we have a decomposition
  \[ S_2(F) = \coprod_{r \geq 0} K'_{S,r}. \]
\end{lemma}
\begin{proof}
  This is \cite[Equation (7.1.7) and (7.1.8)]{ref:AFLspherical}.
\end{proof}

Like last time, $K'_{S,r}$ is the part of $S_2(F)$
for which the most negative valuation among the nine entries is $-r$.
And as before we abbreviate the $r = 0$ term specifically:
\begin{align*}
  K'_S
  &\coloneqq K'_{S,0} \\
  &= \GL_2(\OO_E) \cdot \begin{pmatrix} & 1 \\ 1 \end{pmatrix} \\
  &= \GL_2(\OO_E) \cdot \id_2 = S_2(F) \cap \GL_2(\OO_E).
\end{align*}

Repeating the definition
\[ K'_{S, \le r} \coloneqq S_2(F) \cap \varpi^{-r} \GL_2(\OO_E)
  = K'_{S,0} \sqcup K'_{S,1} \sqcup \dots \sqcup K'_{S,r} \]
we get a basis of indicator functions for the Hecke algebra $\HH(S_2(F))$:
\begin{corollary}
  [Basis of $\HH(S_2(F))$]
  For $r \ge 0$, the indicator functions $\mathbf{1}_{K'_{S, \le r}}$
  form a basis of $\HH(S_2(F))$.
\end{corollary}

\subsection{Parametrization of $\gamma$}
From now on assume $n = 2$,
and that $(\gamma, \uu, \vv^\top) \in (S_2(F) \times V'_2)\rs$ is regular.

\subsubsection{Identifying an orbit representative}
The weighted orbital integral depends only on the $H$-orbit of $(\gamma, \uu, \vv^\top)$.
Consequently, we may assume without loss of generality
(via multiplication by a suitable change-of-basis $h \in H = \GL_2(F)$) that
\[ \uu = \begin{pmatrix} 0 \\ 1 \end{pmatrix}, \qquad
  \vv^\top = \begin{pmatrix} 0 & e \end{pmatrix} \qquad e \in F. \]
(We know $\uu$ is not the zero vector from the regular condition
applied on $(\gamma, \uu, \vv^\top)$.)

Meanwhile, we will let
$\gamma = \begin{pmatrix} a & b \\ c & d \end{pmatrix} \in \GL_2(F)$
for $a,b,c,d \in F$.
Then, viewed as an element of $\GL_3(F)$ via the embedding we described earlier, we have
\[
  (\gamma, \uu, \vv^\top)
  \mapsto \begin{pmatrix}
    a & b & 0 \\
    c & d & 1 \\
    0 & e & 0
  \end{pmatrix} \in \Mat_3(F).
\]
Thus, our definition of regular requires that
$\begin{pmatrix} 0 \\ 1 \end{pmatrix}$
is linearly independent from $\begin{pmatrix} b \\ d \end{pmatrix}$
and
$\begin{pmatrix} 0 & e \end{pmatrix}$
is linearly independent from $\begin{pmatrix} c & d \end{pmatrix}$.
This is just saying that $b$, $c$, $e$ are all nonzero.
We also know that $\gamma \in S_2(F)$, which gives us relations on $a$, $b$, $c$, $d$
(the same as \cite[Equation (7.3.2)]{ref:AFLspherical}); we have
\[
  \begin{pmatrix} 1 & 0 \\ 0 & 1 \end{pmatrix}
  = \begin{pmatrix} a & b \\ c & d \end{pmatrix} \begin{pmatrix} \bar a & \bar b \\ \bar c & \bar d \end{pmatrix}
  \implies
  \begin{aligned}
    \bar b c = b \bar c &= 1 - a \bar a = 1 - d \bar d, \\
    \qquad d &= - \bar a c / \bar c = -\bar a b / \bar b.
  \end{aligned}
\]

\subsubsection{Simplification due to the matching of non-split unitary group}
Like before, we focus on the case where regular $(\gamma, \uu, \vv^\top)$
matches an element in the non-split unitary group.
As we described in \Cref{prop:valuation_delta_matching_semilie},
this is controlled by the parity of $v(\Delta)$, where
\[ \Delta = \det \left( \left( \vv^\top \gamma^{i+j} \uu \right)_{0 \le i,j \le n-1} \right). \]
When $n=2$, for the representatives we described before,
we have
\[ \left( \vv^\top \gamma^{i+j} \uu \right)_{0 \le i,j \le n-1}
  = \begin{pmatrix} e & de \\ de & bce + d^2e \end{pmatrix} \]
so
\[ \Delta = bce^2 = \frac{b}{\bar b}(1-a \bar a) e^2 . \]
Hence, $v(\Delta)$ is odd if and only if $v(1-a \bar a)$ is odd.
Thus, we restrict attention to the following situation:
\begin{assume}
  \label{assume:a_odd}
  We will assume that
  \[ v(1-a \bar a) \equiv 1 \pmod 2. \]
\end{assume}
In particular, $a$ must be a unit.
And since $d = -\bar a c / \bar c$, it follows $d$ is a unit.
In other words, \Cref{assume:a_odd} gives the direct corollary
\[ v(a) = v(d) = 0. \]

\subsection{Parameters used in the calculation of the weighted orbital integral}
We will state our derivative in terms of the five integers
$r$, $v(b)$, $v(c)$, $v(e)$ and $v(d-a)$.
From \Cref{assume:a_odd}, we actually get that
\begin{assume}
  \label{assume:FJ}
  We have that
  \begin{itemize}
    \ii $v(b) + v(c)$ is an odd positive integer;
    \ii $v(d-a) \ge 0$.
  \end{itemize}
\end{assume}
These are the only constraints between these five numbers we will consider
(together with $r \ge 0$).
However, we mention that we will only be interested in the case when $v(e) \ge 0$
since in the case $v(e) < 0$ we will shortly see that
$\Orb(\guv, \phi \otimes \oneV, s) = 0$ identically in $s$ in that situation.

For convenience, we summarize all the assumptions on the shape of $\guv$ in the following single lemma.
\begin{lemma}
  [Parameters for $\guv \in (S_2(F) \times V'_2(F))\rs$]
  \label{lem:semi_lie_params}
  Suppose $\guv \in (S_2(F) \times V'_2(F))\rs$.
  Then one can choose a representative of the $\GL_2(F)$-orbit of $\guv$ of the form
  \[
      \left( \begin{pmatrix} a & b \\ c & d \end{pmatrix},
      \begin{pmatrix} 0 \\ 1 \end{pmatrix}, \begin{pmatrix} 0 & e \end{pmatrix} \right)
  \]
  where $a,b,c,d,e \in F$ satisfy $bce \neq 0$,
  \begin{align*}
    \bar b c = b \bar c &= 1 - a \bar a = 1 - d \bar d, \\
    d &= - \bar a c / \bar c = -\bar a b / \bar b.
  \end{align*}

  Moreover, we always assume $\guv$ matches an element of
  $(\U(\VV_2^-) \times \VV_2^-)\rs$
  rather than $(\U(\VV_2^+) \times \VV_2^+)\rs$;
  this is equivalent to \Cref{assume:a_odd} which states that
  \[ v(1 - a \bar a) \equiv 1 \pmod 2. \]
  In particular, we may assume $v(a) = v(d) = 0$ and $v(b) + v(c) \ge 1$ is odd
  (\Cref{assume:FJ}).
\end{lemma}

\section{Support of the weighted orbital integral for $S_2(F) \times V_2'(F)$}
\label{ch:orbitalFJ1}

We assume $\guv$ is as in \Cref{lem:semi_lie_params} throughout this chapter.

\subsection{Iwasawa decomposition}
The overall method is to take the Iwasawa decomposition in $KAN$ form:
\begin{lemma}
  [Iwasawa decomposition]
  Every element in $h \in \GL_2(F)$ may be parametrized as
  \[ h = k \begin{pmatrix} x_1 & 0 \\ 0 & x_2 \end{pmatrix}
    \begin{pmatrix} 1 & y \\ 0 & 1 \end{pmatrix} \]
  where $k \in K' = \GL_2(\OO_F)$, $x_1, x_2 \in \OO_F^\times$ and $y \in \OO_F$.
\end{lemma}
Because the orbits are invariant under conjugation by $K'$,
the parameter $k$ can be discarded.
The Haar measure in these coordinates
\[ \left\lvert \frac{x_1}{x_2} \right\rvert \odif[\times] x_1 \odif[\times] x_2 \odif y \]
where we take multiplicative Haar measure on $F^\times$
(normalized so that $\OO_F^\times$ has volume $1$)
and additive Haar measure on $F$ (so $\OO_F$ has volume $1$).

\subsection{Action of upper triangular matrices on $(\gamma, \uu, \vv^\top)$}
We now compute the action of an arbitrary
\[ h = \begin{pmatrix} x_1 & 0 \\ 0 & x_2 \end{pmatrix}
  \begin{pmatrix} 1 & y \\ 0 & 1 \end{pmatrix} \]
on $(\gamma, \uu, \vv^\top)$.
The main term is given by
\begin{align*}
  h \gamma h^{-1}
  &=
  \begin{pmatrix} x_1 & 0 \\ 0 & x_2 \end{pmatrix}
  \begin{pmatrix} 1 & y \\ 0 & 1 \end{pmatrix}
  \begin{pmatrix} a & b \\ c & d \end{pmatrix}
  \begin{pmatrix} 1 & -y \\ 0 & 1 \end{pmatrix}
  \begin{pmatrix} x_1^{-1} & 0 \\ 0 & x_2^{-1} \end{pmatrix} \\
  &=
  \begin{pmatrix} x_1 & 0 \\ 0 & x_2 \end{pmatrix}
  \begin{pmatrix} cy + a & -cy^2+(d-a)y+b \\ c & -cy+d \end{pmatrix}
  \begin{pmatrix} x_1^{-1} & 0 \\ 0 & x_2^{-1} \end{pmatrix} \\
  &=
  \begin{pmatrix} cy + a & \frac{x_1}{x_2} \cdot \left( -cy^2+(d-a)y+b \right) \\
    \frac{x_2}{x_1} \cdot c & -cy+d \end{pmatrix}
\end{align*}
Meanwhile, we have
\begin{align*}
  h \uu &=
    \begin{pmatrix} x_1 & 0 \\ 0 & x_2 \end{pmatrix}
    \begin{pmatrix} 1 & y \\ 0 & 1 \end{pmatrix}
    \begin{pmatrix} 0 \\ 1 \end{pmatrix}
    = \begin{pmatrix} x_1 y \\ x_2 \end{pmatrix} \\
  \vv^\top h^{-1} &=
    \begin{pmatrix} 0 & e \end{pmatrix}
    \begin{pmatrix} 1 & -y \\ 0 & 1 \end{pmatrix}
    \begin{pmatrix} x_1^{-1} & 0 \\ 0 & x_2^{-1} \end{pmatrix}
    = \begin{pmatrix} 0 & \frac{e}{x_2} \end{pmatrix}.
\end{align*}

\subsection{Description of support}
From now on we fix the notation
\begin{align*}
  n_1 &\coloneqq v(x_1) \\
  n_2 &\coloneqq v(x_2).
\end{align*}
Note that although $n_2 \ge 0$, the value of $n_1$ will often be non-positive.
In fact $n_1$ is not particularly simple to work with and we will
prefer to introduce the notation
\begin{equation}
  m \coloneqq n_2 + v(c) + r - n_1
  \label{eq:def_semi_lie_m}
\end{equation}
instead to use as a summation variable.
This is chosen so that $\frac{x_2}{x_1} \cdot c \in \varpi^{-r} \OO_F \iff m \ge 0$.
Note that it follows we have
\begin{equation}
  n_1 + n_2 = 2n_2 - m + v(c) + r
  \label{eq:n1_plus_n2_semi_lie}
\end{equation}

\subsubsection{Collating the linear constraints}
For a given $r \ge 0$, we find that $h$ contributes to the integral exactly
if $h\uu$ and $\vv^\top h\inv$ have $\OO_F$-entries,
and all the entries of $h \gamma h\inv$ are in $\varpi^{-r}\OO_F$.
The former condition is just saying that
\begin{align*}
  v(y) &\ge -n_1, \\
  0 &\le n_2 \le v(e).
\end{align*}
Now we consider the entries of $h\gamma h\inv$.
First, because $a$ and $d$ are units by \Cref{assume:a_odd},
and $r \ge 0$, it follows that
\begin{align*}
  cy + a, -cy + d \in \varpi^{-r}\OO_F
  &\iff cy \in \varpi^{-r}\OO_F \\
  &\iff v(y) \ge -v(c) - r.
\end{align*}
Moreover,
\[ \frac{x_2}{x_1} \cdot c \in \varpi^{-r} \OO_F
  \iff n_2 + v(c) -n_1 \ge - r \iff m \ge 0. \]
In summary, up until now we have the following requirements imposed:
\begin{equation}
  \begin{aligned}
  0 &\le n_2 \le v(e) \\
  0 &\le m \\
  v(y) &\ge \max(-n_1, -v(c) - r) \\
  &= \max(m-n_2, 0) - v(c) - r.
  \end{aligned}
  \label{eq:linear_constraints}
\end{equation}

\subsubsection{The quadratic constraint}
As for the quadratic constraint, we seek $y$ such that
\begin{align*}
  \phantom\iff \frac{x_1}{x_2} \cdot (-cy^2+(d-a)y+b) &\in \varpi^{-r} \OO_F \\
  \iff v\left(-y^2+ \frac{d-a}{c} y + \frac bc \right) &\ge n_2 - n_1 - v(c) - r \\
  &= m - 2v(c) - 2r.
\end{align*}

As before, we complete the square:
\[
  -y^2+ \frac{d-a}{c} y + \frac bc
  = -\left( y - \frac{d-a}{2c} \right)^2 + \frac bc + \frac{(d-a)^2}{4c^2}.
\]
Because $b \bar c = 1 - a \bar a$ has odd valuation,
it follows that $\frac b c = \frac{1-a \bar a}{c\bar c}$ has odd valuation to.
On the other hand, $\frac{(d-a)^2}{4c^2}$ has even valuation.

This motivates us to introduce the following parameter:
\begin{definition}
  [$\theta$]
  We define
  \[ \theta \coloneqq \min \left( v(b)+v(c), 2v(d-a) \right) \ge 0. \]
\end{definition}
Note that $v(b) + v(c)$ is odd,
so $\theta$ takes the odd value if $v(b)+v(c) < 2v(d-a)$ and the even value otherwise.
This definition ensures that
\[ v \left( \frac bc + \frac{(d-a)^2}{4c^2} \right) = \theta - 2v(c). \]

\subsection{Cases based on $\theta$}
Henceforth we consider two cases based on $\theta$.
\begin{description}
  \item[Case 1]
  Let's assume first that
  \[ \theta - 2v(c) \ge m - 2v(c) - 2r \iff m \le \theta + 2r. \]
  Then the only additional condition on $y$ is that
  \[
    v\left( y - \frac{d-a}{2c} \right)
    \ge \left\lceil \frac{m}{2} \right\rceil - v(c) - r.
  \]
  We refer to this as \textbf{Case 1}.

  \item[Case 2\ts+ / Case 2\ts-]
  Otherwise assume that
  \[ \theta - 2v(c) < m - 2v(c) - 2r \iff m > \theta + 2r. \]
  Then in order for $y$ to satisfy the constraint,
  we would need to be in a situation where $2v(y - \frac{d-a}{2c}) = \theta - 2v(c)$.
  So this case could only arise at all when $\theta$ is even, that is
  \[ 0 \le 2v(d-a) = \theta < v(b) + v(c) \]
  (note that $v(d-a) \ge 0$ because $a$ and $d$ are units).
  As the quantity $\frac bc + \frac{(d-a)^2}{4c^2}$ must be a perfect square,
  we denote it by $\tau^2$, with
  \[ v(\tau) = \frac{\theta}{2} - v(c). \]
  This gives us the factorization
  \[ \frac bc = \tau^2 - \frac{(d-a)^2}{4c^2}
    = \left( \tau - \frac{d-a}{2c} \right) \left( \tau + \frac{d-a}{2c} \right). \]
  The left-hand side has odd valuation $v(b) - v(c)$,
  so the two factors on the right have unequal valuations
  and hence exactly one of them has valuation the same as $v(\frac{d-a}{2c}) = v(\tau)$.
  Hence, we agree to fix the choice of the square root $\tau$ so that
  \begin{align*}
    v\left( \tau + \frac{d-a}{2c} \right) &= v(b) - v(c) - v(\tau) = v(b) - \frac{\theta}{2} \\
    v\left( \tau - \frac{d-a}{2c} \right) &= v(\tau) = \frac{\theta}{2} - v(c)
  \end{align*}
  and in particular
  $v\left( \tau + \frac{d-a}{2c} \right) > v\left( \tau - \frac{d-a}{2c} \right)$.

  In any case, the constraint on $y$ is that
  \begin{align*}
    v\left( y - \left( \frac{d-a}{2c} \pm \tau \right) \right)
      &\ge \left( m - 2v(c) - 2r \right) - v(\tau) \\
      &= \left( m - 2v(c) - 2r \right) - \left( \frac{\theta}{2} - v(c) \right) \\
      &= m - \frac{\theta}{2} - v(c) - 2r \\
    v\left( y - \left( \frac{d-a}{2c} \mp \tau \right) \right) &= v(\tau)
      = \frac{\theta}{2} - v(c).
  \end{align*}
  By assumption, the second equation is true
  whenever the first inequality is and we may disregard it.
  \textbf{Case 2\ts+} refers to the situation where the $\pm$ sign is $+$
  and \textbf{Case 2\ts-} refers to the situation where the $\mp$ sign is $-$.
  And these cases must be disjoint because the right-hand sides above are unequal.
\end{description}

\subsubsection{Analysis of Case 1}
The triple $(x_1, x_2, y) \in \OO_F^\times \times \OO_F^\times \times \OO_F$
contributes to the weighted orbital integral in Case 1 exactly if the following identities hold:
\begin{align*}
  0 &\le n_2 \le v(e) \\
  0 &\le m \le \theta + 2r \\
  v(y) &\ge \max(m-n_2,0) - v(c) - r \\
  v\left( y - \frac{d-a}{2c} \right) &\ge \left\lceil \frac{m}{2} \right\rceil - v(c) - r.
\end{align*}
However, from the definitions we already know that
\[ v\left( \frac{d-a}{2c} - 0 \right)
  \ge \frac{\theta - 2v(c)}{2} \ge \frac{m}{2} - v(c) - r \]
so the disks in the last two conditions have nonempty intersection.
Hence the earlier \Cref{lem:no_mastercard} applies to tell us that
the locus of valid $y$ is a single disk whose volume in $\OO_F$ is given by
\[ q^{-\max\left( m-n_2, \left\lceil m/2 \right\rceil, 0 \right) + v(c) + r}
  = q^{-\max\left( m-n_2, \left\lceil m/2 \right\rceil \right) + v(c) + r}. \]
The volume contribution for and $x_1 \in \OO_F^\times$ and $x_2 \in \OO_F^\times$
is also $1$, because $v(x_1)$ and $v(x_2)$ are fixed.
Hence the overall volume of the support in $H$ for this pair $(m, n_2)$ is given by
\begin{align*}
  \left\lvert \frac{x_1}{x_2} \right\rvert q^{n_2 - n_1} \Vol(\{ y \mid \dots \})
  &= q^{n_2 - n_1 -\max\left( m-n_2, \left\lceil m/2 \right\rceil \right) + v(c) + r} \\
  &= q^{m - \max\left( m-n_2, \left\lceil m/2 \right\rceil \right)}.
\end{align*}
And again, this case is summed over
\[ 0 \le n_2 \le v(e), \qquad 0 \le m \le \theta + 2r. \]

\subsubsection{Analysis of Case 2\ts+ and Case 2\ts-}
Again, this case could only occur if $\theta$ is even.
The triple $(x_1, x_2, y) \in \OO_F^\times \times \OO_F^\times \times \OO_F$
contributes to the weighted orbital integral in Case 2\ts+ and Case 2\ts-
exactly if the following identities hold:
\begin{align*}
  0 &\le n_2 \le v(e) \\
  \theta + 2r &< m \\
  v(y) &\ge \max(m-n_2,0) - v(c) - r \\
  v\left( y - \left( \frac{d-a}{2c} \pm \tau \right) \right) &\ge m - \frac{\theta}{2} - v(c) - 2r.
\end{align*}
The last two inequalities specify disks.
So in each case, via \Cref{lem:no_mastercard}
we get a nonzero contribution if and only if the distance between the centers
$0$ and $\frac{d-a}{2c} \pm \tau$ has valuation at least
that of the smaller of the two right-hand sides, that is
\begin{align*}
  v\left( \frac{d-a}{2c} \pm \tau \right)
  &\ge \min\left( \max(m-n_2,0) - v(c) - r, m - \frac{\theta}{2} - v(c) - 2r \right) \\
  &= \min\left( \max(m-n_2,0), m - \frac{\theta}{2} - r \right) - v(c) - r.
\end{align*}
Hence the upper bound on $m$ is given by two different requirements,
depending on which of the two values of
$v\left( \frac{d-a}{2c} \pm \tau \right) + v(c) + r$ is given by the case:
\begin{itemize}
  \ii In Case 2\ts+, we need at least one of the inequalities
  \[
    \begin{cases}
    \max(m-n_2, 0) \le v(b) - \frac{\theta}{2} + v(c) + r, \\
    m \le v(b) + v(c) + 2r
    \end{cases}
  \]
  to hold.
  Now the inequality $0 \le v(b) - \frac{\theta}{2} + v(c) + r$ is always true,
  as $\theta < v(b) + v(c)$, so we can disregard it.
  Therefore this can be rewritten as just
  \[ m \le \max\left( r, n_2 - \frac{\theta}{2} \right) + v(b) + v(c) + r. \]

  \ii In Case 2\ts-, we need at least one of the inequalities
  \[
    \begin{cases}
      \max(m-n_2, 0) \le \frac{\theta}{2} + r \\
      m \le \theta + 2r
    \end{cases}
  \]
  to hold.
  But $m \le \theta+2r$ is always false and $0 \le \frac{\theta}{2} + r$ is always true,
  so this simplifies to
  \[ m \le n_2 + \frac{\theta}{2} + r. \]
\end{itemize}
Assuming $m$ lies in the valid range so that the locus of valid $y$ is nonempty,
it follows that the volume is given exactly by
\[ q^{-\max(m-n_2, m - \frac{\theta}{2} - r, 0) - v(c) - r}
  = q^{-\max(m-n_2, m - \frac{\theta}{2} - r) - v(c) - r}. \]
Hence the overall volume of the support in $H$ for this pair $(m, n_2)$ is given by
\begin{align*}
  \left\lvert \frac{x_1}{x_2} \right\rvert q^{n_2 - n_1} \Vol(\{ y \mid \dots \})
  &= q^{n_2 - n_1 - \max\left( m-n_2, m - \frac{\theta}{2} - r \right) + v(c) + r} \\
  &= q^{m - \max\left( m-n_2, m - \frac{\theta}{2} - r \right)} \\
  &= q^{\min\left( n_2, \frac{\theta}{2} + r \right)}.
\end{align*}
And this sum is over two ranges of $m$
(although the ranges obviously overlap, they set of $y$ they cover is disjoint):
\begin{align*}
  \theta + 2r & < m \le \max\left( r, n_2 - \frac{\theta}{2} \right) + v(b) + v(c) + r \\
  \theta + 2r & < m \le n_2 + \frac{\theta}{2} + r.
\end{align*}
Note the second range could be empty if $n_2$ is small enough,
but the first range is always nonempty.

\section{Evaluation of the weighted orbital integral for $S_2(F) \times V'_2(F)$}
\label{ch:orbitalFJ2}

We continue to assume $\guv$ is as in \Cref{lem:semi_lie_params} throughout this chapter.

We now aggregate the supports we found in the previous section together with the
definition of the weighted orbital integral to extract the desired formulas.

Recall that the weighted orbital integral was defined as
\begin{align*}
  & \Orb(\guv, \phi \otimes \oneV, s) \\
  &\coloneqq
  \int_{h \in H} \phi(h\inv \gamma h)
  \oneV(h \uu, \vv^\top h^{-1})
  \eta(h) \left\lvert \det(h) \right\rvert_F^{-s} \odif h
\end{align*}
and that after taking Iwasawa decomposition as
\[ h = k \begin{pmatrix} x_1 & 0 \\ 0 & x_2 \end{pmatrix}
  \begin{pmatrix} 1 & y \\ 0 & 1 \end{pmatrix} \]
we broke the sum based on $n_1 = v(x_1)$ and $n_2 = v(x_2)$.
For $h$ as above, we know that
\begin{align*}
  \eta(h) &= (-1)^{n_1 + n_2} \\
  \left\lvert \det(h) \right\rvert^{-s}_F &= (q^s)^{n_1 + n_2}.
\end{align*}
Applying \eqref{eq:n1_plus_n2_semi_lie} we find that
\[
  \eta(h)
  \left\lvert \det(h) \right\rvert^{-s}_F
  = (q^s)^{2n_2 - m + v(c) + r}.
\]

\subsection{The contribution for Case 1}
We assume $\theta + 2r \ge 0$, because otherwise the entire sum is empty.
Hence, the total contribution for \textbf{Case 1} is
\begin{align*}
  I^{\text{5}}
  &\coloneqq \sum_{n_2 = 0}^{v(e)} \sum_{m = 0}^{\theta + 2r}
  q^{m - \max\left( m-n_2, \left\lceil m/2 \right\rceil \right)}
  (-q^s)^{2n_2 - m + v(c) + r} \\
  &\coloneqq \sum_{n_2 = 0}^{v(e)} \sum_{m = 0}^{\theta + 2r}
  q^{\min\left( n_2, \left\lfloor m/2 \right\rfloor \right)}
  (-q^s)^{2n_2 - m + v(c) + r}.
\end{align*}
We'll change the summation variable to
\[ k \coloneqq 2n_2 - m + v(c) + r
  \iff m = 2n_2 - k + v(c) + r. \]
Then
\begin{align*}
  I^{\text{5}}
  &\coloneqq \sum_{n_2 = 0}^{v(e)}
  \sum_{k = 2n_2 - \theta + v(c) - r}^{2n_2 + v(c) + r}
  q^{\min\left( n_2, n_2 + \left\lfloor \frac{v(c)+r-k}{2} \right\rfloor \right)} (-q^s)^{k} \\
  &= \sum_{n_2 = 0}^{v(e)}
  \sum_{k = 2n_2 - \theta + v(c) - r}^{2n_2 + v(c) + r}
  q^{n_2 - \max\left( 0, \left\lceil \frac{k-(v(c)+r)}{2} \right\rceil \right)} (-q^s)^{k}.
\end{align*}
We then interchange the order of summation so that $k$ is outside.
Then $k$ runs from the lowest value of $k = - \theta + v(c) - r$
to the largest value $k = 2v(e) + v(c) + r$ over all choices of $n_2$.
Since
\[ 2n_2 - \theta + v(c) - r \le k \le 2 n_2 + v(c) + r \]
then in addition to $0 \le n_2 \le v(e)$ we also need
\[ \frac{k-v(c)-r}{2} \le n_2 \le \frac{k + \theta - v(c) + r}{2}. \]
In other words, we obtain
\begin{align*}
  I^{\text{5}}
  &= \sum_{k = - \theta + v(c) - r}^{2v(e) + v(c) + r}
  (-1)^k (q^s)^k \sum_{n_2 = \max\left(0, \left\lceil \frac{k - v(c) - r}{2} \right\rceil \right)}
  ^{\min\left(v(e), \left\lfloor \frac{k + \theta - v(c) + r}{2} \right\rfloor\right)}
  q^{n_2 - \max\left( 0, \left\lceil \frac{k-(v(c)+r)}{2} \right\rceil \right)} \\
  &= \sum_{k = - \theta + v(c) - r}^{2v(e) + v(c) + r}
  (-1)^k (q^s)^k
  \left( q^{\min\left( v(e), \left\lfloor \frac{k+\theta-v(c)+r}{2} \right\rfloor \right) - \max\left( 0, \left\lceil \frac{k-v(c)-r}{2} \right\rceil \right)} + \dots + q^0 \right).
\end{align*}
Here, we adopt the convention that ellipses of the form
\[ q^i + \dots + q^{i'} \]
denote the expression $q^i + q^{i-1} + \dots + q^{i'}$
(i.e.\ within any ellipses, the exponents are understood to decrease by $1$,
and the sums are always nonempty, meaning $i \ge i'$).

To simplify the exponent, write
\begin{equation}
  \begin{aligned}
    &\min\left( v(e), \left\lfloor \tfrac{k+\theta-v(c)+r}{2} \right\rfloor \right)
    - \max\left( 0, \left\lceil \tfrac{k-v(c)-r}{2} \right\rceil \right) \\
    &= \min\left( v(e), \left\lfloor \tfrac{k+\theta-v(c)+r}{2} \right\rfloor \right)
    + \min\left( 0, \left\lfloor \tfrac{v(c)+r-k}{2} \right\rfloor \right) \\
    &= \min\left( \left\lfloor \tfrac{k+\theta-v(c)+r}{2} \right\rfloor,
      v(e) + \left\lfloor \tfrac{v(c)+r-k}{2} \right\rfloor,
      v(e),
      \left\lfloor \tfrac{k+\theta-v(c)+r}{2} \right\rfloor
      + \left\lfloor \tfrac{v(c)+r-k}{2} \right\rfloor \right).
  \end{aligned}
  \label{eq:case_5_exponent}
\end{equation}
This already completes \Cref{thm:semi_lie_formula}
in the situation when $\theta$ is odd since
\textbf{Case 2\ts+} and \textbf{Case 2\ts-} do not appear at all.
However, let's turn to the remaining cases first.

\subsection{The contribution for Case 2\ts+ and Case 2\ts-}
Herein we assume $\theta = 2v(d-a) > v(b) + v(c)$ is even,
and in particular $\theta \ge 0$.
We get a contribution of
\begin{align*}
  I^{\text{6+}}
  &\coloneqq \sum_{n_2 = 0}^{v(e)}
  \sum_{m = \theta + 2r + 1}
  ^{\max\left( r, n_2 - \frac{\theta}{2} \right) + v(b) + v(c) + r}
    q^{\min\left( n_2, \frac{\theta}{2} + r \right)}
    (-q^s)^{2n_2 - m + v(c) + r} \\
  I^{\text{6-}}
  &\coloneqq \sum_{n_2 = 0}^{v(e)}
  \sum_{m = \theta + 2r + 1}^{n_2 + \frac{\theta}{2} + r}
    q^{\min\left( n_2, \frac{\theta}{2} + r \right)}
    (-q^s)^{2n_2 - m + v(c) + r}.
\end{align*}
We will split $I^{\text{6+}}$ into two parts:
\begin{align*}
  I^{\text{6+}}
  &= \sum_{n_2 = 0}^{\frac{\theta}{2} + r}
  \sum_{m = \theta + 2r + 1}^{v(b) + v(c) + 2r}
    q^{n_2} (-q^s)^{2n_2 - m + v(c) + r} \\
  &+ q^{\frac{\theta}{2} + r} \sum_{n_2 = \frac{\theta}{2} + r + 1}^{v(e)}
  \sum_{m = \theta + 2r + 1}^{n_2 - \frac{\theta}{2} + v(b) + v(c) + r}
    (-q^s)^{2n_2 - m + v(c) + r}.
\end{align*}
Note that the second sum is nonempty only when $v(e) > \frac{\theta}{2} + r$.
So we consider cases on this in what follows.

\subsubsection{Sub-case where $v(e) \le \frac{\theta}{2} + r$}
First, suppose $v(e) \le \frac{\theta}{2} + r$.
Then the contribution of \textbf{Case 2\ts{-}} is void,
since the inner sum of $I^{\text{6-}}$ contributes only when $n_2 > \frac{\theta}{2} + r$.
We only need to consider
\begin{align*}
  I^{\text{6+}}
  &= \sum_{n_2=0}^{v(e)} \sum_{m=\theta+2r+1}^{v(b)+v(c)+2r}
    q^{n_2} (-q^s)^{2n_2-m+v(c)+r} \\
  &= \sum_{n_2=0}^{v(e)} \sum_{k=2n_2-v(b)-r}^{2n_2-\theta+v(c)-r-1}
    q^{n_2} (-q^s)^k.
\end{align*}
Swapping the summation order so that $k$ is outside,
the sum runs from the lowest value $k = -v(b) - r$
up to the highest value $k = 2v(e) - \theta + v(c) - r - 1$,
subject to $0 \le n_2 \le v(e)$ and
\begin{align*}
  2n_2 - v(b) - r &\le k \le 2n_2 - \theta + v(c) - r - 1 \\
  \iff \left\lceil \frac{k + \theta - v(c) + r + 1}{2} \right\rceil
  &\le n_2 \le \left\lfloor \frac{k + v(b) + r}{2} \right\rfloor.
\end{align*}
Thus,
\[ I^{\text{6+}}
  = \sum_{k = -v(b) - r}^{2v(e) - \theta + v(c) - r -1}
    \sum_{n_2 = \max(0, \left\lceil \frac{k + \theta - v(c) + r + 1}{2} \right\rceil)}
    ^{\min(v(e), \left\lfloor \frac{k + v(b) + r}{2} \right\rfloor)}
    q^{n_2} (-q^s)^k.
\]

\subsubsection{Sub-case where $v(e) > \frac{\theta}{2} + r$}
We start on $I^{\text{6-}}$; note if $n_2 \le \frac{\theta}{2} + r$
then the inner sum of $I^{\text{6-}}$ has empty range anyway.
Consequently, we can simply write
\begin{align*}
  I^{\text{6-}}
  &= q^{\frac{\theta}{2} + r} \sum_{n_2 = \frac{\theta}{2} + r + 1}^{v(e)}
  \sum_{m = \theta + 2r + 1}^{n_2 + \frac{\theta}{2} + r} (-q^s)^{2n_2 - m + v(c) + r}
\end{align*}
which in particular is nonempty.
In that case, simplifying the inner sum gives
\[
  I^{\text{6-}}
  = q^{\frac{\theta}{2} + r} \sum_{n_2 = \frac{\theta}{2} + r + 1}^{v(e)}
  \left(
    (-q^s)^{2n_2 - \theta + v(c) - r - 1}
    + \dots
    + (-q^s)^{n_2 - \frac{\theta}{2} + v(c)}
  \right).
\]
We collect the coefficient of $(-q^s)^k$ for each $k$.
The lowest value of $k$ which appears is $k = v(c) + r + 1$;
the highest one is $k = 2v(e) - \theta + v(c) - r - 1$.
For these $k$,
the coefficient is the number of integers $n_2$ such that
\[ \frac{\theta}{2} + r + 1 \le n_2 \le v(e) \]
and
\begin{align*}
  n_2 - \frac{\theta}{2} + v(c) &\le k \le 2n_2 - \theta - r - 1 + v(c) \\
  \iff \frac{k + \theta - v(c) + r + 1}{2} &\le n_2 \le k + \frac{\theta}{2} - v(c).
\end{align*}
Note we already have $\frac{k + \theta - v(c) + r + 1}{2} \ge \frac{\theta}{2}+r+1$
for $k$ in the desired range.
Hence we have
\begin{align*}
  I^{\text{6-}}
  &=
  q^{\frac{\theta}{2} + r}
  \sum_{k = v(c) + r + 1}^{2v(e) - \theta + v(c) - r - 1}
  \bigg( 1 + \min\left( v(e), k + \frac{\theta}{2} - v(c) \right) \\
    &\hspace{16ex} - \max\left( \frac{\theta}{2} + r + 1,
      \left\lceil \frac{k + \theta - v(c) + r + 1}{2} \right\rceil \right) \bigg) (-q^s)^k \\
  &=
  q^{\frac{\theta}{2} + r}
  \sum_{k = v(c) + r + 1}^{2v(e) - \theta + v(c) - r - 1}
  \left( 1 + \min\left( v(e), k + \frac{\theta}{2} - v(c) \right)
    - \left\lceil \frac{k + \theta - v(c) + r + 1}{2} \right\rceil
  \right) (-q^s)^k.
\end{align*}

The second double sum of $I^{\text{6+}}$ is again nonempty
since $v(e) > \frac{\theta}{2} + r$.
So we compute it
in a similar way to $I^{\text{6-}}$ by putting
\begin{align*}
  &q^{\frac{\theta}{2} + r} \sum_{n_2 = \frac{\theta}{2} + r + 1}^{v(e)}
  \sum_{m = \theta + 2r + 1}^{n_2 - \frac{\theta}{2} + v(b) + v(c) + r}
    (-q^s)^{2n_2 - m + v(c) + r} \\
  &= q^{\frac{\theta}{2} + r} \sum_{n_2 = \frac{\theta}{2} + r + 1}^{v(e)}
  \left(
    (-q^s)^{2n_2 - \theta + v(c) - r - 1}
    + \dots
    + (-q^s)^{n_2 + \frac{\theta}{2} - v(b)}
  \right).
\end{align*}
Again we calculate the coefficient of $(-q^s)^k$.
The values of $k$ run from the lowest value $k = \theta - v(b) + r + 1$
and end at the highest value $k = 2v(e) - \theta + v(c) - r - 1$.
In this range we need $\frac{\theta}{2} + r + 1 \le n_2 \le v(e)$ and
\begin{align*}
  n_2 + \frac{\theta}{2} - v(b) &\le k \le 2n_2 - \theta + v(c) - r - 1 \\
  \iff \frac{k + \theta - v(c) + r + 1}{2} &\le n_2 \le k - \frac{\theta}{2} + v(b).
\end{align*}
The double sum therefore becomes
\begin{align*}
  &\phantom=
  q^{\frac{\theta}{2} + r}
  \sum_{k = \theta - v(b) + r + 1}^{2v(e) - \theta + v(c) - r - 1}
  \bigg( 1 + \min\left( v(e), k - \frac{\theta}{2} + v(b) \right) \\
    &\hspace{16ex} - \max\left( \frac{\theta}{2} + r + 1,
      \left\lceil \frac{k + \theta - v(c) + r + 1}{2} \right\rceil \right) \bigg) (-q^s)^k.
\end{align*}
It is natural to split this sum into $k \le v(c) + r$ and $k > v(c) + r$.
In the former case, we have both $k - \frac{\theta}{2} + v(b) \le v(e)$
and $\frac{\theta}{2} + r + 1 \ge \left\lceil \frac{k + \theta - v(c) + r + 1}{2} \right\rceil$;
in the latter case we have just
$\frac{\theta}{2} + r + 1 \le \left\lceil \frac{k + \theta - v(c) + r + 1}{2} \right\rceil$
instead.
Hence, the double sum simplifies further to
\begin{align*}
  &\phantom= q^{\frac{\theta}{2} + r}
  \sum_{k = \theta - v(b) + r + 1}^{v(c) + r}
  \bigg( 1 + \left( k - \frac{\theta}{2} + v(b) \right)
    - \left( \frac{\theta}{2} + r + 1 \right) \bigg) (-q^s)^k \\
  &+ q^{\frac{\theta}{2} + r}
  \sum_{k = v(c) + r + 1}^{2v(e) - \theta + v(c) - r - 1}
  \bigg( 1 + \min\left( v(e), k - \frac{\theta}{2} + v(b) \right)
    - \left\lceil \frac{k + \theta - v(c) + r + 1}{2} \right\rceil \bigg) (-q^s)^k \\
  &= q^{\frac{\theta}{2} + r}
  \sum_{k = \theta - v(b) + r + 1}^{v(c) + r}
  \left( k - \theta + v(b) - r \right) (-q^s)^k \\
  &+ q^{\frac{\theta}{2} + r}
  \sum_{k = v(c) + r + 1}^{2v(e) - \theta + v(c) - r - 1}
  \bigg( 1 + \min\left( v(e), k - \frac{\theta}{2} + v(b) \right)
    - \left\lceil \frac{k + \theta - v(c) + r + 1}{2} \right\rceil \bigg) (-q^s)^k.
\end{align*}
Meanwhile, the first sum within $I^{\text{6+}}$ can be computed as
\begin{align*}
  \sum_{n_2 = 0}^{\frac{\theta}{2} + r}
  \sum_{m = \theta + 2r + 1}^{v(b) + v(c) + 2r}
    q^{n_2} (-q^s)^{2n_2 - m + v(c) + r}
  &= \sum_{n_2 = 0}^{\frac{\theta}{2} + r} q^{n_2}
    \sum_{m = \theta + 2r + 1}^{v(b) + v(c) + 2r}
      (-q^s)^{2n_2 - m + v(c) + r} \\
  &= \sum_{n_2 = 0}^{\frac{\theta}{2} + r} q^{n_2}
    \sum_{k = 2n_2 - v(b) - r}^{2n_2 - \theta + v(c) - r - 1} (-q^s)^k.
\end{align*}
We now interchange the summation so that $k$ is outside,
running from the lowest value $k = -v(b) - r$
to the highest value $k = v(c) + r - 1$.
From
\[ 2n_2 - v(b) - r \le k \le 2n_2 - \theta + v(c) - r - 1 \]
we require that $0 \le n_2 \le \frac{\theta}{2} + r$ and
\[ \frac{k + \theta - v(c) + r + 1}{2} \le n_2 \le \frac{k + r + v(b)}{2}. \]
In other words, we get
\[
  \sum_{k = - v(b) - r}^{v(c) + r - 1}
  (-q^s)^k
  \sum_{n_2 = \max\left(0, \left\lceil \frac{k + \theta - v(c) + r + 1}{2} \right\rceil \right)}
  ^{\min\left( \frac{\theta}{2} + r, \left\lfloor \frac{v(b) + r + k}{2} \right\rfloor \right) } q^{n_2}.
\]
Hence the total contribution from \textbf{Case 2} can be written as
\begin{align*}
  I^{\text{6+}} + I^{\text{6-}}
  &=
  \sum_{k = - v(b) - r}^{v(c) + r - 1} (-q^s)^k \left(
    q^{\min\left( \frac{\theta}{2} + r, \left\lfloor \frac{v(b) + r + k}{2} \right\rfloor \right)}
    + \dots
    + q^{\max\left(0, \left\lceil \frac{k + \theta - v(c) + r + 1}{2} \right\rceil \right)}
    \right) \\
  &+ q^{\frac{\theta}{2} + r}
  \sum_{k = \theta - v(b) + r + 1}^{v(c) + r}
  \left( k - \theta + v(b) - r \right) (-q^s)^k \\
  &+ q^{\frac{\theta}{2} + r}
  \sum_{k = v(c) + r + 1}^{2v(e) - \theta + v(c) - r - 1}
  \bigg( 2 +
    \min\left( v(e), k - \frac{\theta}{2} + v(b) \right) \\
    &+\hspace{16ex} \min\left( v(e), k + \frac{\theta}{2} - v(c) \right) \\
    &\hspace{16ex} - 2\left\lceil \frac{k + \theta - v(c) + r + 1}{2} \right\rceil \bigg) (-q^s)^k.
\end{align*}
We'd like to further simplify the coefficient of $q^{\frac{\theta}{2}+r}$ as follows.
First, we may as well write
\begin{align*}
  2 - 2\left\lceil \frac{k + \theta - v(c) + r + 1}{2} \right\rceil
  &= 2 - \left( (k + \theta - v(c) + r + 1)
  + \mathbf{1}_{k + \theta + v(c) + r \equiv 1 \bmod 2} \right) \\
  &= \mathbf{1}_{k + \theta + v(c) + r \equiv 0 \bmod 2}
  + v(c) - \theta - k - r.
\end{align*}
Set aside the indicator function
$\mathbf{1}_{k + \theta + v(c) + r \equiv 0 \bmod 2}$ momentarily;
we will merge it in a moment.
To consolidate the minimum's in the third double sum,
note that we have
\[ v(c) + r + 1 \le v(e) + \frac{\theta}{2} - v(b)
< v(e) + v(c) - \frac{\theta}{2} \le 2v(e) - \theta + v(c) - r - 1. \]
Hence, based on the value of $k$, we get the following coefficients:
\begin{itemize}
  \ii If $v(c) + r + 1 \le k \le v(e) + \frac{\theta}{2} - v(b)$, we get
  \begin{align*}
    \left( v(c) - \theta - k - r \right)
    &+ \left( k - \frac{\theta}{2} + v(b) \right)
    + \left( k + \frac{\theta}{2} - v(c) \right)  \\
    &= k - \theta + v(b) - r.
  \end{align*}

  \ii If $v(e) + \frac{\theta}{2} - v(b) \le k \le v(e) + v(c) - \frac{\theta}{2}$, we get
  \begin{align*}
    \left( v(c) - \theta - k - r \right)
    &+ v(e)
    + \left( k + \frac{\theta}{2} - v(c) \right)  \\
    &= v(e) - \frac{\theta}{2} - r.
  \end{align*}

  \ii If $v(e) + v(c) - \frac{\theta}{2} \le k \le 2v(e) - \theta + v(c) - r$, we get
  \begin{align*}
    \left( v(c) - \theta - k - r \right)
    &+ v(e) + v(e) \\
    &= 2v(e) + v(c) - \theta - r.
  \end{align*}
\end{itemize}
Noting the expression in the first bullet also matches the coefficient of $(-q^s)^k$
for $\theta - v(b) + r + 1 \le k \le v(c) + r$, we can now write
\begin{align*}
  I^{\text{6+}} + I^{\text{6-}}
  &=
  \sum_{k = - v(b) - r}^{v(c) + r - 1} (-q^s)^k \left(
    q^{\min\left( \frac{\theta}{2} + r, \left\lfloor \frac{v(b) + r + k}{2} \right\rfloor \right)}
    + \dots
    + q^{\max\left(0, \left\lceil \frac{k + \theta - v(c) + r + 1}{2} \right\rceil \right)}
    \right) \\
  &+ q^{\frac{\theta}{2} + r}
  \sum_{k = \theta - v(b) + r + 1}^{2v(e) - \theta + v(c) - r - 1} \cc_\guv(k)(-q^s)^k \\
  &+ q^{\frac{\theta}{2} + r}
  \sum_{k = v(c) + r + 1}^{2v(e) - \theta + v(c) - r - 1}
  \mathbf{1}_{k + \theta + v(c) + r \equiv 0 \bmod 2} (-q^s)^k
\end{align*}
as the overall contribution from \textbf{Case 2}, where
\[ \cc_\guv(k) \coloneqq \min \left( k - \theta + v(b) - r,
  v(e) - \frac{\theta}{2} - r,
  2v(e) + v(c) - \theta - r \right). \]

\subsection{Proof of \Cref{thm:semi_lie_formula}}
\label{sec:proof_semi_lie_formula}
We now prove \Cref{thm:semi_lie_formula}, which we restate here.

\semilieformula*

For reference, we provide \Cref{fig:semi_lie_sketch} sketching the shapes
of $\nn_\guv$ and $\cc_\guv$, which may be easier to think about.

\begin{figure}
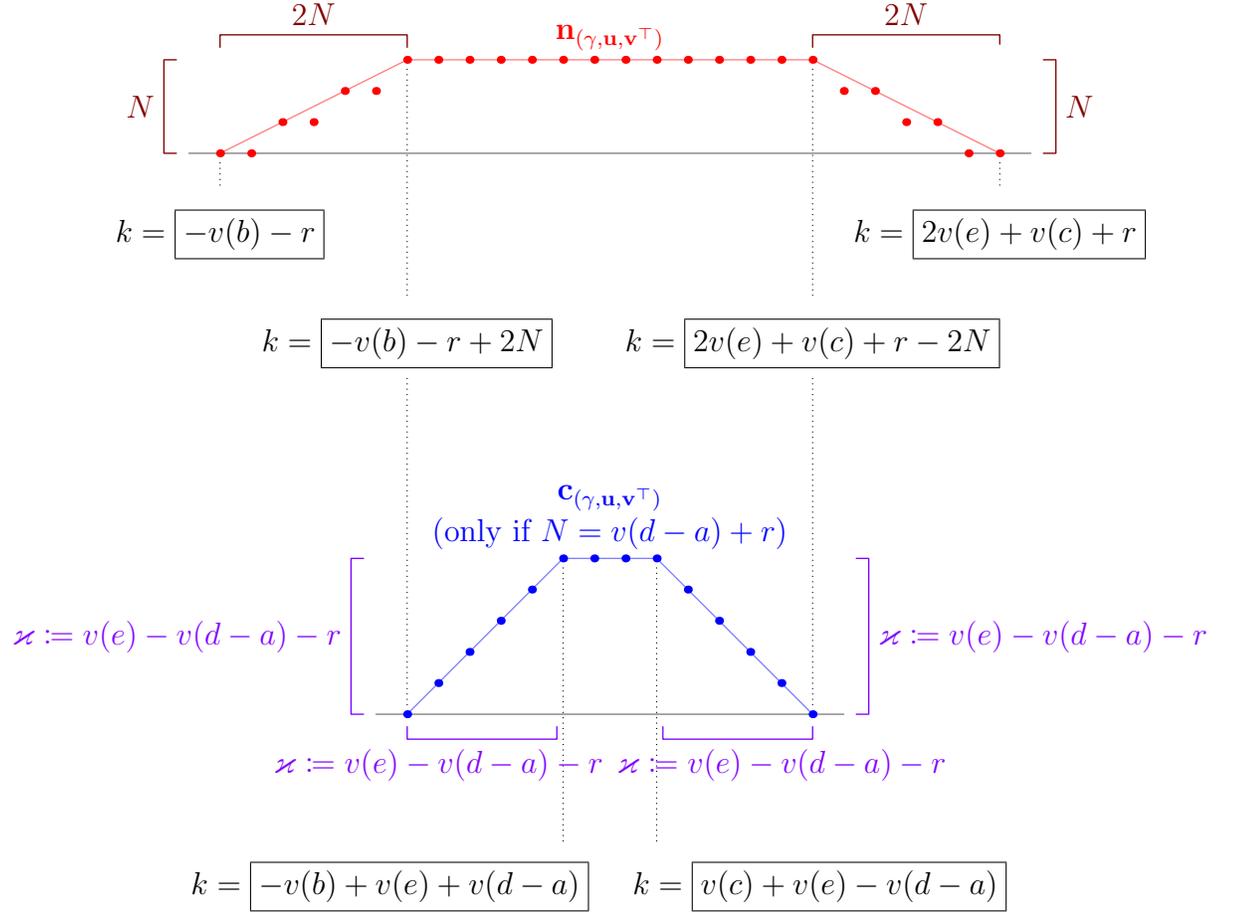

  \centering
  \begin{asy}
    usepackage("amsmath");
    usepackage("amssymb");
    usepackage("mathtools");
    size(16cm);
    draw((0,0)--(6,3)--(19,3)--(25,0), lightred);
    draw((-1,0)--(26,0), grey);
    draw((-1.4,0)--(-1.8,0)--(-1.8,3)--(-1.4,3), brown);
    label("$N$", (-1.8,1.5), dir(180), brown);
    draw((26.4,0)--(26.8,0)--(26.8,3)--(26.4,3), brown);
    label("$N$", (26.8,1.5), dir(0), brown);

    label("$2N$", (3,3.8), dir(90), brown);
    label("$2N$", (22,3.8), dir(90), brown);
    draw((0,3.4)--(0,3.8)--(6,3.8)--(6,3.4), brown);
    draw((19,3.4)--(19,3.8)--(25,3.8)--(25,3.5), brown);

    draw((0,0)--(0,-1.5), dotted, Margins);
    draw((6,3)--(6,-5), dotted, Margins);
    draw((19,3)--(19,-5), dotted, Margins);
    draw((25,0)--(25,-1.5), dotted, Margins);

    for (int i=0; i<=25; ++i) {
      real y = min(floor(i/2), floor((25-i)/2), 3);
      dot((i, y), red);
    }

    label("$k=\boxed{-v(b)-r}$", (0,-1.5), dir(-90));
    label("$k=\boxed{-v(b)-r+2N}$", (6,-5), dir(-90));
    label("$k=\boxed{2v(e)+v(c)+r-2N}$", (19,-5), dir(-90));
    label("$k=\boxed{2v(e)+v(c)+r}$", (25,-1.5), dir(-90));
    label("$\mathbf{n}_{(\gamma, \mathbf{u}, \mathbf{v}^\top)}$", (12.5,3), dir(90), red);

    real h = 18; // offset downwards
    draw((6,-h)--(11,5-h)--(14,5-h)--(19,-h), lightblue);
    draw((5,-h)--(20,-h), grey);
    for (int i=6; i<=19; ++i) {
      real y = min(i-6, 19-i, 5);
      dot((i, y - h), blue);
    }
    label("$\mathbf{c}_{(\gamma, \mathbf{u}, \mathbf{v}^\top)}$", (12.5,5-h), 5*dir(90), blue);
    label("(only if $N = v(d-a) + r$)", (12.5,5-h), dir(90), blue);

    draw((6,-h)--(6,-6.8), dotted, Margins);
    draw((19,-h)--(19,-6.8), dotted, Margins);
    draw((11,5-h)--(11,-4.5-h), dotted, Margins);
    draw((14,5-h)--(14,-4.5-h), dotted, Margins);
    label("$k=\boxed{-v(b)+v(e)+v(d-a)}$", (11,-4.5-h), dir(230));
    label("$k=\boxed{v(c)+v(e)-v(d-a)}$", (14,-4.5-h), dir(310));

    draw((6,-h-0.4)--(6,-h-0.8)--(10.8,-h-0.8)--(10.8,-h-0.4), purple);
    draw((14.2,-h-0.4)--(14.2,-h-0.8)--(19,-h-0.8)--(19,-h-0.4), purple);
    label("$\varkappa \coloneqq v(e)-v(d-a)-r$", (7, -h-0.8), dir(-90), purple);
    label("$\varkappa \coloneqq v(e)-v(d-a)-r$", (18, -h-0.8), dir(-90), purple);
    draw((4.6,-h)--(4.2,-h)--(4.2,5-h)--(4.6,5-h), purple);
    draw((20.4,-h)--(20.8,-h)--(20.8,5-h)--(20.4,5-h), purple);
    label((4.2,2.5-h), "$\varkappa \coloneqq v(e)-v(d-a)-r$", dir(180), purple);
    label((20.8,2.5-h), "$\varkappa \coloneqq v(e)-v(d-a)-r$", dir(0), purple);
  \end{asy}
  \caption{Sketch of the functions in \Cref{thm:semi_lie_formula}.
    The boxed numbers indicate values of $k$.}
  \label{fig:semi_lie_sketch}
\end{figure}

\begin{proof}[Proof of \Cref{thm:semi_lie_formula}]
  First, suppose $\theta = v(b) + v(c) < 2v(d-a)$ is odd.
  Then
  \[ \nn_\guv(k) \coloneqq \min\left( \left\lfloor \tfrac{k + (v(b)+r)}{2} \right\rfloor,
      \left\lfloor \tfrac{(2v(e)+v(c)+r)-k}{2} \right\rfloor,
      v(e), \tfrac{v(b)+v(c)-1}{2} + r \right) \]
  and $\cc_\guv$ terms do not appear.
  Now, in that case, the exponent \eqref{eq:case_5_exponent} can be simplified, because
  \begin{align*}
    \left\lfloor \frac{k+\theta-v(c)+r}{2} \right\rfloor
    &= \left\lfloor \frac{k+v(b)+r}{2} \right\rfloor \\
    v(e) + \left\lfloor \frac{v(c)+r-k}{2} \right\rfloor
    &= \left\lfloor \frac{(2v(e) + v(c) + r) - k}{2} \right\rfloor \\
    \left\lfloor \frac{k+\theta-v(c)+r}{2} \right\rfloor
    + \left\lfloor \frac{v(c)+r-k}{2} \right\rfloor
    &= \frac{k+\theta-v(c)+r}{2} + \frac{v(c)+r-k}{2} - \half \\
    &= \frac{\theta-1}{2} + r \\
    &= \frac{v(b)+v(c)-1}{2} + r < v(d-a) + r.
  \end{align*}
  Hence, $\nn_\guv$ coincides with the exponent in \eqref{eq:case_5_exponent}.
  So the result is true in this case.

  Now assume instead $\theta = 2v(d-a) < v(b) + v(c)$ is even.
  Notice that
  \[ \left\lfloor \frac{k+\theta-v(c)+r}{2} \right\rfloor
    + \left\lfloor \frac{v(c)+r-k}{2} \right\rfloor
    = \frac{\theta}{2} + r - \mathbf{1}_{k + v(c) + r \equiv 0 \bmod 2}. \]

  First assume that $v(e) \le \frac{\theta}{2} + r$ and
  consider \eqref{eq:case_5_exponent}.
  For the range of values of $k$ in $I^{\text{6+}}$, that is
  \[ -v(b) - r \le k \le 2v(e)-\theta+v(c)-r-1 \]
  we have the first term of \eqref{eq:case_5_exponent} is smallest, as
  \begin{align*}
    \left\lfloor \frac{k + \theta - v(c) + r}{2} \right\rfloor &< v(e) \\
    &\le v(e) + \frac{v(b)+v(c)-1}{2}
      \le v(e) + \left\lfloor \frac{v(c)+r-k}{2} \right\rfloor \\
    \left\lfloor \frac{k + \theta - v(c) + r}{2} \right\rfloor &\le v(e)-1
      \le \frac{\theta}{2} + r -1.
  \end{align*}
  So the contributions from \textbf{Case 1} and \textbf{Case 2}
  fit together to give
  \begin{align*}
    \sum_{j=0}^{\left\lfloor \frac{k+\theta-v(c)+r}{2} \right\rfloor} q^j
    &+
    \left(
    q^{\min\left( v(e), \left\lfloor \frac{v(b) + r + k}{2} \right\rfloor \right)}
    + \dots
    + q^{\max\left(0, \left\lceil \frac{k + \theta - v(c) + r + 1}{2} \right\rceil \right)}
    \right) \\
    &=
    q^{\min(v(e), \left\lfloor \frac{v(b) + r + k}{2} \right\rfloor)} + \dots + q^0
  \end{align*}
  which thus matches the formula for $\nn_\guv(k)$.

  Now suppose instead $v(e) > \frac{\theta}{2} + r$.
  First, a similar analysis gives that the first part of $I^{\text{6+}}$ fits
  together with $I^{\text{5}}$ again.
  Indeed if
  \[ -v(b) - r \le k \le v(c) + r - 1 \]
  then in \eqref{eq:case_5_exponent} we get the first exponent again, and hence
  we again get the fit
  \begin{align*}
    \sum_{j=0}^{\left\lfloor \frac{k+\theta-v(c)+r}{2} \right\rfloor} q^j
    &+
    \left(
    q^{\min\left( \frac{\theta}{2} + r, \left\lfloor \frac{v(b) + r + k}{2} \right\rfloor \right)}
    + \dots
    + q^{\max\left(0, \left\lceil \frac{k + \theta - v(c) + r + 1}{2} \right\rceil \right)}
    \right) \\
    &=
    q^{\min(\frac{\theta}{2}+r, \left\lfloor \frac{v(b) + r + k}{2} \right\rfloor)} + \dots + q^0
  \end{align*}
  which matches the claimed formula for $\nn_\guv$ in this range.

  The remaining contribution from \textbf{Case 2\ts{+} and Case 2\ts{-}} is
  \begin{align*}
  &\phantom+ q^{\frac{\theta}{2} + r}
  \sum_{k = \theta - v(b) + r + 1}^{2v(e) - \theta + v(c) - r - 1} \cc_\guv(k)(-q^s)^k \\
  &+ q^{\frac{\theta}{2} + r}
  \sum_{k = v(c) + r + 1}^{2v(e) - \theta + v(c) - r - 1}
  \mathbf{1}_{k + \theta + v(c) + r \equiv 0 \bmod 2} (-q^s)^k.
  \end{align*}

  The first sum matches the claimed coefficient $\cc_\guv$
  (except the summation in the theorem statement includes
  endpoints at $k = \theta - v(b) + r$
  and $k = 2v(e) - \theta + v(c) - r$,
  but $\cc_\guv(k) = 0$ at these two endpoints,
  so there is no change).

  Meanwhile the second sum accounts for the discrepancy between
  the final term of \eqref{eq:case_5_exponent} and the formula for $\nn_\guv$.
  That is, the range of $k$ for which \eqref{eq:case_5_exponent}
  achieves the last minimum is exactly
  \[ v(c) + r + 1 \le k \le 2v(e) - \theta + v(c) - r - 1 \]
  and only in those cases does \eqref{eq:case_5_exponent}
  differs from $\nn_\guv$ by exactly
  $\mathbf{1}_{k + \theta + v(c) + r \equiv 0 \bmod 2}$.
  This final step shows the claimed formulas coincide.
\end{proof}

\begin{example}
  [The special case $v(e)=0$]
  When $v(e) = 0$ the expression is particularly simple.
  The assumption $v(d-a) \ge v(e)-r$ is automatically true, and
  $\nn_{\guv}$ is identically zero, so
  \[ \Orb(\guv, \mathbf{1}_{K'_{S, \le r}} \otimes \oneV, s)
    = \sum_{k=-(v(b)+r)}^{v(c)+r} (-q^s)^k. \]
\end{example}

\begin{example}
  Suppose $r = 14$, $v(b) = -5$, $v(c) = 100$, $v(e) = 3$.
  We have $v(d-a) \ge 0 > -11 = v(e) - r$.
  Hence the above formula reads
  \begin{align*}
    \Orb(\guv, \mathbf{1}_{K'_{S, \le 14}} \otimes \oneV, s)
    &= -q^{-9s} \\
    &+ q^{-8s} \\
    &- (q+1) \cdot q^{-7s} \\
    &+ (q+1) \cdot q^{-6s} \\
    &- (q^2+q+1) \cdot q^{-5s} \\
    &+ (q^2+q+1) \cdot q^{-4s} \\
    &- (q^3+q^2+q+1) \cdot q^{-3s} \\
    &+ (q^3+q^2+q+1) \cdot q^{-2s} \\
    &- (q^3+q^2+q+1) \cdot q^{-s} \\
    &+ (q^3+q^2+q+1) \cdot q^{0} \\
    &- (q^3+q^2+q+1) \cdot q^{s} \\
    &+ (q^3+q^2+q+1) \cdot q^{2s} \\
    &\vdotswithin= \\
    &- (q^3+q^2+q+1) \cdot q^{111s} \\
    &+ (q^3+q^2+q+1) \cdot q^{112s} \\
    &- (q^3+q^2+q+1) \cdot q^{113s} \\
    &+ (q^3+q^2+q+1) \cdot q^{114s} \\
    &- (q^2+q+1) \cdot q^{115s} \\
    &+ (q^2+q+1) \cdot q^{116s} \\
    &- (q+1) \cdot q^{117s} \\
    &+ (q+1) \cdot q^{118s} \\
    &- q^{119s} \\
    &+ q^{120s}.
  \end{align*}
\end{example}
\begin{example}
  Suppose $r = 2$, $v(b) = -5$, $v(c) = 100$, $v(e) = 20$, $v(d-a) = 1$.
  Then we have
  \begin{align*}
    \Orb(\guv, \mathbf{1}_{K'_{S, \le 2}} \otimes \oneV, s)
    &= -q^{3s} \\
    &+ q^{4s} \\
    &- (q+1) \cdot q^{5s} \\
    &+ (q+1) \cdot q^{6s} \\
    &- (q^2+q+1) \cdot q^{7s} \\
    &+ (q^2+q+1) \cdot q^{8s} \\
    &- (q^3+q^2+q+1) \cdot q^{9s} \\
    &+ (2q^3+q^2+q+1) \cdot q^{10s} \\
    &- (3q^3+q^2+q+1) \cdot q^{9s} \\
    &+ (4q^3+q^2+q+1) \cdot q^{10s} \\
    &- (5q^3+q^2+q+1) \cdot q^{9s} \\
    &+ (6q^3+q^2+q+1) \cdot q^{10s} \\
    &\vdotswithin= \\
    &- (17q^3+q^2+q+1) \cdot q^{25s} \\
    &+ (18q^3+q^2+q+1) \cdot q^{26s} \\
    &- (18q^3+q^2+q+1) \cdot q^{27s} \\
    &+ (18q^3+q^2+q+1) \cdot q^{28s} \\
    &\vdotswithin= \\
    &- (18q^3+q^2+q+1) \cdot q^{117s} \\
    &+ (18q^3+q^2+q+1) \cdot q^{118s} \\
    &- (18q^3+q^2+q+1) \cdot q^{119s} \\
    &+ (17q^3+q^2+q+1) \cdot q^{120s} \\
    &- (16q^3+q^2+q+1) \cdot q^{121s} \\
    &+ (15q^3+q^2+q+1) \cdot q^{122s} \\
    &\vdotswithin= \\
    &+ (3q^3+q^2+q+1) \cdot q^{134s} \\
    &- (2q^3+q^2+q+1) \cdot q^{135s} \\
    &+ (q^3+q^2+q+1) \cdot q^{136s} \\
    &- (q^2+q+1) \cdot q^{137s} \\
    &+ (q^2+q+1) \cdot q^{138s} \\
    &- (q+1) \cdot q^{139s} \\
    &+ (q+1) \cdot q^{140s} \\
    &- q^{141s} \\
    &+ q^{142s}.
  \end{align*}
\end{example}

\subsection{Proof of \Cref{cor:semi_lie_derivative_single}}
With \Cref{thm:semi_lie_formula} established, we aim to calculate the derivative of
\[ \Orb(\guv, \mathbf{1}_{K', S, \le r} \otimes \oneV, s) \]
now at $s = 0$.
Our goal is to prove \Cref{cor:semi_lie_derivative_single}:
\semiliederiv*

\begin{proof}[Proof of \Cref{cor:semi_lie_derivative_single}]
For this calculation it will be more convenient to reformat \Cref{thm:semi_lie_formula}
as a sum over $q^j$ rather than $(-1)^k (q^s)^k$.
To that, continuing to write
\[ N \coloneqq \min \left(
    v(e), \frac{v(b)+v(c)-1}{2} + r,
    v(d-a) + r \right) \]
consider any index $0 \le j \le N$.
Then
\[ \nn_\guv(k) \ge j \iff 2j - v(b) - r \le k \le 2v(e)+v(c)+r-2j. \]
In other words, the first part of \Cref{cor:semi_lie_combo} can be rewritten as
\[ \sum_{k = -(v(b)+r)}^{2v(e)+v(c)+r} (-1)^k
  \left( 1 + \dots + q^{\nn_\guv(k)} \right) (q^s)^k
  = q^j \sum_{j=0}^N \left( \sum_{k=2j-v(b)-r}^{2v(e)+v(c)+r-2j} (-q^s)^k \right). \]
If we take the derivative at $s = 0$ with respect to $k$, we get
\[ \log q \sum_{j=0}^N \left( q^j \sum_{k=2j-v(b)-r}^{2v(e)+v(c)+r-2j} (-1)^k k \right). \]
The number of terms inside the summation is
$2v(e)+v(b)+v(c)+2r-4j+1$, an even number.
Each consecutive pair differs by $(-1)^{v(c)+r}$.
Hence we get
\[ (-1)^{v(c)+r} \log q \sum_{j=0}^N \left( q^j
  \cdot \left( \frac{2v(e)+v(b)+v(c)+1}{2} + r - 2j \right) \right). \]
Now in the case that $v(e) > v(d-a) + r$ and $2v(d-a) < v(b) + v(c)$,
we have to handle the additional contribution obtained when we differentiate
\begin{equation}
  q^{v(d-a)+r}
  \sum_{k = 2v(d-a)-v(b)+r}^{2v(e)+v(c)-2v(d-a)-r} (-1)^k \cc_\guv(k) (-q^s)^k
  \label{eq:top_derivative}
\end{equation}
at $s = 0$.
For brevity, we will define
\[ \varkappa \coloneqq v(e) - v(d-a) - r \ge 0. \]
(We allow the degenerate case $\kappa = 0$ for convenience,
in which case $\cc_\guv$ is still identically zero.)

We return to differentiating \eqref{eq:top_derivative}.
We will use the following extremely easy lemma:
\begin{lemma}
  [Extremely easy]
  \label{lem:extremely_easy_lemma}
  If $a_0 \le a_1$ are integers then
  \begin{align*}
    \sum_{k=a_0}^{a_1} (k-a_0) \cdot k \cdot (-1)^k
    &= (-1)^{a_1} \cdot \frac{a_1(a_1-a_0+1)}{2} - \frac{(-1)^{a_0} + (-1)^{a_1}}{4} \cdot a_0 \\
    \sum_{k=a_0}^{a_1} (a_1-k) \cdot k \cdot (-1)^k
    &= (-1)^{a_0} \cdot \frac{a_0(a_1-a_0+1)}{2} - \frac{(-1)^{a_0} + (-1)^{a_1}}{4} \cdot a_1.
  \end{align*}
\end{lemma}
\begin{proof}
  This follows trivially by induction on $a_1$.
  (Alternatively, use a symbolic engine like WolframAlpha;
  see
  \href{https://www.wolframalpha.com/input?i=sum+\%28k-a\%29*k*\%28-1\%29\%5Ek+from+k\%3Da+to+b}{here}
  for the first sum and
  \href{https://www.wolframalpha.com/input?i=sum+\%28b-k\%29*k*\%28-1\%29\%5Ek+from+k\%3Da+to+b}{here}
  for the second sum.)
\end{proof}
For the range of $k$ given, we split it into three parts.
\begin{itemize}
  \ii For $-v(b) + r + 2v(d-a) \le k < -v(b) + v(e) + v(d-a)$,
  apply the first part of the lemma to conclude that the contribution to the derivative is
  \begin{align*}
    \log q \cdot \Big( & (-1)^{v(b)+v(e)+v(d-a)-1} \cdot
      \frac{(-v(b)+v(e)+v(d-a)-1) \cdot \varkappa}{2} \\
      &
      - \frac{(-1)^{v(b)+v(e)+v(d-a)+1} + (-1)^{v(b)+r}}{4} \cdot (-v(b) + r + 2v(d-a))
    \Big).
  \end{align*}

  \ii For $v(c)+v(e)-v(d-a) < k \le 2v(e)+v(c)-2v(d-a)-r$
  apply the second part of the lemma to conclude that the contribution to the derivative is
  \begin{align*}
    \log q \cdot \Big( & (-1)^{v(c)+v(e)+v(d-a)+1} \cdot
      \frac{(v(c)+v(e)-v(d-a)+1) \cdot \varkappa}{2} \\
      &
      - \frac{(-1)^{v(c)+v(e)+v(d-a)+1} + (-1)^{v(c)+r}}{4} \cdot (2v(e)+v(c)-r-2v(d-a))
    \Big).
  \end{align*}

  \ii For the region $-v(b)+v(e)+v(d-a) \le k \le v(c)+v(e)-v(d-a)$,
  we have $\cc_{\guv}(k) = \varkappa$
  and the values of $k$ form $\left\lceil \frac{v(c)+v(b)-2v(d-a)}{2} \right\rceil$ consecutive pairs.
  So the contribution to the derivative here is exactly
  \[ \log q \cdot (-1)^{v(c)+v(e)+v(d-a)} \varkappa \cdot \frac{v(b)+v(c)-2v(d-a)+1}{2}. \]
\end{itemize}
If we sum all three,
we get a contribution of $\varkappa \log q \cdot  (-1)^{v(c)+v(e)+v(d-a)}$ times
\begin{align*}
  \frac{v(b)+v(c)-2v(d-a)+1}{2} &- \frac{v(c)+v(e)-v(d-a)+1}{2} \\
  &+ \frac{-v(b)+v(e)+v(d-a)-1}{2} = -\half
\end{align*}
If $\varkappa$ is even, then that's all she wrote; we simply get
\[ (-1)^{v(c)+r+1} \cdot \frac{\varkappa}{2} \cdot q^{v(d-a)+r} \log q. \]
On the other hand, if $\varkappa$ is odd we get instead
\begin{align*}
  &\phantom+ (-1)^{v(c)+r} \frac{\varkappa}{2} q^{v(d-a)+r} \log q \\
  &- \half q^{v(d-a)+r} \log q \cdot (-1)^{v(c)+v(e)+v(d-a)} \cdot
    \Big( \left( 2v(e)+v(c)-r-2v(d-a)\right) \\
    &\hspace{20ex} - \left( -v(b)+r+2v(d-a) \right) \Big) \\
  &= (-1)^{v(c)+r} q^{v(d-a)+r} \log q \cdot \left[
    -\left( v(e)+\frac{v(b)+v(c)}{2}-2v(d-a)-r \right) + \frac{\varkappa}{2} \right].
\end{align*}
These formulas match \Cref{cor:semi_lie_derivative_single}, proving it.
(Note that we changed $\varkappa > 0$ to $\varkappa \ge 0$,
which makes no change since then the contribution is zero anyway.)
\end{proof}

\subsection{Proof of \Cref{cor:semi_lie_combo}}
From \Cref{cor:semi_lie_derivative_single} we can now prove \Cref{cor:semi_lie_combo}
\semiliecombo*
\begin{proof}[Proof of \Cref{cor:semi_lie_combo}]
Observe that when we add the right-hand side of \Cref{cor:semi_lie_derivative_single}
to the same right-hand side with $r$ replaced by $r-1$, almost all the terms cancel.
Indeed, the main sum for $0 \le j \le N-1$ line up:
\begin{align*}
  &(-1)^{r+v(c)} \sum_{j=0}^{N-1} q^j \cdot \left( \frac{2v(e)+v(b)+v(c)+1}{2} + r - 2j \right) \\
  &\phantom+
  (-1)^{(r-1)+v(c)} \sum_{j=0}^{N-1} q^j \cdot \left( \frac{2v(e)+v(b)+v(c)+1}{2} + (r-1) - 2j \right) \\
  &= q^{N-1} + \dots + q^0.
\end{align*}
So we consider three cases:
\begin{itemize}
\ii In the case where $v(d-a)+r > v(e)$
and $\frac{v(b)+v(c)-1}{2} + r > v(e)$
then the value of $N = v(e)$ does not change when we decrease $r$ by one.
Hence the term for $j = N$ cancels in both sums and we get exactly
we get \[ (-1)^{r+v(c)} \log q (1 + q + \dots + q^N). \]

\ii Next suppose $v(e) \ge \frac{v(b)+v(c)-1}{2} + r$
and also $2v(d-a) > v(b) + v(c)$ (so the extra terms involving $\varkappa$ are absent).
In that case $N = \frac{v(b) + v(c) - 1}{2} + r$.
Then we are left with $(-1)^{r+v(c)} \log q$ times
\begin{align*}
  &\phantom= \sum_{j=0}^{N-1} (q^j) + \left( \frac{2v(e)+v(b)+v(c)+1}{2} + r - (v(b)+v(c)-1+2r) \right) q^N \\
  &= \sum_{j=0}^{N-1} (q^j) + \left( \frac{2v(e)-v(b)-v(c)-1}{2} - r \right) q^N \\
  &= \sum_{j=0}^{N} (q^j) + \left( v(e) - r - \frac{v(b)+v(c)-1}{2} \right) q^N.
\end{align*}
This matches \Cref{thm:semi_lie_formula}.

\ii Finally, suppose $\varkappa \ge 0$ and $v(b) + v(c) > 2v(d-a)$
so that $N = v(d-a) + r$.
Then we are left with $(-1)^{r+v(c)} \log q$ times
\begin{align*}
  & \sum_{j=0}^{N-1} (q^j) + \left( \frac{2v(e)+v(b)+v(c)+1}{2} + r - 2(v(d-a)+r) \right) q^N \\
  & + q^{N} \cdot
  \begin{cases}
    -\frac{\varkappa}{2} & \text{if }\varkappa \equiv 0 \pmod 2 \\
    \frac{\varkappa}{2} - \left( v(e)+\frac{v(b)+v(c)}{2}-2v(d-a)-r \right)
    & \text{if }\varkappa \equiv 1 \pmod 2 \\
  \end{cases} \\
  & + q^{N-1} \cdot
  \begin{cases}
    -\frac{\varkappa+1}{2} & \text{if }\varkappa+1 \equiv 0 \pmod 2 \\
    \frac{\varkappa+1}{2} - \left( v(e)+\frac{v(b)+v(c)}{2}-2v(d-a)-(r-1) \right)
    & \text{if }\varkappa+1 \equiv 1 \pmod 2
  \end{cases} \\
  &= \sum_{j=0}^{N-1} (q^j)
   + q^{N} \cdot
  \begin{cases}
    \left( \frac{2v(e)+v(b)+v(c)+1}{2} + r - 2(v(d-a)+r) \right) -\frac{\varkappa}{2} & \text{if }\varkappa \equiv 0 \pmod 2 \\
    \half + \frac{\varkappa}{2} & \text{if }\varkappa \equiv 1 \pmod 2
  \end{cases} \\
  & - q^{N-1} \cdot
  \begin{cases}
    \frac{\varkappa-1}{2} - \left( v(e)+\frac{v(b)+v(c)}{2}-2v(d-a)-(r-1) \right)
    & \text{if }\varkappa \equiv 0 \pmod 2 \\
    -\frac{\varkappa+1}{2} & \text{if }\varkappa \equiv 1 \pmod 2
  \end{cases} \\
  &= \sum_{j=0}^{N} (q^j)
   + q^{N} \cdot
  \begin{cases}
    \left( \frac{2v(e)+v(b)+v(c)-1}{2} + r - 2(v(d-a)+r) \right) -\frac{\varkappa}{2} & \text{if }\varkappa \equiv 0 \pmod 2 \\
    \frac{\varkappa-1}{2} & \text{if }\varkappa \equiv 1 \pmod 2
  \end{cases} \\
  & - q^{N-1} \cdot
  \begin{cases}
    \frac{\varkappa-1}{2} - \left( v(e)+\frac{v(b)+v(c)}{2}-2v(d-a)-(r-1) \right)
    & \text{if }\varkappa \equiv 0 \pmod 2 \\
    -\frac{\varkappa+1}{2} & \text{if }\varkappa \equiv 1 \pmod 2
  \end{cases} \\
  &= \sum_{j=0}^{N} (q^j)
   + q^{N} \cdot
  \begin{cases}
    \left( \frac{v(e)+v(b)+v(c)-1}{2} - \frac32v(d-a) - \frac12r \right) & \text{if }\varkappa \equiv 0 \pmod 2 \\
    \frac{\varkappa+1}{2} & \text{if }\varkappa \equiv 1 \pmod 2
  \end{cases} \\
  & - q^{N-1} \cdot
  \begin{cases}
    -\frac{1}{2} - \left( \frac12v(e)+\frac{v(b)+v(c)}{2}-\frac32v(d-a)-\frac12r \right)
    & \text{if }\varkappa \equiv 0 \pmod 2 \\
    -\frac{\varkappa+1}{2} & \text{if }\varkappa \equiv 1 \pmod 2
  \end{cases} \\
  &= \sum_{j=0}^{N} (q^j)
   + q^{N} \cdot
  \begin{cases}
    \left( \frac{\varkappa+v(b)+v(c)-1-2v(d-a)}{2}  \right) & \text{if }\varkappa \equiv 0 \pmod 2 \\
    \frac{\varkappa-1}{2} & \text{if }\varkappa \equiv 1 \pmod 2
  \end{cases} \\
  & + q^{N-1} \cdot
  \begin{cases}
    \frac{\kappa+v(b)+v(c)+1-2v(d-a)}{2}
    & \text{if }\varkappa \equiv 0 \pmod 2 \\
    -\frac{\varkappa+1}{2} & \text{if }\varkappa \equiv 1 \pmod 2.
  \end{cases}
\end{align*}
This matches \Cref{thm:semi_lie_formula}, and the proof is complete. \qedhere
\end{itemize}
\end{proof}

\section{Large kernel and large image}
\label{ch:ker}

In this chapter we use \Cref{cor:semi_lie_combo} to prove
both \Cref{thm:semi_lie_ker_trivial} and \Cref{thm:semi_lie_ker_huge}
for the orbital integral on $S_2(F) \times V'_2(F)$,
We also comment on an analogous result for $S_3(F)$,
although we do not provide all the details.

We assume $\guv$ is as in \Cref{lem:semi_lie_params} throughout this chapter.

\subsection{In the semi-Lie case, the kernel is trivial if we allow $v(e)$ to vary, for every fixed choice of $\gamma$}
We prove \Cref{thm:semi_lie_ker_trivial} in this section.
We treat $\gamma$ as fixed, satisfying the requirements of \Cref{lem:semi_lie_params},
and we let
\[ \theta \coloneqq \min\left( v(b)+v(c), 2v(d-a) \right) \ge 0 \]
as we did in \Cref{ch:orbitalFJ1}.

\begin{lemma}
  [The matrix of $\partial\Orb$'s has full rank]
  \label{lem:semi_lie_ker_full_rank}
  Fix $\gamma$. Let $N \ge 0$ be a nonnegative integer.
  We define an $(N+ \left\lfloor \frac{\theta}{2} \right\rfloor + 2) \times (N+1)$ matrix $M$ as follows:
  for $0 \le i \le N+ \left\lfloor \frac{\theta}{2} \right\rfloor + 1$ and $0 \le r \le N$,
  the $i$\ts{th} row and $r$\ts{th} column takes the value
  \[
    M_{i,r} \coloneqq \frac{(-1)^r}{\log q} \partial \Orb\left(
      \left(\gamma, \begin{pmatrix} 0 \\ 1 \end{pmatrix}, \begin{pmatrix} 0 & \varpi^i \end{pmatrix} \right),
      \mathbf{1}_{K', \le r} \right).
  \]
  Then $M$ has full rank.
\end{lemma}
The basic strategy of the proof will be to perform some sequences of row operations.
Specifically, we introduce the following definition.
\begin{itemize}
  \ii For each $i = N + \left\lfloor \frac{\theta}{2} \right\rfloor,\dots,0$ in that order,
  subtract the $i$\ts{th} row of $M$
  from the $(i+1)$\ts{th} row of $M$.
  Denote the new matrix as $M'$.

  \ii For each $i = N + \left\lfloor \frac{\theta}{2} \right\rfloor - 1,\dots,0$ in that order,
  subtract the $i$\ts{th} row of $M'$
  from the $(i+2)$\ts{nd} row of $M'$.
  Denote the new matrix as $M''$.
\end{itemize}
Then the basic premise is to show that $M''$ has an upper triangular submatrix.
This is easier to see with some illustrations, which we give below.

\begin{example}
  For example, for $N = 4$ and $v(b)+v(c)=1$, $v(d-a) = 0$ (hence $\theta = 0$), we have
  \[
    M = \begin{pmatrix}
    1 & 2 & 3 & 4 & 5 \\
    1 & q + 3 & 2q + 4 & 3q + 5 & 4q + 6 \\
    2 & q + 4 & q^{2} + 3q + 5 & 2q^{2} + 4q + 6 & 3q^{2} + 5q + 7 \\
    2 & 2q + 5 & q^{2} + 4q + 6 & q^{3} + 3q^{2} + 5q + 7 & 2q^{3} + 4q^{2} + 6q + 8 \\
    3 & 2q + 6 & 2q^{2} + 5q + 7 & q^{3} + 4q^{2} + 6q + 8 & q^{4} + 3q^{3} + 5q^{2} + 7q + 9 \\
    3 & 3q + 7 & 2q^{2} + 6q + 8 & 2q^{3} + 5q^{2} + 7q + 9 & q^{4} + 4q^{3} + 6q^{2} + 8q + 10
    \end{pmatrix}
  \]
  hence
  \[ M'= \begin{pmatrix}
    1 & 2 & 3 & 4 & 5 \\
    0 & q + 1 & 2q + 1 & 3q + 1 & 4q + 1 \\
    1 & 1 & q^{2} + q + 1 & 2q^{2} + q + 1 & 3q^{2} + q + 1 \\
    0 & q + 1 & q + 1 & q^{3} + q^{2} + q + 1 & 2q^{3} + q^{2} + q + 1 \\
    1 & 1 & q^{2} + q + 1 & q^{2} + q + 1 & q^{4} + q^{3} + q^{2} + q + 1 \\
    0 & q + 1 & q + 1 & q^{3} + q^{2} + q + 1 & q^{3} + q^{2} + q + 1
    \end{pmatrix}
  \]
  and finally
  \[ M'' = \begin{pmatrix}
    1 & 2 & 3 & 4 & 5 \\
    0 & q + 1 & 2q + 1 & 3q + 1 & 4q + 1 \\
    0 & -1 & q^{2} + q - 2 & 2q^{2} + q - 3 & 3q^{2} + q - 4 \\
    0 & 0 & -q & q^{3} + q^{2} - 2q & 2q^{3} + q^{2} - 3q \\
    0 & 0 & 0 & -q^{2} & q^{4} + q^{3} - 2q^{2} \\
    0 & 0 & 0 & 0 & -q^{3}
    \end{pmatrix}. \]
\end{example}

\begin{example}
  For example, for $N = 4$, $v(b)+v(c)=17$, $v(d-a) = 2$ (hence $\theta=4$), we have
  \[
    M = \begin{pmatrix}
    9 & 10 & 11 & \dots \\
    8q + 10 & 9q + 11 & 10q + 12 & \dots \\
    7q^{2} + 9q + 11 & 8q^{2} + 10q + 12 & 9q^{2} + 11q + 13 & \dots \\
    q^{2} + 10q + 12 & 7q^{3} + 9q^{2} + 11q + 13 & 8q^{3} + 10q^{2} + 12q + 14 & \dots \\
    8q^{2} + 11q + 13 & q^{3} + 10q^{2} + 12q + 14 & 7q^{4} + 9q^{3} + 11q^{2} + 13q + 15 & \dots \\
    2q^{2} + 12q + 14 & 8q^{3} + 11q^{2} + 13q + 15 & q^{4} + 10q^{3} + 12q^{2} + 14q + 16 & \dots \\
    9q^{2} + 13q + 15 & 2q^{3} + 12q^{2} + 14q + 16 & 8q^{4} + 11q^{3} + 13q^{2} + 15q + 17 & \dots \\
    3q^{2} + 14q + 16 & 9q^{3} + 13q^{2} + 15q + 17 & 2q^{4} + 12q^{3} + 14q^{2} + 16q + 18 & \dots
    \end{pmatrix}
  \]
  hence
  \[ M' = \begin{pmatrix}
    9 & 10 & 11 & 12 & \dots \\
    8q + 1 & 9q + 1 & 10q + 1 & 11q + 1 & \dots \\
    7q^{2} + q + 1 & 8q^{2} + q + 1 & 9q^{2} + q + 1 & 10q^{2} + q + 1 & \dots \\
    -6q^{2} + q + 1 & 7q^{3} + q^{2} + q + 1 & 8q^{3} + q^{2} + q + 1 & 9q^{3} + q^{2} + q + 1 & \dots \\
    7q^{2} + q + 1 & -6q^{3} + q^{2} + q + 1 & 7q^{4} + q^{3} + q^{2} + q + 1 & 8q^{4} + \dots + 1 & \dots \\
    -6q^{2} + q + 1 & 7q^{3} + q^{2} + q + 1 & -6q^{4} + q^{3} + q^{2} + q + 1 & 7q^{5} + \dots + 1 & \dots \\
    7q^{2} + q + 1 & -6q^{3} + q^{2} + q + 1 & 7q^{4} + q^{3} + q^{2} + q + 1 & -6q^{5} + \dots + 1 & \dots \\
    -6q^{2} + q + 1 & 7q^{3} + q^{2} + q + 1 & -6q^{4} + q^{3} + q^{2} + q + 1 & 7q^{5} + \dots + 1 & \dots
    \end{pmatrix}
  \]
  and finally
  \[
    M'' = \begin{pmatrix}
    9 & 10 & 11 & 12 & 13 \\
    8q + 1 & 9q + 1 & 10q + 1 & 11q + 1 & 12q + 1 \\
    7q^{2} + q - 8 & 8q^{2} + q - 9 & 9q^{2} + q - 10 & 10q^{2} + q - 11 & 11q^{2} + q - 12 \\
    -6q^{2} - 7q & 7q^{3} + q^{2} - 8q & 8q^{3} + q^{2} - 9q & 9q^{3} + q^{2} - 10q & 10q^{3} + q^{2} - 11q \\
    0 & -6q^{3} - 7q^{2} & 7q^{4} + q^{3} - 8q^{2} & 8q^{4} + q^{3} - 9q^{2} & 9q^{4} + q^{3} - 10q^{2} \\
    0 & 0 & -6q^{4} - 7q^{3} & 7q^{5} + q^{4} - 8q^{3} & 8q^{5} + q^{4} - 9q^{3} \\
    0 & 0 & 0 & -6q^{5} - 7q^{4} & 7q^{6} + q^{5} - 8q^{4} \\
    0 & 0 & 0 & 0 & -6q^{6} - 7q^{5}
    \end{pmatrix}.
  \]
\end{example}
\begin{example}
  For example, for $N = 4$, $v(b)+v(c)=5$, $v(d-a) = 8$ (hence $\theta=5$), we have
  \[ M = \begin{pmatrix}
    3 & 4 & 5 & \dots \\
    2q + 4 & 3q + 5 & 4q + 6 & \dots \\
    q^{2} + 3q + 5 & 2q^{2} + 4q + 6 & 3q^{2} + 5q + 7 & \dots \\
    2q^{2} + 4q + 6 & q^{3} + 3q^{2} + 5q + 7 & 2q^{3} + 4q^{2} + 6q + 8 & \dots \\
    3q^{2} + 5q + 7 & 2q^{3} + 4q^{2} + 6q + 8 & q^{4} + 3q^{3} + 5q^{2} + 7q + 9 & \dots \\
    4q^{2} + 6q + 8 & 3q^{3} + 5q^{2} + 7q + 9 & 2q^{4} + 4q^{3} + 6q^{2} + 8q + 10 & \dots \\
    5q^{2} + 7q + 9 & 4q^{3} + 6q^{2} + 8q + 10 & 3q^{4} + 5q^{3} + 7q^{2} + 9q + 11 & \dots \\
    6q^{2} + 8q + 10 & 5q^{3} + 7q^{2} + 9q + 11 & 4q^{4} + 6q^{3} + 8q^{2} + 10q + 12 & \dots \\
    \end{pmatrix} \]
  hence
  \[ M' = \begin{pmatrix}
    3 & 4 & 5 & 6 & 7 \\
    2q + 1 & 3q + 1 & 4q + 1 & 5q + 1 & 6q + 1 \\
    q^{2} + q + 1 & 2q^{2} + q + 1 & 3q^{2} + q + 1 & 4q^{2} + q + 1 & 5q^{2} + q + 1 \\
    q^{2} + q + 1 & q^{3} + q^{2} + q + 1 & 2q^{3} + q^{2} + q + 1 & 3q^{3} + q^{2} + q + 1 & 4q^{3} + q^{2} + q + 1 \\
    q^{2} + q + 1 & q^{3} + q^{2} + q + 1 & q^{4} + \dots + 1 & 2q^{4} + \dots + 1 & 3q^{4} + \dots + 1 \\
    q^{2} + q + 1 & q^{3} + q^{2} + q + 1 & q^{4} + \dots + 1 & q^{5} + \dots + 1 & 2q^{5} + \dots + 1 \\
    q^{2} + q + 1 & q^{3} + q^{2} + q + 1 & q^{4} + \dots + 1 & q^{5} + \dots + 1 & q^{6} + \dots + 1 \\
    q^{2} + q + 1 & q^{3} + q^{2} + q + 1 & q^{4} + \dots + 1 & q^{5} + \dots + 1 & q^{6} + \dots + 1
    \end{pmatrix} \]
  and finally
  \[ M'' = \begin{pmatrix}
    3 & 4 & 5 & 6 & 7 \\
    2q + 1 & 3q + 1 & 4q + 1 & 5q + 1 & 6q + 1 \\
    q^{2} + q - 2 & 2q^{2} + q - 3 & 3q^{2} + q - 4 & 4q^{2} + q - 5 & 5q^{2} + q - 6 \\
    q^{2} - q & q^{3} + q^{2} - 2q & 2q^{3} + q^{2} - 3q & 3q^{3} + q^{2} - 4q & 4q^{3} + q^{2} - 5q \\
    0 & q^{3} - q^{2} & q^{4} + q^{3} - 2q^{2} & 2q^{4} + q^{3} - 3q^{2} & 3q^{4} + q^{3} - 4q^{2} \\
    0 & 0 & q^{4} - q^{3} & q^{5} + q^{4} - 2q^{3} & 2q^{5} + q^{4} - 3q^{3} \\
    0 & 0 & 0 & q^{5} - q^{4} & q^{6} + q^{5} - 2q^{4} \\
    0 & 0 & 0 & 0 & q^{6} - q^{5}
    \end{pmatrix}. \]
\end{example}

\begin{proof}
  [Proof of \Cref{lem:semi_lie_ker_full_rank}]
  In order to prove $M$ has full rank, it suffices to prove $M''$ has full rank.
  We now confirm the patterns shown by the example above.

  By quoting \Cref{cor:semi_lie_derivative_single} we will write
  \begin{align*}
    M_{i,r}
    &= \sum_{j=0}^{\min(i,r + \left\lfloor \frac{\theta}{2} \right\rfloor)}
      \left( i + \frac{v(b)+v(c)+1}{2} + r - 2j \right) q^j \\
    &- \mathbf{1}_{\substack{\theta \equiv 0 \bmod 2 \\ i \ge r + \theta/2}}
      \cdot q^{v(d-a)+r} \cdot \left( \frac{i-r}{2} + t_i \right)
  \end{align*}
  where
  \[
    t_i = \begin{cases}
      -\frac{v(d-a)}{2} &\text{if } i + r \equiv v(d-a) \pmod 2 \\
      \frac{v(b)+v(c)-3v(d-a)}{2} &\text{if } i + r \not\equiv v(d-a) \pmod 2 \\
    \end{cases}
  \]
  depends only on the parity of $i$.
  Hence for $i \ge 1$ we always have
  \begin{align*}
    M'_{i,r}
    &= M_{i,r} - M_{i-1,r} \\
    &=- \mathbf{1}_{\substack{\theta \equiv 0 \bmod 2 \\ i \ge r + \theta/2}}
      \cdot q^{v(d-a)+r} \cdot \left( \frac{i-r}{2} + t_i \right) \\
    &+ \mathbf{1}_{\substack{\theta \equiv 0 \bmod 2 \\ i-1 \ge r + \theta/2}}
      \cdot q^{v(d-a)+r} \cdot \left( \frac{i-1-r}{2} + t_{i-1} \right) \\
    &+ \mathbf{1}_{i \le r + \left\lfloor \frac{\theta}{2} \right\rfloor} \cdot
      \left( \frac{v(b)+v(c)+1}{2} + r - i \right) q^i \\
    &+ \sum_{j=0}^{\min(i-1,r + \left\lfloor \frac{\theta}{2} \right\rfloor)} q^j.
  \end{align*}
  From this we can make the following deductions
  on \[ M''_{i,r} = M'_{i,r} - M'_{i-2,r} \]
  by cancelling most of the terms.
  \begin{itemize}
    \ii If $i \ge r + \left\lfloor \frac{\theta}{2} \right\rfloor + 3$
    then $M''_{i,r} = 0$ is clear.
    \ii If $i = r + \left\lfloor \frac{\theta}{2} \right\rfloor + 2$
    we contend that $M''_{i,r} = 0$ too.
    \begin{itemize}
      \ii When $\theta = v(b) + v(c)$ is odd, the surviving terms are
      \[
        - \left( \frac{v(b)+v(c)+1}{2} + r - (i-2) \right) q^{ i-2}
        + q^{r + \left\lfloor \frac{\theta}{2} \right\rfloor}
      \]
      Substituting in $i = r + \frac{v(b)+v(c)-1}{2} + 2$ gives zero, as needed.

      \ii When $\theta = 2v(d-a)$ is even, the surviving terms are
      \begin{align*}
        &-q^{v(d-a)+r} \cdot \left( \frac{i-r}{2} + t_{i} \right)
          + q^{v(d-a)+r} \cdot \left( \frac{(i-2)-r}{2} + t_{i-2} \right) \\
        &\qquad+ q^{v(d-a)+r} \cdot \left( \frac{i-1-r}{2} + t_{i-1} \right) \\
        &\qquad- \left( \frac{v(b)+v(c)+1}{2} + r - (i-2) \right) q^{i-2} \\
        &\qquad+ q^{r + \left\lfloor \frac{\theta}{2} \right\rfloor} \\
        &= q^{v(d-a)+r} \cdot \left( -1 + \frac{i-1-r}{2} + t_{i-1} - \frac{v(b)+v(c)+1}{2} - r  + i - 2 +1 \right) \\
        &= q^{v(d-a)+r} \cdot \left( -\frac{3}{2}r + t_{i-1} - \frac{v(b)+v(c)}{2} + \frac32 i \right)
      \end{align*}
      Substituting in $i = r + v(d-a) + 2$ (and since $t_i = t_{i-2}$),
      we also get exactly $0$,
      since $t_{i-1} = \frac{v(b)+v(c)-3v(d-a)}{2}$.
    \end{itemize}

    \ii If $i = r + \left\lfloor \frac{\theta}{2} \right\rfloor + 1$,
    we consider again cases on the parity of $\theta$.
    \begin{itemize}
      \ii When $\theta = v(b)+v(c)$ is odd, the surviving terms are
      \[
        -\left( \frac{v(b)+v(c)+1}{2} + r - (i-2) \right) q^{i-2} \\
        + q^{r+\left\lfloor \frac{\theta}{2} \right\rfloor} + q^{r+\left\lfloor \frac{\theta}{2} \right\rfloor-1} \\
        = q^{r+\left\lfloor \frac{\theta}{2} \right\rfloor} - q^{r+\left\lfloor \frac{\theta}{2} \right\rfloor - 1}.
      \]

      \ii When $\theta = 2v(d-a)$ is even, the surviving terms are
      \begin{align*}
        &-q^{v(d-a)+r} \cdot \left( \frac{i-r}{2} + t_{i} \right) \\
        &\qquad+ q^{v(d-a)+r} \cdot \left( \frac{i-1-r}{2} + t_{i-1} \right) \\
        &\qquad- \left( \frac{v(b)+v(c)+1}{2} + r - (i-2) \right) q^{i-2} \\
        &\qquad+ q^{r + \left\lfloor \frac{\theta}{2} \right\rfloor}
        + q^{r + \left\lfloor \frac{\theta}{2} \right\rfloor - 1} \\
        &= \left( t_{i-1}-t_i+\half  \right) q^{r + \left\lfloor \frac{\theta}{2} \right\rfloor} \\
        &\qquad- \left( \frac{v(b)+v(c)+1}{2} + r - (i-2) - 1 \right) q^{r + \left\lfloor \frac{\theta}{2} \right\rfloor - 1} \\
        &= \left( -\frac{v(d-a)}{2} - \frac{v(b)+v(c)-3v(d-a)}{2} + \frac12 \right) q^{r + \left\lfloor \frac{\theta}{2} \right\rfloor} \\
        &\qquad- \left( \frac{v(b)+v(c)+1}{2} + r - ((r+v(d-a)+1)-2) - 1 \right) q^{r + \left\lfloor \frac{\theta}{2} \right\rfloor - 1} \\
        &= - \frac{v(b)+v(c)-1-2v(d-a)}{2} q^{r + \left\lfloor \frac{\theta}{2} \right\rfloor} \\
        &\qquad- \frac{v(b)+v(c)+1-2v(d-a)}{2} q^{r + \left\lfloor \frac{\theta}{2} \right\rfloor - 1}.
      \end{align*}
      In the edge case where $r = 0$ and $\theta \le 1$,
      it can be checked the same formula still holds with the last term omitted.
    \end{itemize}
  \end{itemize}
  Hence we have found a diagonal of $M''$ below with all entries are zero,
  and on which all entries are nonzero except possibly $M''_{1,0} = 0$
  in the case where $r = 0$ and $\theta \le 1$.

  Assuming $r + \left\lfloor \frac{\theta}{2} \right\rfloor \ge 2$,
  suppose we take the rows from $i = r + \left\lfloor \frac{\theta}{2} \right\rfloor + 1$
  up to $i = r + \left\lfloor \frac{\theta}{2} \right\rfloor + N + 1$.
  Then the resulting matrix is upper triangular.
  The determinant is the product of the diagonal entries;
  up to multiplication by sign and a power of $q$, it equals
  and the determinant is equal to
  \[
    \begin{cases}
    (q-1)^{N+1} & \text{if } \theta \equiv 1 \pmod 2 \\
    \left( \frac{v(b)+v(c)-1-2v(d-a)}{2} q + \frac{v(b)+v(c)+1-2v(d-a)}{2} \right)^{N+1} & \text{if } \theta \equiv 0 \pmod 2
    \end{cases}
  \]
  which is manifestly nonzero for any odd prime power $q$.

  In the situation where $r + \left\lfloor \frac{\theta}{2} \right\rfloor = 1$,
  we use the same rows except that we replace the row for $i=1$
  with the row for $i=0$, which has leftmost entry $M_{0,0} = \frac{v(b)+v(c)+1}{2} > 0$.
  Hence the same proof still shows that the determinant is nonzero.
\end{proof}

\subsubsection{Proof of \Cref{thm:semi_lie_ker_trivial}}
We can now deduce:

\semiliekertriv*

\begin{proof}[Proof of \Cref{thm:semi_lie_ker_trivial}]
Suppose we are given some function
\[ \phi = \sum_{r=0}^N (-1)^r c_r \mathbf{1}_{K'_{S, \le r}} \in \HH(S_2(F)). \]
Letting $M$ be the matrix given in \Cref{lem:semi_lie_ker_full_rank},
we are supposed to have
\[ M \begin{pmatrix} c_0 \\ c_1 \\ \vdots \\ c_N \end{pmatrix} = \mathbf{0}. \]
Since $M$ has full rank, it follows that $c_0 = \dots = c_N = 0$.
\end{proof}

\subsection{In the semi-Lie case, the kernel has finite codimension for fixed $v(e)$}
We prove \Cref{thm:semi_lie_ker_huge} in this section.
We start with the following lemma.
\begin{lemma}
  [A combination vanishing for large $r$]
  \label{lem:semi_lie_large_r}
  Let $\guv \in (S_2(F) \times V'_2(F))\rs^-$. If $r \ge v(e) + 2$, we have
  \[
    \partial \Orb \left( \guv,
        \mathbf{1}_{K'_{S, \le r}} + 2 \mathbf{1}_{K'_{S, \le (r-1)}} + \mathbf{1}_{K'_{S, \le (r-2)}}
      \right) = 0.
  \]
\end{lemma}
\begin{proof}
  This follows directly from \Cref{cor:semi_lie_combo} which gives
  \begin{align*}
    \sum_{j=0}^{v(e)} q^j
    &= \frac{(-1)^r}{\log q} \partial
    \Orb \left( \guv,
      \left(\mathbf{1}_{K'_{S, \le r}} + \mathbf{1}_{K'_{S, \le (r-1)}}\right)
      \right) \\
    &= \frac{(-1)^{r-1}}{\log q} \partial
    \Orb \left( \guv,
      \left(\mathbf{1}_{K'_{S, \le (r-1)}} + \mathbf{1}_{K'_{S, \le (r-2)}}\right)
      \right). \qedhere
  \end{align*}
\end{proof}

We need one more lemma:
\begin{lemma}
  \label{lem:chebyshev}
  There is no $Y \in \CC^\times$ such that
  \[ q^3(Y^r + Y^{-r}) - 3q^2(Y^{r-1} + Y^{-(r-1)})
    + 3q(Y^{r-2} + Y^{-(r-2)}) - (Y^{r-3} + Y^{-(r-3)}) = 0 \]
  holds for all sufficiently large integers $r$.
\end{lemma}
\begin{proof}
  Assume for contradiction such a $Y \in \CC^\times$ existed.
  Let $Y^{\pm k} \coloneqq Y^k + Y^{-k}$ for brevity for every integer $k \ge 1$.
  By writing the recursion relations
  \begin{align*}
    Y^{\pm (r-1)} &= Y^{\pm 1} \cdot Y^{\pm (r-2)} - Y^{\pm (r-3)} \\
    Y^{\pm r} &= Y^{\pm 1} \cdot Y^{\pm (r-1)} - Y^{\pm (r-2)} \\
    &= ((Y^{\pm 1})^2 - 1) \cdot Y^{\pm (r-2)} - Y^{\pm 1} \cdot Y^{\pm (r-3)}
  \end{align*}
  we can deduce that
  \begin{align*}
    0 &= q^3 \cdot Y^{\pm r} - 3q^2 \cdot Y^{\pm (r-1)} + 3q \cdot Y^{\pm (r-2)} - Y^{\pm (r-3)} \\
    &= \left[ q^3 \cdot \left( (Y^{\pm 1})^2 - 1 \right)
      - 3q^2 \cdot Y^{\pm 1} + 3q \right] \cdot Y^{\pm (r-2)}
    - \left[ q^3 \cdot Y^{\pm 1} - 3q^2 + 1 \right] \cdot Y^{\pm (r-3)}.
  \end{align*}

  Now, in general there is no complex number $Y \in \CC^\times$ such that
  $Y^r + Y^{-r} = 0$ for two consecutive values of $r$.
  Hence, if either bracketed coefficient is zero, then so must be the other one.
  However, in that case, we would conclude that
  $Y^{\pm 1} = \frac{3q^2-1}{q^3}$ from the second bracketed coefficient, meaning
  \[ 0 = q^3 \cdot \left( \left( \frac{3q^2-1}{q^3} \right)^2 - 1 \right)
    - 3q^2 \cdot \frac{3q^2-1}{q^3} + 3q = \frac{(q^2-1)^3}{q^3} \]
  which is a contradiction, because $q > 1$.

  Hence neither bracketed coefficient can be zero,
  from which we conclude that there is some nonzero constant $c$ such that
  \[ Y^{\pm (r-2)} = c \cdot Y^{\pm(r-3)} \neq 0 \]
  holds for all large $r$.
  But then
  \[ c \cdot Y^{\pm(r-3)}
    = Y^{\pm 1} \cdot Y^{\pm (r-3)} - Y^{\pm (r-4)}
    = Y^{\pm 1} \cdot Y^{\pm (r-3)} - \frac{Y^{\pm (r-3)}}{c} \]
  and hence $c = Y^{\pm 1} - \frac 1c$.
  So either $c = Y$ or $c = \frac 1Y$.
  Then from $Y^{\pm (r-2)} = c \cdot Y^{\pm(r-3)}$ we derive that $Y = \pm 1$.

  But substituting $Y = 1$ in the original equation would imply $(q-1)^3 = 0$
  while $Y = -1$ would imply $(q+1)^3 = 0$ neither of which is possible.
  This contradiction completes the proof of the lemma.
\end{proof}

\subsubsection{Proof of \Cref{thm:semi_lie_ker_huge}}
We now prove:
\semiliekerlarge*

\begin{proof}[Proof of \Cref{thm:semi_lie_ker_huge}]
The first part of \Cref{thm:semi_lie_ker_huge} follows directly from
\Cref{lem:semi_lie_large_r}.

It remains to show the kernel is not contained in any maximal ideal.
Consider the composed isomorphism from \Cref{ch:satake}
given by
\[ \HH(S_2(F)) \xrightarrow{\BC^\eta_S} \HH(\U(\VV_2^+)) \xrightarrow{\Sat} \QQ[Y + Y^{-1}]. \]
By combining \cite[Equation (7.1.9)]{ref:AFLspherical}
(which is also \Cref{lem:finale_base_change} later)
and \cite[Equation (7.1.4)]{ref:AFLspherical}
we find that
\begin{align*}
  \Sat\left(\BC^\eta_S\left(\mathbf{1}_{K', \le r} + \mathbf{1}_{K', \le (r-1)}\right)\right)
  &= (-1)^r \Sat(\mathbf{1}_{\varpi^{-r} \Mat_2(\OO_E) \cap \VV_n^+}) \\
  &= (-1)^r \left( q^r \sum_{j=-r}^r Y^j - q^{r-1} \sum_{j=-(r-1)}^{r-1} Y^j \right).
\end{align*}
Hence, if we define the polynomial
\begin{align*}
  P_r(Y) &\coloneqq (-1)^r \Sat\left(\BC^\eta_S\left(
    \mathbf{1}_{K', \le r} + 2\mathbf{1}_{K', \le (r-1)}
    + \mathbf{1}_{K', \le (r-2)} \right)\right) \\
  &= \left( q^r \sum_{j=-r}^r Y^j - q^{r-1} \sum_{j=-(r-1)}^{r-1} Y^j \right)
    - \left( q^{r-1} \sum_{j=-(r-1)}^{r-1} Y^j - q^{r-2} \sum_{j=-(r-2)}^{r-2} Y^j \right) \\
  &= q^r \sum_{j=-r}^r Y^j - 2q^{r-1} \sum_{j=-(r-1)}^{r-1} Y^j + q^{r-2} \sum_{j=-(r-2)}^{r-2} Y^j
\end{align*}
then for any $r \ge N+2$,
all the polynomials $P_r(Y) - P_{N+2}(Y)$ lie in the kernel.

We now prove there is no choice of a single $Y \in \CC^\times$ for which
$P_r(Y)$ is eventually constant, which would complete the proof.
Indeed, if we write
\begin{align*}
  &P_r(Y) - q P_{r-1}(Y) \\
  &= q^r(Y^r + Y^{-r}) - 2q^{r-1} (Y^{r-1} + Y^{-(r-1)}) + q^{r-2}(Y^{r-2} + Y^{-(r-2)})
\end{align*}
then
\begin{align*}
  &(P_r(Y) - q P_{r-1}(Y)) - (P_{r-1}(Y) - q P_{r-2}(Y)) \\
  &= q^3(Y^r + Y^{-r}) - 3q^2(Y^{r-1} + Y^{-(r-1)}) \\
  &\qquad + 3q(Y^{r-2} + Y^{-(r-2)}) - (Y^{r-3} + Y^{-(r-3)}).
\end{align*}
So we are done by \Cref{lem:chebyshev}.
\end{proof}

\subsection{A sequence of test functions almost lying in the kernel in the semi-Lie case}
We make one additional remark on the kernel that unifies both of the preceding two sections.
For this section, we define the following indicator function for $r \ge 3$:
\[ \phi_r \coloneqq \mathbf{1}_{K', \le r} + \mathbf{1}_{K', \le (r-1)}
   - q^2 (\mathbf{1}_{K', \le (r-2)} + \mathbf{1}_{K', \le (r-3)}). \]
We give the following theorem which can be thought of as a simultaneously
refined version of both \Cref{lem:semi_lie_ker_full_rank} and
\Cref{lem:semi_lie_large_r}.
Roughly, it says that we can define a sequence of test functions
\[ \phi_r + (q+1)\phi_{r-1} + q\phi_{r-2} \qquad r \ge 5 \]
such that for any fixed $\guv$, there are at most three values of $r$
for which the orbital integral does not vanish.
\begin{theorem}
  [A sequence in $\HH(S_2(F))$]
  \label{thm:semi_lie_finite_codim_full}
  Suppose $\guv \in (S_2(F) \times V_2'(F))\rs^-$ is as in \Cref{lem:semi_lie_params}.
  Then
  \[ \partial \Orb \left( \guv, \phi_r + (q+1)\phi_{r-1} + q\phi_{r-2} \right) = 0 \]
  holds for all $r \ge 5$ with at most three exceptions,
  namely those $r$ with
  \begin{align*}
    &v(e) - \min\left(\frac{v(b)+v(c)-1}{2}, v(d-a)\right) + 2 \\
    &\le r \le v(e) - \min\left(\frac{v(b)+v(c)-1}{2}, v(d-a)\right) + 4.
  \end{align*}
\end{theorem}
\begin{proof}
  As always $\theta \coloneqq \min\left( v(b)+v(c), 2v(d-a) \right)$
  as in \Cref{ch:orbitalFJ1}.
  Consider \Cref{cor:semi_lie_combo}, and let $N$, $C$, $C'$ be as in the statement.
  Let $N^{\flat\flat}$, $C^{\flat\flat}$, $(C')^{\flat\flat}$
  be the changes to those constants when one replaces $r$ by $r-2$.
  Then one can record the changes to these parameters explicitly, see \Cref{tab:semi_lie_change_by_two}.

  \begin{table}
    \begin{tabular}{lll}
      \toprule
      Assumptions & Parameters for $r$ & Parameters for $r-2$ \\
      \toprule
      $r \ge v(e) - \left\lfloor \tfrac{\theta}{2} \right\rfloor + 3$
        & $\begin{aligned} N &= v(e) \\ C &= C' = 0 \end{aligned}$
        & $\begin{aligned} N^{\flat\flat} &= v(e) \\ C^{\flat\flat} &= (C')^{\flat\flat} = 0 \end{aligned}$ \\
      \midrule
      $\begin{aligned} r &= v(e) - \left\lfloor \tfrac{\theta}{2} \right\rfloor + 2 \\ \theta &= 2v(d-a) \\ &\text{(exceptional)} \end{aligned}$
        & $\begin{aligned} N &= v(e) \\ C &= C' = 0 \end{aligned}$
        & $\begin{aligned} N^{\flat\flat} &= (r-2) + \left\lfloor \tfrac{\theta}{2} \right\rfloor = v(e) \\
          C^{\flat\flat} &= \tfrac{-1 + (v(b)+v(c)-2v(d-a))}{2} \\
          (C')^{\flat\flat} &= C^{\flat\flat}+1 \end{aligned}$ \\
      \midrule
      $\begin{aligned} r &= v(e) - \left\lfloor \tfrac{\theta}{2} \right\rfloor + 1 \\ \theta &= 2v(d-a) \end{aligned}$
        & $\begin{aligned} N &= v(e) \\ C &= C' = 0 \end{aligned}$
        & $\begin{aligned} N^{\flat\flat} &= (r-2) + \left\lfloor \tfrac{\theta}{2} \right\rfloor =v(e)-1 \\
          C^{\flat\flat} &= 0 \\
          (C')^{\flat\flat} &= 1 \end{aligned}$ \\
      \midrule
      $\begin{aligned} r &\le v(e) - \left\lfloor \tfrac{\theta}{2} \right\rfloor \\ \theta &= 2v(d-a) \\ \varkappa &\equiv 0 \pmod 2 \end{aligned}$
        & $\begin{aligned} N &= r + \left\lfloor \tfrac{\theta}{2} \right\rfloor \\
          C &= \tfrac{\varkappa-1 + (v(b)+v(c)-2v(d-a))}{2} \\
          C' &= C+1 \end{aligned}$
        & $\begin{aligned} N^{\flat\flat} &= (r-2) + \left\lfloor \tfrac{\theta}{2} \right\rfloor \\
          C^{\flat\flat} &= \tfrac{(\varkappa-2)-1 + (v(b)+v(c)-2v(d-a))}{2} \\
          (C')^{\flat\flat} &= C^{\flat\flat}+1 \end{aligned}$ \\
      \midrule
      $\begin{aligned} r &\le v(e) - \left\lfloor \tfrac{\theta}{2} \right\rfloor \\ \theta &= 2v(d-a) \\ \varkappa &\equiv 1 \pmod 2 \end{aligned}$
        & $\begin{aligned} N &= r + \left\lfloor \tfrac{\theta}{2} \right\rfloor \\
          C &= \tfrac{\varkappa-1}{2} \\
          C' &= C+1 \end{aligned}$
        & $\begin{aligned} N^{\flat\flat} &= (r-2) + \left\lfloor \tfrac{\theta}{2} \right\rfloor \\
          C^{\flat\flat} &= \tfrac{(\varkappa-2)-1}{2} \\
          (C')^{\flat\flat} &= C^{\flat\flat}+1 \end{aligned}$ \\
      \midrule
      $\begin{aligned} r &\le v(e) - \left\lfloor \tfrac{\theta}{2} \right\rfloor \\ \theta &= v(b)+v(c) \end{aligned}$
        & $\begin{aligned} N &= r + \left\lfloor \tfrac{\theta}{2} \right\rfloor \\
          C &= v(e) - \left\lfloor \tfrac{\theta}{2} \right\rfloor + r \\
          C' = 0 \end{aligned}$
        & $\begin{aligned} N^{\flat\flat} &= (r-2) + \left\lfloor \tfrac{\theta}{2} \right\rfloor \\
          C^{\flat\flat} &= v(e) - \left\lfloor \tfrac{\theta}{2} \right\rfloor + (r-2) \\
          (C')^{\flat\flat} &= 0 \end{aligned}$ \\
      \bottomrule
    \end{tabular}
    \caption{Comparison of $N$ to $N^{\flat\flat}$, etc.,
      needed to carry out the proof of \Cref{thm:semi_lie_finite_codim_full}.
      Note the exceptional case $r = v(e) - \left\lfloor \frac{\theta}{2} \right\rfloor + 2$
      differs from all the others because $C-C^{\flat\flat}$ can be large.}
    \label{tab:semi_lie_change_by_two}
  \end{table}
  Note in particular except for the single exceptional value
  $r = v(e) - \left\lfloor \frac{\theta}{2} \right\rfloor + 2$
  we should always have $(N^{\flat\flat}+2)-N \in \{0,1,2\}$,
  $C - C^{\flat\flat} \in \{0,1,2\}$, $(C') - (C')^{\flat\flat} \in \{0,1\}$.
  More explicitly, we have the following result from \Cref{tab:semi_lie_change_by_two}
  for every $r \ge 3$:
  \begin{align*}
    &\frac{(-1)^{r+v(c)}}{\log q}
    \partial \Orb \left( \guv, \phi_r \right) \\
    &=
    \begin{cases}
      -q^{r+\left\lfloor \frac{\theta}{2} \right\rfloor} -q^{r+\left\lfloor \frac{\theta}{2} \right\rfloor - 1} + q + 1
        & \text{if } r \le v(e) - \left\lfloor \frac{\theta}{2} \right\rfloor + 1 \text{ and } \theta = v(b)+v(c) \\
      -2q^{r+\left\lfloor \frac{\theta}{2} \right\rfloor} + q + 1 & \text{if } r \le v(e) - \left\lfloor \frac{\theta}{2} \right\rfloor + 1 \text{ and } \theta = 2v(d-a) \\
      -q^{v(e)+2} - q^{v(e)+1} + q + 1  & \text{if } r \ge v(e) - \left\lfloor \frac{\theta}{2} \right\rfloor + 3.
    \end{cases}
  \end{align*}
  It follows that for $r \ge v(e) - \left\lfloor \frac{\theta}{2} \right\rfloor + 4$ we have
  \begin{equation}
    \frac{(-1)^{r+v(c)}}{\log q}
    \partial \Orb \left( \guv, \phi_r + \phi_{r-1} \right) = 0
    \label{eq:semi_lie_finite_codim_large_r}
  \end{equation}
  while when $r \le v(e) - \left\lfloor \frac{\theta}{2} \right\rfloor + 1$ we have
  \begin{equation}
    \frac{(-1)^{r+v(c)}}{\log q}
    \partial \Orb \left( \guv, \phi_r + q\phi_{r-1} \right)
    = 1-q^2.
    \label{eq:semi_lie_finite_codim_small_r}
  \end{equation}
  Then \Cref{thm:semi_lie_finite_codim_full} follows directly
  from \eqref{eq:semi_lie_finite_codim_large_r} and \eqref{eq:semi_lie_finite_codim_small_r}.
\end{proof}

\section{Transfer factors}
\label{ch:transf}

In this short chapter we document the definitions of the transfer
factors appearing in \Cref{conj:inhomog_spherical} and \Cref{conj:semi_lie_spherical}.

\subsection{Transfer factor for the inhomogeneous group AFL}
The definition of the transfer factor in \Cref{conj:inhomog_spherical} is given below:
\begin{definition}
  [{\cite[Equation 2.7]{ref:ZhiyuParahoric}}]
  Choose any $\gamma \in S_n(F)\rs$.
  Let $\mathbf{e} = \begin{pmatrix} 0 & \dots & 0 & 1 \end{pmatrix}^\top \in F^n$
  be a column vector.
  Then the transfer factor $\omega(\gamma)$ is defined by
  \[ \omega(\gamma) \coloneqq
    \eta\left( \det\left( \left( \gamma^i \mathbf{e} \right)_{i=0}^{n-1} \right) \right).
  \]
\end{definition}

\subsection{Transfer factor for the semi-Lie AFL}
\begin{definition}
  [{\cite[Equation 2.2]{ref:ZhiyuParahoric}}]
  Choose any $\guv \in (S_n(F) \times V'_n(F))\rs$.
  Then the transfer factor $\omega(\gamma)$ is defined by
  \[ \omega\guv \coloneqq
    \eta\left( \det\left( \left( \gamma^i \uu \right)_{i=0}^{n-1} \right) \right).
  \]
\end{definition}

For the purposes of our $n=2$ calculation, we compute the transfer factor when
\[ \guv = \left( \begin{pmatrix} a & b \\ c & d \end{pmatrix},
  \begin{pmatrix} 0 \\ 1 \end{pmatrix}, \begin{pmatrix} 0 & e \end{pmatrix} \right)
  \in (S_2(F) \times V'_2(F))\rs^-. \]
is as described in \Cref{lem:semi_lie_params}.
Applying the definition above, we find that
\begin{equation}
    \omega\guv
    = \eta\left( \det
      \left( \gamma^0 \uu, \gamma^1 \uu \right)
    \right)
    = \eta\left( \det
      \begin{pmatrix} 0 & b \\ 1 & d \end{pmatrix}
    \right)
    = (-1)^{v(b)} = (-1)^{v(c)+1}.
  \label{eq:semi_lie_transfer}
\end{equation}

\section{The geometric side}
\label{ch:geo}

\subsection{Rapoport-Zink spaces}
We briefly recall the theory of Rapoport-Zink spaces.
This follows the exposition in \cite[\S4.1]{ref:survey}.

Let $\breve F$ denote the completion of a maximal unramified extension of $F$,
and let $\FF$ denote the residue field of $\OO_{\breve F}$.
Suppose $S$ is a $\Spf \OO_{\breve F}$-scheme.
Then we can consider triples $(X, \iota, \lambda)$ consisting of the following data.
\begin{itemize}
  \ii $X$ is a formal $\varpi$-divisible $n$-dimensional $\OO_F$-module over $S$
  whose relative height is $2n$.

  \ii $\iota \colon \OO_E \to \End(X)$ is an action of $\OO_E$
  such that the induced action of $\OO_F$ on $\Lie X$
  is via the structure morphism $\OO_F \to \Sheaf_S$.

  We require that $\iota$ satisfies the Kottwitz condition of signature $(n-1,1)$,
  meaning that for all $a \in \OO_E$,
  the characteristic polynomial of $\iota(a)$ on $\Lie X$
  is exactly \[ (T-a)^{n-1} (T-\bar a) \in \Sheaf_S[T]. \]

  \ii $\lambda \colon X \to X^\vee$ is a principal $\OO_F$-relative polarization.

  We require that the Rosati involution of $\lambda$
  induces the map $a \mapsto \bar a$ on $\OO_F$
  (i.e.\ the nontrivial automorphism of $\Gal(E/F)$).
\end{itemize}
The triple is called supersingular if $X$ is a supersingular strict $\OO_F$-module.

For each $n \ge 1$, over $\FF$
we choose a supersingular triple $(\XX_n, \iota_{\XX_n}, \lambda_{\XX_n})$;
it's unique up to $\OO_E$-linear quasi-isogeny compatible with the polarization,
and refer to it as the \emph{framing object}.
We can now define the Rapoport-Zink space:
\begin{definition}
  [Rapoport-Zink space; {\cite[\S5.1]{ref:AFLspherical}}]
  For each $n \ge 1$, we let $\RZ_n$ denote the
  functor over $\Spf \OO_{\breve F}$ defined as follows.
  Let $S$ be an $\Spf \OO_{\breve F}$ scheme, and let
  $\ol S \coloneqq S \times_{\Spf \OO_{\breve F}} \Spec \FF$.
  For every $\Spf \OO_{\breve F}$ scheme, we let $\RZ_n(S)$
  be the set of isomorphism classes of quadruples
  \[ (X, \iota, \lambda, \rho) \]
  where $(X, \iota, \lambda)$ is one of the triples as we described, and
  \[ \rho \colon X \times_S \ol S \to \XX_n \times_{\Spec \FF} \ol S \]
  is a \emph{framing}, meaning it is a height zero $\OO_F$-linear quasi-isogeny
  and satisfies
  \[ \rho^\ast((\lambda_{\XX_n})_{\ol S}) = \lambda_{\ol S}. \]
\end{definition}
Then $\RZ_n$ is formally smooth over $\OO_{\breve F}$ of relative dimension $n-1$.

Henceforth, we also make the following abbreviation.
\begin{definition}
  [$\RZ_{m,n}$]
  For integers $m$ and $n$,
  \[ \RZ_{m,n} \coloneq \RZ_{m} \times_{\Spf \OO_{\breve F}} \RZ_n. \]
\end{definition}

\subsection{A realization of the non-split Hermitian space $\VV_n^-$ of dimension $n$}
For the following definition (and later on), we need a variation of $\RZ_1$:
\begin{definition}
  [$\EE$]
  Let $(\EE, \iota_\EE, \lambda_\EE)$ be the unique triple over $\FF$
  whose Rosati involution has signature $(1,0)$
  (note this is different from $\RZ_1$ where the signature is $(0,1)$ instead).
\end{definition}

At the same time, we can define the following Hermitian space.
\begin{definition}
  [{\cite[\S5.2]{ref:AFLspherical}}]
  For each $n \ge 1$, let
  \[ \VV_n^- \coloneqq \Hom_{\OO_E}^\circ (\EE, \XX_n) \]
  which we call the space of special homomorphisms.
  When endowed with the form
  \[ \left< x,y \right> = \lambda_{\EE}^{-1} \circ y^\vee \circ \lambda_{\XX_n} \circ x
    \in \End_{E}^\circ(\EE) \simeq E \]
  it becomes an $n$-dimensional $E/F$-Hermitian space.
  \label{def:VV_n_nonsplit}
\end{definition}
\begin{proposition}
  [Realization of $\VV_n^-$]
  Up to isomorphism, $\VV_n^-$ is the unique $n$-dimensional
  nondegenerate non-split $E/F$-Hermitian space.
\end{proposition}
\begin{proof}
  See the comment in \cite[\S5.2]{ref:AFLspherical}
  or the comment after \cite[Equation (4.2)]{ref:survey}.
\end{proof}

As described in \cite[Equation (4.3)]{ref:survey}, there is an action of
$\U(\VV_n^-)$ on $(\XX_n, \iota_{\XX_n}, \lambda_{\XX_n})$
and hence each $g \in \U(\VV_n^-)$ acts on $\RZ_n$ by
\[ g \cdot (X, i, \lambda, \rho) \coloneqq (X, i, \lambda, g \circ \rho). \]

\subsection{Intersection numbers for the group version of AFL for the full spherical Hecke algebra}
Here we reproduce the definition of the intersection number used in \Cref{conj:inhomog_spherical}.

Compared to the formulation of the group version and semi-Lie version of the AFL,
the intersection number requires the introduction of a
\emph{Hecke operator} $\TT_{\varphi}$ for an element
\[ \varphi \in \HH(G^\flat \times G, K^\flat \times K) \]
as introduced in \cite{ref:AFLspherical}.
This definition is too involved to reproduce here in its entirety,
we give a summary for this special cases in which we need.

First consider the given $f \in \HH(\U(\VV_n^+))$.
The main work of the construction is to define another
derived formal scheme $\mathcal{T}_{\mathbf{1}_{K^\flat} \otimes f}$
(see \cite[\S6.1]{ref:AFLspherical}) together with two projection maps
\begin{center}
\begin{tikzcd}
  & \ar[ld] \mathcal T_{\mathbf{1}_{K^\flat} \otimes f} \ar[rd] & \\
  \RZ_{n-1,n} && \RZ_{n-1,n}.
\end{tikzcd}
\end{center}
This definition is carried out in \cite[\S5]{ref:AFLspherical},
by defining it first for so-called \emph{atomic elements} of the spherical Hecke algebra,
which form basis elements of a certain presentation of this Hecke algebra
as the unitary group for a polynomial algebra;
we refer the reader to \emph{loc.\ cit.}~for the full details.

Now, take the natural closed embedding
\[ \RZ_{n-1} \to \RZ_n \]
and let
\[ \Delta \colon \RZ_{n-1} \hookrightarrow \RZ_{n-1,n} \]
be the associated graph morphism; let $\Delta_{\RZ_{n-1}}$ denote the image,
with an inclusion $\iota \colon \Delta_{\RZ_{n-1}} \hookrightarrow \RZ_{n-1,n}$.
Once this is done, consider then the diagram
\begin{center}
\begin{tikzcd}
  & \pi_1^\ast(\Delta_{\RZ_{n-1}}) \ar[ld] \ar[r] \ar[rd]
    & \ar["\pi_1" near end, ld] \mathcal T_{\mathbf{1}_{K^\flat} \otimes f}
      \ar["\pi_2" near end, rd] \\
  \Delta_{\RZ_{n-1}} \ar[r, "\iota"', hook] & \RZ_{n-1,n}
  & (\pi_2)_\ast(\pi_1^\ast(\Delta_{\RZ_{n-1,n}})) \ar[r] & \RZ_{n-1,n}.
\end{tikzcd}
\end{center}
That is, one takes the pullback of
$\Delta_{\RZ_{n-1}} \xhookrightarrow{\iota} \RZ_{n-1,n}$
along the projection
\[ \RZ_{n-1,n} \xleftarrow{\pi_1} \mathcal{T}_{\mathbf{1}_{K^\flat} \otimes f} \]
and then takes the pushforward along the other projection
\[ \mathcal{T}_{\mathbf{1}_{K^\flat} \otimes f} \xrightarrow{\pi_2} \RZ_{n-1,n}. \]
\begin{definition}
  [Hecke operator]
  Set
  \[
    \TT_{\mathbf{1}_{K^\flat} \otimes f} (\Delta_{\RZ_{n-1}})
    \coloneqq (\pi_2)_\ast(\pi_1^\ast(\Delta_{\RZ_{n-1}})).
  \]
\end{definition}
This is the part of the intersection number depending on $f$
(or rather $\TT_{\mathbf{1}_{K^\flat} \otimes f}$).
As for our $g \in \U(\VV_n^-)\rs$,
consider the translation $(1,g) \cdot \Delta_{\RZ_{n-1}}$.
The intersection number is then defined as by taking the intersection
of these two objects using the derived tensor product $\jiao$ of the structure sheaves.

\begin{definition}
  [$\Int((1,g), \mathbf{1}_{K^\flat} \otimes f)$; {\cite[Equation (6.1.1)]{ref:AFLspherical}}
  or {\cite[Equation (4.4)]{ref:survey}}]
  \label{def:intersection_number_inhomog}
  We define the intersection number in \Cref{conj:inhomog_spherical} by
  \begin{align*}
    \Int((1,g), \mathbf{1}_{K^\flat} \otimes f)
    &\coloneqq \left\langle
      \TT_{\mathbf{1}_{K^\flat} \otimes f} \Delta_{\RZ_{n-1}},
      (1,g) \cdot \Delta_{\RZ_{n-1}} \right\rangle _{\RZ{n-1,n}} \\
    &\coloneqq \chi_{\RZ_{n-1,n}} \left(
      \Sheaf_{\TT_{\mathbf{1}_{K^\flat} \otimes f} (\Delta_{\RZ_{n-1}})}
      \jiao_{\Sheaf_{\RZ_{n-1,n}}} \Sheaf_{(1,g) \cdot \Delta_{\RZ_{n-1}}} \right).
  \end{align*}
\end{definition}

Here $\chi$ denotes the Euler-Poincar\'{e} characteristic,
meaning that if $X$ is a formal scheme over $\Spf \OO_{\breve F}$
then given a finite complex $\mathcal{F}$ of $\Sheaf_X$-modules we set
\[ \chi_X(\mathcal{F}) = \sum_i \sum_j (-1)^{i+j}
  \operatorname*{len}_{\OO_{\breve F}} H^j(X, H_i(\mathcal F)) \]
provided all the lengths are finite.

\begin{remark}
  In general we could adapt this definition so $(1,g)$ is replaced by
  by an element of $(\U(\VV_{n-1}^-) \times \U(\VV_n^-))\rs$
  if we wish to work with the full group version of the AFL
  rather than just the inhomogeneous version.
  However, this simpler definition will be sufficient for our purposes.
\end{remark}

\subsection{Intersection numbers for the semi-Lie version of AFL for the full spherical Hecke algebra}
Now we continue to define an intersection number needed for the proposed
\Cref{conj:semi_lie_spherical} from earlier.
The definition mirrors the one given in the last section.
Here we reproduce the definition of the intersection number used in \Cref{conj:inhomog_spherical}.

We work here with $\RZ_{n,n}$ rather than $\RZ_{n-1,n}$.
The change is that we need to incorporate the new $u \in \VV_n^-$ that was not present before.
In order to do this one considers a certain relative Cartier divisor $\ZD(u)$
on $\RZ_n$ for each nonzero $u \in \VV_n^-$.
This divisor was defined by Kudala and Rapport in \cite{ref:KR}
and accordingly we call it a \emph{KR-divisor} following \cite[\S4.3]{ref:survey}.
The definition is given as follows.
\begin{definition}
  [$\ZD(u)$; {\cite[Definition 3.2]{ref:KR}}]
  Recall that $(\EE, \iota_\EE, \lambda_\EE)$ is the unique triple over $\FF$
  whose Rosati involution has signature $(1,0)$.
  Hence the formal $\OO_F$-module has a unique lifting called its \emph{canonical lifting},
  which we denote by the triple $(\mathcal{E}, \iota_{\mathcal{E}}, \lambda_{\mathcal E})$.

  Then the KR-divisor $\ZD(u)$ is the locus where the quasi-homomorphism
  \[ u \colon \EE \to \XX_n \]
  lifts to a homomorphism from $\mathcal{E}$ to the universal object over $\RZ_n$.
  More explicitly, it consists of those $X \in \RZ_n$
  such that there exists a map $\varphi : \mathcal{E} \to X$ with the following property.
  Let $S$ be an $\Spf \OO_{\breve F}$ scheme and consider the map on special fiber
  \[ \mathcal{E} = \EE \times_S \ol S \xrightarrow{\varphi \times_S \ol S} X' \times_S \ol S. \]
  Since $X \in \RZ_n$ we also have $\rho \colon X \times_S \ol S \to \XX_n \times_{\Spec \FF} \ol S$.
  Moreover, $u$ gives a map
  \[
    \EE \times_{\Spec \FF} \ol S
    \xrightarrow{u \times _{\Spec \FF} \ol S}
    \XX_n \times_{\Spec \FF} \ol S.
  \]
  Then we require the following diagram to commute:
  \begin{center}
  \begin{tikzcd}
    \EE \times_S \ol S \ar[rr, "\varphi \times_S \ol S"] \ar[d, "\rho_{\mathcal E}"] && X \times_S \ol S\ar[d, "\rho"] \\
    \EE \times_{\Spec \FF} \ol S \ar[rr, "u \times_{\Spec \FF} \ol S"] && \XX_n \times_{\Spec \FF} \ol S.
  \end{tikzcd}
  \end{center}
\end{definition}
That is, $\ZD(u)$ is the locus where $u$ lifts to a homomorphism $\mathcal{E} \to \mathcal{X}_n$.
Note also by the definition that $g \ZD(u) = \ZD (gu)$.
See \cite{ref:KR} for a full definition.

The main change is then that we can consider $\Delta_{\ZD(u)}$ as the image of
\[ \ZD(u) \hookrightarrow \RZ_n \xrightarrow{\Delta} \RZ_{n,n} \]
where $\Delta \colon \RZ_n \to \RZ_{n,n}$ now denotes the diagonal map.
If one defines an appropriate space $\mathcal T_{\mathbf{1}_K \otimes f}$
for $f \in \HH(\U(\VV_n^+))$ again following \cite{ref:AFLspherical}, together with
\begin{center}
\begin{tikzcd}
  & \ar[ld] \mathcal T_{\mathbf{1}_K \otimes f} \ar[rd] & \\
  \RZ_{n,n} && \RZ_{n,n}
\end{tikzcd}
\end{center}
then one can then repeat the diagram from before:
\begin{center}
\begin{tikzcd}
  & \pi_1^\ast(\Delta_{\ZD(u)}) \ar[ld] \ar[r] \ar[rd]
    & \ar["\pi_1" near end, ld] \mathcal T_{\mathbf{1}_K \otimes f}
      \ar["\pi_2" near end, rd] \\
  \Delta_{\ZD(u)} \ar[r, "\iota"', hook] & \RZ_{n,n}
  & (\pi_2)_\ast(\pi_1^\ast(\Delta_{\ZD(u)})) \ar[r] & \RZ_{n,n}.
\end{tikzcd}
\end{center}
In other words, we again take a pullback followed by a pushforward
but this time of $\Delta_{\ZD(u)} \hookrightarrow \RZ_{n,n}$.
This lets us write an analogous definition:
\begin{definition}[Hecke operator]
  Set
  \[ \TT_{\mathbf{1}_K \otimes f} (\Delta_{\ZD(u)})
    \coloneqq (\pi_2)_\ast(\pi_1^\ast(\Delta_{\ZD(u)})). \]
\end{definition}
Meanwhile to replace $(1,g) \Delta_{\RZ_{n,n-1}}$, we let
\[ \Gamma_g \subseteq \RZ_{n,n} \]
denote the graph of the automorphism of $\RZ_n$ induced by $g$.
This finally allows us to write a definition of the intersection number in the semi-Lie case:
\begin{definition}
  [$\Int((g,u), f)$]
  \label{def:intersection_number_semi_lie_spherical}
  In analog to \Cref{def:intersection_number_inhomog} for the group version of the AFL,
  we now define the intersection number in \Cref{conj:semi_lie_spherical} as
  \begin{align*}
    \Int((g,u), f)
    &\coloneqq \left\langle \TT_{\mathbf{1}_{K} \otimes f} \Delta_{\ZD(u)}, \Gamma_g \right\rangle _{\RZ{n,n}} \\
    &\coloneqq \chi_{\RZ_{n,n}} \left(
      \Sheaf_{\TT_{\mathbf{1}_K \otimes f}(\Delta_{\ZD(u)})}
      \jiao_{\Sheaf_{\RZ_{n,n}}} \Sheaf_{\Gamma_g} \right).
    \end{align*}
\end{definition}
In the situation where $f = \mathbf{1}_K$,
this coincides with the existing definition in \cite[Equation (4.9)]{ref:survey}
and \cite[Remark 3.1]{ref:annalsAFL}.

\subsection{An analogy between the geometric and analytic sides}
With the intersection number now defined for \Cref{conj:semi_lie_spherical},
we provide some intuitive discussion about the connection.
All of this is for philosophical cheerleading only,
and is not meant to formally assert any definitions or results.
But it may help in motivating the formulation of the conjecture.

For this section write $G \coloneqq \U(\VV_n^+)$ and $K \coloneqq G \cap \GL_n(\OO_E)$
the hyperspecial maximal compact subgroup of $G$.
For simplicity we only focus on the semi-Lie AFL originally proposed by Liu to start;
which is the special case of \Cref{conj:semi_lie_spherical}
when $f = \mathbf{1}_K$ and $\phi = \mathbf{1}_{K'}$.

\paragraph{The geometric side}
On the geometric side, $\RZ_n$ is the RZ-space acted on by $\U(\VV_n^-)$,
and hence $\U(\VV_n^-) \times \U(\VV_n^-)$ acts on $\RZ_{n,n}$.
Roughly speaking, we are considering the two morphisms
\[ \RZ_n \xrightarrow{\Delta} \RZ_{n,n} \xleftarrow{\Gamma_g} \RZ_n \]
with $\Delta$ being thought of as the diagonal morphism
and $\Gamma_g$ as the graph under multiplication by $g \in \U(\VV_n^-)\rs$.

Hence loosely speaking, the intersection $\Int\left( (g,u), \mathbf{1}_K \right)$
can be thought of as the intersection of three images in $\RZ_{n,n}$:
\begin{itemize}
  \ii A ``diagonal'' object $\Delta$;
  \ii A ``graph'' object $\Gamma$;
  \ii A third object $\ZD(u)$, the KR-divisor, parametrized by diagrams
  \begin{center}
  \begin{tikzcd}
    \mathcal{E}\ar[r] \ar[d, dash] & \mathcal{X}_n \ar[d, dash] \\
    \EE \ar[r, "u"] & \XX_n.
  \end{tikzcd}
  \end{center}
\end{itemize}
The derived tensor product $\jiao$
is used together with some formalism to make this intersection idea precise.
The intersection of the ``diagonal'' and ``graph'' is the \emph{fixed point locus},
and in fact could be formally defined as the intersection
\[ \Gamma_g \cap \Delta_{\RZ_n} \]
viewed as a closed formal subscheme of $\RZ_n$ (or $\RZ_{n,n}$);
see \cite[Equation (4.6)]{ref:survey}.

\paragraph{The analytic side}
On the other hand, consider the analytic side.
We will try to explain how the weighted orbital integral in \Cref{def:orbitalFJ}
can be thought of as some weighted intersection of analogous objects.

Note the quotient $G/K$ can be identified as
\[ G/K \simeq \left\{ \Lambda \subseteq \VV_n^+ \mid \Lambda^\vee = \Lambda \right\} \]
that is, the set of self-dual lattices $\Lambda$ of full rank,
which thus has a natural action of $G$.
Henceforth we denote elements of $G/K$ by $h$,
and fix one particular such lattice $\Lambda_0$, acted on by $\OO_E$.
Hence $G/K$ can be thought of as
\[ G/K \simeq \left\{ h \Lambda_0 \mid h \in G/K \right\}. \]

Recall from \eqref{eq:unweighted_orbital_semi_lie}
that we have an orbital integral on the unitary side of the shape
\[ \int_{h \in G/K} \mathbf{1}_K(h^{-1} g h) \mathbf{1}_{\Lambda_0}(h \cdot u) \odif h \]
where $u \in \VV_n^+$, and $\Lambda_0$ is a fixed particular lattice in $G/K$.
See for example the ``relative fundamental lemma''
stated as \cite[Conjecture 1.9]{ref:liuFJ}.

Like before, we can consider two maps
\[ G/K \xrightarrow{\Delta} G/K \times G/K \xleftarrow{\Gamma_g} G/K \]
which are the diagonal morphism and the graph of the action of $g$.
Hence the intersection are those cosets $hK$ for which
\[ hK = ghK \iff h^{-1} g h \in K. \]
Hence the indicator function $\mathbf{1}_K(h^{-1} g h)$
plays the analog of the fixed point locus in the geometric side.

Meanwhile, the term $h \cdot u$ plays a role analogous
to the KR-divisor on the geometric side, giving the third intersection object.
We have
\[ h \cdot u \in \Lambda_0 \iff u \in h^{-1} \Lambda_0 \]
and so the object corresponding to the KR-divisor $\ZD(u)$
is the subset in $G/K$ of those lattices containing $u$, that is
\[ \left\{ \Lambda' \mid \Lambda' \ni u \right\}. \]
The analog to the earlier diagram that we described for $\ZD(u)$ is then
\begin{align*}
  \OO_E &\to \Lambda' \\
  1 &\mapsto u.
\end{align*}

Up until now this whole section is written for $f = \mathbf{1}_K$ and $\phi = \mathbf{1}_{K'}$.
In the general situation,
if one replaces $\mathbf{1}_K$ in the above integral by a general $f$,
then this corresponds to changing the analog of the fixed point locus;
the idea of \cite{ref:AFLspherical} is that
this should correspond to replacing $\Delta_{\ZD(u)}$
with $\TT_f(\Delta_{\ZD(u)})$ on the geometric side.

\section{Intersection numbers for $\Int((g,u), f)$ for $n = 2$}
\label{ch:jiao}
This chapter is dedicated to computing intersection numbers
for the semi-Lie version of AFL in the special case $n = 2$.

\subsection{Background on quaternion division algebra}
Through this section we let $\DD$ be a quaternion division algebra over $F$,
with a fixed maximal order $\OO_\DD$.
We will make $\DD$ explicit in the following way for our calculations to follow.

\subsubsection{Structure as a noncommutative algebra}
As $F$-vector spaces we will write
\[ \DD = E \oplus E \Pi \]
where $\Pi$ is selected so that $\Pi^2 = \varpi$.
We endow $\DD$ with a noncommutative multiplication according to
\[ \Pi t = \bar t \Pi \qquad \text{ for all } t \in E \]
where $\ol{t}$ is the image of $t \in E$ under the nontrivial element of $\Gal(E/F)$.

\subsubsection{Conjugation of elements of $\DD$}
In general, suppose $x \in \DD$ is any element
decomposed as $x = a + b \Pi$ for $a,b \in E$.
Then we denote by $\bar x \in \DD$ the conjugate in $\DD$ defined by
\[ \bar x \coloneqq \bar a - b \Pi \]
where, again, $\bar a$ is the image of $a \in E$ under the nontrivial element of $\Gal(E/F)$.
It is an anti-involution, meaning that $\ol{\bar x} = x$ and $\ol{xy} = \bar y \bar x$.

(Notice that we have a slight abuse of notation here in that we have
used the same notation to denote both conjugation under the Galois action of $\Gal(E/F)$
as well as the conjugation in $\DD$.
However, there is no ambiguity resulting because when $E$ is viewed as a subset of $\DD$,
the two symbols denote the same element of $E$:
that is we have
\[ \ol{a + 0 \Pi} = \bar{a} + 0 \Pi \]
in any event.
In other words, the restriction of the quaternion conjugation to $E$
coincides with the nontrivial element of $\Gal(E/F)$,
so we do not need to introduce a separate notation for it.)

This allows us to define the reduced norm and trace in $\DD$.
The reduced trace is given by
\[ \tr x \coloneqq x + \bar x = \Tr{E/F}(a) = 2x_0 \in F. \]
We may thus define
\[ \DT \coloneqq \left\{ u \in \DD \mid \tr u = 0 \right\} \]
which has codimension $1$ inside $\DD$ (i.e.\ is three-dimensional as an $F$-vector space).
Since $\tr(a+b\Pi) = \Tr_{E/F}(a)$, we could also write
\[ \DT = \{ a+b\Pi \mid a,b \in E \text{ and } \Tr_{E/F}(a) = 0 \}. \]

The reduced norm is similarly defined by
\begin{align*}
  \Nm x &= x \bar x = (a + b \Pi)(\bar a - b \Pi) \\
  &= a \bar a + b \Pi \bar a - a b \Pi - b \Pi b \Pi \\
  &= a \bar a - b \bar b \varpi \\
  &= \Norm_{E/F}(a) - \Norm_{E/F}(b) \varpi \in F.
\end{align*}

As an $F$-vector space, $\DD$ has a basis given by
$\{1, \sqrt{\eps}, \Pi, \sqrt{\eps}\Pi\}$, that is
\[ \DD = F \oplus F\sqrt\eps \oplus F \Pi \oplus F\sqrt\eps\Pi. \]
It will be convenient to introduce the following notation:
\begin{definition}
  [$x_0$ and $x_-$]
  For $x \in \DD$, we introduce the notation $x_0$ and $x_-$ to mean
  \begin{itemize}
    \ii $x_0$ is the projection into the first component $F$; and
    \ii $x_- = x - x_0$ is the projection into
    $\DT = F\sqrt\eps \oplus F \Pi \oplus F\sqrt\eps\Pi$.
  \end{itemize}
\end{definition}
In particular, the formula for conjugation then reads as the simpler
\[ \bar x = x_0 - x_-. \]

\subsubsection{Hermitian structure}
We view $\DD$ as an $E/F$-Hermitian space under left multiplication by $E$;
that is, for $a,b,t \in E$ we consider
\[ t \cdot (a + b \Pi) = a t + b t \Pi. \]
as the action of $E$ on $\DD$.
Then we equip $\DD$ with a $E/F$-Hermitian form
$\left\langle \bullet, \bullet \right\rangle \colon \DD \times \DD \to E$
defined by
\[ \left\langle x,y \right\rangle = \half \Tr_{\DD/E}(x \bar y) \]
i.e.\ the projection of $x \bar y \in \DD = E \oplus E\Pi$ onto the first component.
In particular, note that
\[ \left\langle x,x \right\rangle = x \bar{x} = \Nm x \]
or equivalently
\[ \left\langle a+b\Pi, a+b\Pi \right\rangle = a \bar a - b \bar b \varpi. \]

\subsubsection{Identification of $\VV_n^-$ with $\DD$}
We continue using the notation $(\EE, \iota_\EE, \lambda_\EE)$ as the triple over $\FF$
whose Rosati involution has signature $(1,0)$.
Moreover, we will take the identification
\[ \End(\EE) \simeq \OO_\DD \]
see \cite[Remark 2.5]{ref:KR},
and hence the corresponding identification
\[ \VV_n^- \simeq \DD. \]

\subsection{The invariants for the orbit of $(g,u)$}
\label{sec:g_u_invariants}

We specialize to the situation where $u \in \OO_\DD$ and $g \in \U(\VV_n^-)$.

\subsubsection{Coordinates for $g$}
To impose coordinates on $g$, we appeal to the following fact.
\begin{lemma}
  [Description of $\U(\VV_2^-)$]
  Every unitary map in $\U(\VV_2^-)$ can be described in the form
  \[ x \mapsto \lambda\inv x (\alpha + \beta \Pi) \]
  for some quaternion $\alpha + \beta \Pi \in \DD^\times$
  and an element $\lambda \in E^\times$ such that
  \[ \Nm(\alpha + \beta \Pi) = \Norm_{E/F}(\lambda). \]

  Moreover, such a description is unique up to multiplication by elements of $F$.
  In other words,
  \[ \U(\VV_2^-) \simeq (E^\times \times \DD^\times)^\circ / \Delta(F^\times) \]
  where $(\DD^\times \times E^\times)^\circ$
  denotes those pairs $(\lambda, \alpha + \beta \Pi)$
  with $\Norm_{E/F}(\lambda) = \Nm(\alpha + \beta \Pi)$,
  and $\Delta(F^\times)$ is the diagonal embedding of $F^\times$.
\end{lemma}
\begin{remark}
  In this paper we will not have a need to compose multiple such unitary maps.
  However, if we did, then our notation above swaps the multiplication order.
  That is, if we have $g_1, g_2 \in \U(\VV_2^-)$ represented by pairs
  $g_1 \leftrightarrow (\lambda_1, \alpha_1 + \beta \Pi_1)$
  and $g_2 \leftrightarrow (\lambda_2, \alpha_2 + \beta \Pi_2)$
  under the isomorphism above, then
  \[ g_1 \circ g_2 \leftrightarrow (\lambda_2 \lambda_1, (\alpha_2 + \beta_2 \Pi)(\alpha_1 + \beta \Pi)). \]
\end{remark}

Note that in the definition for $g$, if $v(\lambda) \neq 0$
then we can factor out powers of $\varpi = \bar\varpi$ from $\lambda$
and put them into $\alpha$ and $\beta$ instead.
Hence, by relabeling $\alpha$ and $\beta$, we may assume without loss of generality that:
\begin{assume}
  We assume WLOG that $v(\lambda) = 0$ (and hence $v(\alpha)=0$, $v(\beta) \ge 0$).
\end{assume}
This frees us from having to deal with $v(\lambda)$ offsets in subsequent calculation.

We encode $g$ as a matrix on $\DD$ now (with the obvious $E$-basis $1$, $\Pi$,
again viewing $\DD$ as a left-$E$ module)
so we can compute its determinant and trace.
We have
\begin{align*}
  g(1) &= \lambda\inv \cdot 1 \cdot (\alpha + \beta \Pi)
    = \lambda\inv \alpha + \lambda\inv \beta \Pi \\
  g(\Pi) &= \lambda\inv \cdot \Pi (\alpha + \beta \Pi)
    = \lambda\inv \bar\beta \varpi + \lambda\inv \bar\alpha \Pi.
\end{align*}
Hence, written as a matrix with respect to the obvious basis $\{1, \Pi\}$ we have
\[ g = \lambda\inv \begin{pmatrix}
  \alpha & \bar\beta \varpi \\
  \beta &  \bar\alpha
  \end{pmatrix}. \]

\subsubsection{Coordinates for $u$}
We also impose coordinates for $u$ according to
\[ u = s + t \Pi \in \OO_\DD \qquad s, t \in E. \]
To make the calculation that follows less complicated,
we are going to make the following assumption on $u$.
\begin{assume}
  We assume WLOG that either $u \in E$ or $u \in E \Pi$.
  That is, either $s = 0$ or $t = 0$.
  \label{assume:st_zero}
\end{assume}
This assumption can be made without loss of generality because
the invariants and the intersection only depend
on the $\SU(2)$-orbit of the pair $(g,u)$,
and any element $u \in \DD^\times$ can be mapped under an element of $\SU(2)$
into a pair for which $u \in E$ or $u \in E \Pi$.

In order for $(g,u)$ to be regular semisimple we require that
\begin{align*}
  u &= s + t \Pi \\
  g(u) &= \lambda\inv (s + t \Pi)(\alpha + \beta \Pi) \\
  &= \lambda\inv \left( (s \alpha + t \bar\beta \varpi) + (s \beta + t \bar\alpha)\Pi \right)
\end{align*}
are linearly independent, meaning
\begin{align*}
  0 &\neq
  \det \begin{pmatrix}
    s & s \alpha + t \bar \beta \varpi \\
    t & s \beta + t \bar \alpha
  \end{pmatrix}
  = st(\bar\alpha - \alpha) + \beta s^2 - \bar\beta t^2 \varpi \\
  &= \beta s^2 - \bar\beta t^2 \varpi \\
  &= \begin{cases}
    -\bar\beta t^2 \varpi & \text{if } s = 0 \\
    \beta s^2 & \text{if } t = 0.
  \end{cases}
\end{align*}
Hence, we have a requirement that $\beta \neq 0$
and $s$ and $t$ are not both zero (we require $st = 0$ from \Cref{assume:st_zero}).
\begin{remark}[$\alpha \neq 0$]
  Note that necessarily $\alpha$ is nonzero as well.
  This follows from the requirement that
  $\alpha \bar \alpha - \beta \bar \beta \varpi = \lambda \bar\lambda$;
  if $\alpha = 0$ we would get a left-hand side with odd valuation
  but a right-hand side with even valuation.
\end{remark}

\subsubsection{The invariants of the matching}
At this point we can state:
\begin{lemma}
  \label{lem:g_u_invariants}
  Under the coordinates we just described,
  the four corresponding invariants of \Cref{def:matching_semi_lie} are:
  \begin{align*}
    \Tr g &= \lambda^{-1} (\alpha + \bar \alpha) \\
    \det g &= \lambda^{-2} (\alpha \bar \alpha - \beta \bar\beta \varpi) \\
    \left\langle u,u \right\rangle &= s \bar s - t \bar t \varpi \\
    \left\langle g(u), u \right\rangle &= \begin{cases}
        \lambda\inv \bar\alpha \Nm u & \text{if } s = 0 \\
        \lambda\inv \alpha \Nm u & \text{if } t = 0.
      \end{cases}
  \end{align*}
\end{lemma}
\begin{proof}
  The first three are immediate.
  The last one follows by computing
  \begin{align*}
    \left\langle g(u), u \right\rangle
    &= \left\langle \lambda\inv(s + t \Pi)(\alpha + \beta \Pi), s + t \Pi \right\rangle \\
    &= \lambda\inv \left\langle
      (s \alpha + t \bar \beta \varpi) + (t \bar \alpha + s \beta) \Pi,
      s + t \Pi \right\rangle \\
    &= \lambda\inv \cdot \half \Tr_{\DD/E} \left[
      ((s \alpha + t \bar \beta \varpi) + (t \bar \alpha + s \beta) \Pi)
      (\bar s - t \Pi) \right] \\
    &= \lambda\inv \left( s \bar s \alpha - t \bar t \bar \alpha \varpi \right)
  \end{align*}
  and recalling that $s t = 0$.
\end{proof}

\subsection{A basis for $\HH(\U(\VV_2^+))$}
\label{sec:hecke_unitary_basis}

As before $K = \GL_2(\OO_E) \cap \U(\VV_2^+)$denotes the maximal hyperspecial compact subgroup of $\U(\VV_2^+)$.
For each $r \ge 0$, we define
\[ \fr \coloneqq \mathbf{1}_{\varpi^{-r} \Mat_2(\OO_E) \cap \U(\VV_2^+)} \in \HH(\U(\VV_2^+)). \]
For convenience $\fr = 0$ for $r < 0$.
We also set
\[  \mathbf{1}_{K, r} \coloneqq \fr - \mathbf{1}_{K, \le (r-1)} \]
which is the indicator function for the coset
\[ K \begin{pmatrix} 0 & \varpi^r \\ \varpi^{-r} & 0 \end{pmatrix} K. \]
Note when $r = 0$, $\mathbf{1}_{K, 0} = \mathbf{1}_K = \mathbf{1}_{K, \le 0}$.

Analogous to \Cref{sec:hecke_basis_FJ} we then have the following result.
\begin{proposition}[$\fr$ basis]
  The functions $\fr$ (for $r \ge 0$) form a basis of $\HH(\U(\VV_2^+))$.
  (Similarly, so do $\mathbf{1}_{K, r}$ for $r \ge 0$.)
\end{proposition}
\begin{proof}
  This follows from the fact that
  \[ \U(\VV_2^+) = \coprod_{r \ge 0}
    K \begin{pmatrix} 0 & \varpi^r \\ \varpi^{-r} & 0 \end{pmatrix} K. \]
  See the comment in \cite[Equation (7.1.5)]{ref:AFLspherical}.
\end{proof}

The base change for this basis is given later in \Cref{lem:finale_base_change}.

\subsection{Background on special divisors for $n = 2$}
\subsubsection{The Rapoport-Zink space $\RZ_2$ and $\TTr$}
Recall $\RZ_2$ from \Cref{ch:geo}.
With the Hecke operator $\TT$ from \cite{ref:AFLspherical}
(see \Cref{ch:geo} for discussion)
we introduce $\TTr = \TT_{\mathbf{1}_K \otimes \mathbf{1}_{K, \le r}}(\Delta_{\RZ_n})$
so that we have the diagram
\begin{center}
\begin{tikzcd}
  & \ar[ld] \TTr \ar[rd] & \\
  \RZ_{2} && \RZ_{2}
\end{tikzcd}
\end{center}

\subsubsection{The Lubin-Tate space $\MM_2$}
We introduce the notation $\MM_2$ for the \emph{Lubin-Tate space} for $n = 2$.
It is defined almost in the same way as $\RZ_2$ except that we replace
$\XX_2$ with $\EE$ now.

\begin{definition}
  [Lubin-Tate space]
  We let $\MM_2$ denote the functor over $\Spf \OO_{\breve F}$ defined as follows.
  Let $S$ be an $\Spf \OO_{\breve F}$ scheme, and let
  $\ol S \coloneqq S \times_{\Spf \OO_{\breve F}} \Spec \FF$.
  For every $\Spf \OO_{\breve F}$ scheme, we let $\RZ_n(S)$
  be the set of isomorphism classes of quadruples
  \[ (Y, \iota, \lambda, \rho) \]
  where $(Y, \iota, \lambda)$ is one of the triples as we described, and
  \[ \rho \colon Y \times_S \ol S \to \EE \times_{\Spec \FF} \ol S \]
  is a \emph{framing}, meaning it is a height zero $\OO_F$-linear quasi-isogeny
  and satisfies
  \[ \rho^\ast((\lambda_{\EE})_{\ol S}) = \lambda_{\ol S}. \]
\end{definition}

\begin{proposition}
  [{\cite[Example 5.5.6]{ref:AFLspherical}}]
  The Serre tensor construction produces an identification
  \[ \Serre \colon \RZ_2 \xrightarrow{\sim} \MM_2. \]
  By abuse of notation we will also use the same symbol for the map
  \[ \Serre \colon \RZ_{2,2} \xrightarrow{\sim} \MM_2 \times \MM_2. \]
\end{proposition}

Recall we have an action of $\U(\VV_2^-)$ (actually $\PU(\VV_2^-)$) on $\RZ_2$.
We describe now the corresponding action on $\MM_2$.
We have an isomorphism of short exact sequences
\begin{center}
\begin{tikzcd}
  1 \ar[r] & \OO_E^\times / \OO_F \ar[r] & (\OO_\DD^\times \times \OO_E^\times)^\circ / \Delta(\OO_F^\times) \ar[r] & \OO_\DD^\times / \OO_F^\times \ar[r] & 1 \\
  1 \ar[r] & \U(1) \ar[r] \ar[u, equals] & \U(\VV_2^-) \ar[r] \ar[u, equals] & \PU(\VV_2^-) \ar[r] \ar[u, equals] & 1.
\end{tikzcd}
\end{center}
The image of $\alpha + \beta \Pi \in \OO_\DD^\times / \OO_F^\times$
is then $(\lambda, \alpha + \beta \Pi) \in \PU(2)$
for any choice of $\lambda$ with $\lambda \bar \lambda = \Nm (\alpha + \beta \Pi)$.

\subsubsection{The divisor $\ZO4$ on $\MM_2 \times \MM_2$}
Now we define the special orthogonal divisor $\ZO4(u)$
on $\MM_2 \times \MM_2$ as follows.
\begin{definition}
  [$\ZO4(u)$]
  Let $u \in \OO_\DD$.
  Then we define the divisor $\ZO4(u)$ to be the pairs $(Y, Y') \in \MM_2 \times \MM_2$
  such that there exists $\varphi \colon Y \to Y'$ with the following property.
  Let $S$ be an $\Spf \OO_{\breve F}$ scheme and consider the map on special fiber
  \[ Y \times_S \ol S \xrightarrow{\varphi \times_S \ol S} Y' \times_S \ol S. \]
  Also, from $Y \in \MM_2$ and $Y' \in \MM_2$ we have the data of framings
  $\rho \colon Y \times_S \ol S \to \EE \times_{\Spec \FF} \ol S$
  and $\rho' \colon Y' \times_S \ol S \to \EE \times_{\Spec \FF} \ol S$.
  Moreover, $u$ gives a map
  \[
    \EE \times_{\Spec \FF} \ol S
    \xrightarrow{u \times _{\Spec \FF} \ol S}
    \EE \times_{\Spec \FF} \ol S.
  \]
  Then we require the following diagram to commute:
  \begin{center}
  \begin{tikzcd}
    Y \times_S \ol S \ar[rr, "\varphi \times_S \ol S"] \ar[d, "\rho"] && Y' \times_S \ol S\ar[d, "\rho'"] \\
    \EE \times_{\Spec \FF} \ol S \ar[rr, "u \times_{\Spec \FF} \ol S"] && \EE \times_{\Spec \FF} \ol S.
  \end{tikzcd}
  \end{center}
\end{definition}

We need the following additional hypothesis:
\begin{hypothesis}
  [$\TTr \simeq \ZO4(\varpi^r)$]
  \label{hypo:serre_pullback_space}
  The Serre tensor construction gives an isomorphism
  \[ \TTr \simeq \ZO4(\varpi^r) \]
  such that we get an analogous diagram
  \begin{center}
  \begin{tikzcd}
    & \ar[ld] \ZO4(\varpi^r) \ar[rd] & \\
    \MM_{2} && \MM_{2}.
  \end{tikzcd}
  \end{center}
\end{hypothesis}
This hypothesis is shown in the in-preparation
\cite{ref:CLZ} by Ryan C.~Chen, Weixiao Lu, and Wei Zhang.
Even without this hypothesis,
for $n = 2$ we could even go as far as to take this as a definition of $\TTr$,
but then the generalization \Cref{conj:semi_lie_spherical} would be less obvious,
since a definition like this would not easily extend to $n > 2$.

\subsubsection{The divisor $\ZO3$ on $\MM_2$}
Turning to $\MM_2 \times \MM_2$, we will henceforth always identify $\MM_2$
with its image under the diagonal map
\begin{center}
\begin{tikzcd}
  \ZO4(1) \ar[r, "\sim"] & \MM_2 \ar[r, hook, "\Delta_{\MM_2}"] & \MM_2 \times \MM_2.
\end{tikzcd}
\end{center}

\begin{definition}
  [$\ZO3(u)$]
  Suppose now $u \in \DT$.
  Then we define the divisor $\ZO3(u)$ to be those $X \in \MM_2$
  for which we have a diagram
  \begin{center}
  \begin{tikzcd}
    X \ar[r, "\varphi"] \ar[d, dash] & X \ar[d, dash] \\
    \EE \ar[r, "u"] & \EE.
  \end{tikzcd}
  \end{center}
  Note that basically by definition, for $u \in \OO_\DD$ and $\tr u = 0$ we have
  \[ \ZO3(u) \simeq \ZO4(u) \cap \ZO4(1) \]
  when we identify $\MM_2$ with its image in $\MM_2 \times \MM_2$.
\end{definition}

\subsection{Comparison of the unitary and orthogonal special divisors}
We now relate $\ZD(u)$ to $\ZO3(u)$ through our
isomorphism $\RZ_{2,2} \xrightarrow{\sim} \MM_2 \times \MM_2$.
Recall that we have the notation
\[ \ZD(u)^\circ \coloneqq \ZD(u) - \ZD\left( \frac{u}{\varpi} \right). \]
Define $\ZO4(u)^\circ$ and $\ZO3(u)^\circ$ similarly.

\begin{lemma}
  [$\ZO3(\bar u \sqrt{\eps} u)^\circ$, {\cite{ref:CLZ}}]
  \label{lem:serre_pullback_divisor}
  Let $u \in \VV_n^-$, and consider it as an element $u \in \DD$.
  Then pullback along the Serre tensor construction identifies
  \[ \Serre^\ast \ZD(u)^\circ \simeq \ZO3(\bar u \sqrt{\eps} u)^\circ. \]
\end{lemma}
\begin{proof}
  This is shown in the in-preparation
  \cite{ref:CLZ} by Ryan C.~Chen, Weixiao Lu, and Wei Zhang.
  (Although \cite{ref:CLZ} is technically written for $F = \QQ_p$,
  that restriction is for other parts of the paper that don't affect this lemma.)
\end{proof}

\subsection{The intersection number as a triple product}
We return to the intersection number
\[ \Int((g,u), \fr) = \chi_{\RZ_{n,n}} \left(
      \Sheaf_{\TT_{\mathbf{1} \otimes \fr}(\Delta_{\ZD(u)})}
      \jiao_{\Sheaf_{\RZ_{n,n}}} \Sheaf_{\Gamma_g} \right) \]
which we will rewrite more succinctly using angle brackets as
\[ \Int((g,u), \fr) = \Big\langle
  \mathbb{T}_{\mathbf{1}_K \otimes \mathbf{1}_{K, \le r}}
  \Delta_{\ZD(u)}, \Gamma_g \Big\rangle_{\RZ_{2,2}} \]
in analogy to \cite[\S6.1]{ref:AFLspherical}.
(Note that $\Delta$ here is the diagonal map $\RZ_2 \to \RZ_{2,2}$.)
For this calculation, it would be sufficient to split
\[ \ZD(u) = \sum_{i \ge 0} \ZD(u/\varpi^i)^\circ. \]
Accordingly, let us introduce the notation
\begin{align*}
  \Int^\circ((g,u), \fr)
  &\coloneqq \Int((g,u), \fr) - \Int\left( \left( g, \frac{u}{\varpi} \right), \fr \right) \\
  &= \Big\langle \mathbb{T}_{\mathbf{1}_K \otimes \mathbf{1}_{K, \le r}}
    \Delta_{\ZD(u)^\circ}, \Gamma_g \Big\rangle_{\RZ_{2,2}}.
\end{align*}
From \Cref{lem:serre_pullback_divisor} we get
\begin{align*}
  \Int^\circ((g,u), \fr)
  &= \left< \ZO4(g \cdot \varpi^r), \; \Delta(\Serre^\ast(\ZD(u))^\circ) \right>_{\MM_2 \times \MM_2} \\
  &= \left< \ZO4(g \cdot \varpi^r), \; \ZO3(\bar u \sqrt{\eps} u)^\circ \right>_{\MM_2 \times \MM_2} \\
  &= \left< \ZO4(g \cdot \varpi^r), \; \ZO4(1)^\circ, \; \ZO4(\bar u \sqrt{\eps} u)^\circ \right>_{\MM_2 \times \MM_2} \\
  &= \left< \ZO4(g \cdot \varpi^r), \; \ZO4(1), \; \ZO4(\bar u \sqrt{\eps} u) \right>_{\MM_2 \times \MM_2} \\
  &\qquad- \left< \ZO4(g \cdot \varpi^r), \; \ZO4(1), \; \ZO4\left( \frac{\bar u \sqrt{\eps} u}{\varpi} \right) \right>_{\MM_2 \times \MM_2}.
\end{align*}
In that case we have
\begin{equation}
  \Int((g,u), \fr)
  = \sum_{i \ge 0} \Int^\circ\left( \left( g, \frac{u}{\varpi^i} \right), \fr \right).
  \label{eq:int_to_int_circ}
\end{equation}

\subsection{The formula of Gross-Keating}
\label{sec:GK}
In what follows, we let
\[ \left\langle x, y \right\rangle _0
  = \frac{\left\langle x,y \right\rangle + \ol{\left\langle x,y \right\rangle}}{2} \]
denote half the $E/F$-trace of $\left\langle x,y \right\rangle \in E$.
Let $\ODT \coloneqq \OO_\DD \cap \DT$.

\begin{proposition}
  [Gross-Keating]
  \label{prop:GK}
  Let $x,y \in \ODT$ and let
  \begin{align*}
    n_1 &= \min\left( v(\left\langle x,x \right\rangle _0), v(\left\langle x,y \right\rangle _0), v(\left\langle y,y \right\rangle _0) \right) \\
    n_2 &= v\left( \left\langle x,x \right\rangle _0 \left\langle y,y \right\rangle _0 - \left\langle x,y \right\rangle^2_0 \right) - n_1
  \end{align*}
  so that $0 \le n_1 \le n_2$.
  Then if $n_1$ is odd, we have
  \[
    \left< \ZO4(1), \; \ZO4(x), \; \ZO4(y) \right>_{\MM_2 \times \MM_2}
    = \sum_{j=0}^{\frac{n_1-1}{2}} (n_1+n_2-4j) q^j
  \]
  while if $n_1$ is even we instead have
  \[
    \left< \ZO4(1), \; \ZO4(x), \; \ZO4(y) \right>_{\MM_2 \times \MM_2}
    = \frac{n_2-n_1+1}{2} q^{n_1/2} + \sum_{j=0}^{n_1/2-1} (n_1+n_2-4j) q^j.
  \]
\end{proposition}
\begin{proof}
  This is a rewriting of \cite[Proposition 14.6]{ref:Kudla1997}
  which is itself a special case of \cite[Proposition 5.4]{ref:GK}.
\end{proof}

We now compute all the quantities needed to invoke the Gross-Keating formula.
We start by writing
\begin{align*}
  \bar u \sqrt{\eps} u
  &= (\ol s - t \Pi) \sqrt{\eps} (s + t \Pi) \\
  &= (\ol s - t \Pi) (s \sqrt{\eps} + t \sqrt{\eps} \Pi) \\
  &= s \ol s \sqrt{\eps} - t \ol{s\sqrt{\eps}} \Pi + \ol s t \sqrt{\eps} \Pi - t \ol{t \sqrt{\eps}} \varpi \\
  &= (s \ol s + t \ol t \varpi) \sqrt{\eps} + 2 \ol s t \sqrt{\eps} \Pi.
\end{align*}
We now invoke \Cref{assume:st_zero} to simplify this to just
\[ \bar u \sqrt{\eps} u = (s \ol s + t \ol t \varpi) \sqrt{\eps}. \]
This assumption will also let us write
\[ (s \ol s + t \ol t \varpi)^2 = (s \ol s - t \ol t \varpi)^2 = (\Nm u)^2. \]

Next we consider
\[ g \cdot \varpi^r = \varpi^r(\alpha + \beta \Pi). \]
(Here the action of $g$ is the one on $\MM_2$,
which is why we write $g \cdot \varpi^r$ rather than $g(\varpi^r$).)
From now on, let's write
\[ \alpha = \alpha_0 + \alpha_1 \sqrt{\eps} \]
for $\alpha_0, \alpha_1 \in F$.
Then we use the notation
\begin{align*}
  x &\coloneqq \bar u \sqrt{\eps} u = (s \bar s + t \bar t \varpi) \sqrt{\eps} \in \ODT \\
  y &\coloneqq (g \cdot \varpi^r)_- = \varpi^r(\alpha_1\sqrt{\eps} + \beta \Pi) \in \ODT.
\end{align*}
Then we can compute
\begin{align*}
  \left\langle x,x \right\rangle _0
  &= \Nm x \\
  &= \Norm_{E/F}((s \bar s + t \bar t \varpi)\sqrt{\eps}) \\
  &= -\eps(s \bar s + t \bar t \varpi)^2 = -\eps (\Nm u)^2 \\
  \left\langle y,y \right\rangle _0
  &= \Nm \left( \varpi^r(\alpha_1\sqrt{\eps} + \beta \Pi) \right)  \\
  &= \varpi^{2r} (-\alpha_1^2\eps - \beta\bar\beta \varpi) \\
  \left\langle x,y \right\rangle _0 &= (\bar x y)_0 \\
  &= \left[
    -(s \bar s + t \bar t \varpi) \sqrt{\eps}
    \cdot
    \varpi^{r} (\alpha_1\sqrt{\eps} + \beta \Pi)
  \right]_0 \\
  &= - \varpi^r \alpha_1 \eps (s \bar s + t \bar t \varpi).
\end{align*}
This lets us compute the determinant
\begin{align*}
  \left\langle x,x \right\rangle _0 \left\langle y,y \right\rangle _0 - \left\langle x,y \right\rangle^2_0
  &= - \eps (\Nm u)^2\cdot \varpi^{2r}
    (-\alpha_1^2 \eps - \beta \bar\beta \varpi)
  - (\varpi^r \alpha_1 \eps (s \bar s + t \bar t \varpi))^2 \\
  &= \eps (\Nm u)^2\cdot \varpi^{2r} \cdot
    (\varpi \beta \bar\beta).
\end{align*}
Hence we arrive at an exact formula for
\[ \left< \ZO4(1), \; \ZO4(x), \; \ZO4(y) \right>_{\MM_2 \times \MM_2} \]
in terms of the valuations of the above formulas,
which we will explicate in the next section after matching $(g,u)$
to the corresponding element in $(S_2(F) \times V_2'(F))\rs$.

\section{Proof of \Cref{thm:semi_lie_n_equals_2}}
\label{ch:finale}
We now put together all the results from the previous chapters to
prove \Cref{thm:semi_lie_n_equals_2}.

On the orbital side, we assume $\guv$ is as in \Cref{lem:semi_lie_params} throughout this chapter.
On the geometric side, we assume $(g,u)$ are as described in \Cref{sec:g_u_invariants}.

\subsection{Matching $\guv$ and $(g,u)$, and the invariants for the matching}
\subsubsection{The invariants for the orbit of $\guv$}
Recall the relations in \Cref{lem:semi_lie_params}.
The invariants in this case as described in \Cref{def:matching_semi_lie} are:
\begin{itemize}
  \ii $\Tr \gamma = a + d$
  \ii $\det \gamma = ad - bc$
  \ii $\vv^\top \uu = e$
  \ii $\vv^\top \gamma \uu = \begin{pmatrix} 0 & e \end{pmatrix}
  \begin{pmatrix} a & b \\ c & d \end{pmatrix} \begin{pmatrix} 0 \\ 1 \end{pmatrix} = de$.
\end{itemize}
Note that the parameters $b$ and $c$ are absent; but we have
\[ v(b) + v(c) = v(\det \gamma - a d). \]

\subsubsection{Matching}
We now take these results and line them up with \Cref{lem:g_u_invariants}
to deduce the following lemma.

\begin{lemma}
  [Explicit matching of invariants of $\U(\VV_2^-)$ and $(S_2(F) \times V'_2(F))\rs^-$]
  \label{lem:finale_match}
  Let
  \[ g = \lambda\inv \begin{pmatrix}
    \alpha & \bar\beta \varpi \\
    \beta & \bar\alpha
    \end{pmatrix} \in \U(\VV_2^-) \]
  and $u = s + t \Pi$ with $s t = 0$ and $v(\lambda) = 0$.
  Suppose $(g,u)$ matches with
  \[ (\gamma, \uu, \vv^\top) = \left( \begin{pmatrix} a & b \\ c & d \end{pmatrix},
    \begin{pmatrix} 0 \\ 1 \end{pmatrix}, \begin{pmatrix} 0 & e \end{pmatrix} \right)
    \in (S_2(F) \times V'_2(F))\rs. \]
  Then we have
  \begin{align*}
    a &= \begin{cases}
      \lambda\inv \bar\alpha & \text{if } s = 0 \\
      \lambda\inv \alpha & \text{if } t = 0 \\
    \end{cases} \\
    d &= \begin{cases}
      \lambda\inv \alpha & \text{if } s = 0 \\
      \lambda\inv \bar\alpha & \text{if } t = 0 \\
    \end{cases} \\
    bc &= \lambda^{-2} \beta \bar\beta \varpi \\
    e &= \Nm u.
  \end{align*}
  Thus we also have the identity
  \[ v(d - a) = v(\alpha_1). \]
\end{lemma}
\begin{proof}
  Setting the invariants from the previous subsection equal
  to the ones determined in \Cref{lem:g_u_invariants} gives
  \begin{align*}
    a + d &= \lambda\inv (\alpha + \bar\alpha) \\
    \det \gamma = ad - bc &= \det g = \lambda^{-2} (\alpha \bar\alpha - \beta \bar\beta \varpi) \\
    e &= \Nm u \\
    de &= \begin{cases}
      \lambda\inv \bar\alpha \Nm u & \text{if } s = 0 \\
      \lambda\inv \alpha \Nm u & \text{if } t = 0.
    \end{cases}
  \end{align*}
  So the equations for $e$, $d$ and $a$ are immediate.
  In both cases we get $ad = \lambda^{-2} \alpha \bar \alpha$ and hence
  \[ \lambda^{-2} \beta \bar \beta \varpi
    = -(\det g -  \lambda^{-2} \alpha \bar \alpha)
    = -(\det \gamma - ad) = bc. \qedhere \]
\end{proof}

\begin{remark}
  [Deriving \Cref{lem:semi_lie_params} from \Cref{lem:finale_match}]
  Note that many of the assumptions in \Cref{lem:semi_lie_params}
  can also be extracted from \Cref{lem:finale_match}.
  For example, taking the valuation of
  \[ \lambda \bar \lambda = \alpha \bar \alpha - \beta \bar \beta \varpi \]
  implies that (since the left-hand side is a unit)
  \[ v(a) = v(\alpha) = 0 <  2v(\beta) + 1 = v(b) + v(c). \]
  So indeed $v(1-a \bar a) = v(\bar bc)$ must be odd.
\end{remark}

\subsection{Translation of the Gross-Keating data to the orbital side}
We combine the results of \Cref{sec:GK} with \Cref{lem:finale_match}.
Retaining the notation
\begin{align*}
  x &\coloneqq \bar u \sqrt{\eps} u = (s \bar s + t \bar t \varpi) \sqrt{\eps} \in \ODT \\
  y &\coloneqq (g \cdot \varpi^r)_- = (\alpha_1\sqrt\eps + \beta \Pi) \in \ODT
\end{align*}
from \Cref{sec:GK}, we obtain the following:
\begin{align*}
  v(\left\langle x,x \right\rangle _0)
    &= 2 v(\Nm u) \\
    &= 2v(e) \\
  v\left( \left\langle x,y \right\rangle _0 \right)
    &= r + v(\alpha_1) + v(\Nm u) \\
    &= r + v(d-a) + v(e) \\
  v(
    \left\langle x,x \right\rangle _0 \left\langle y,y \right\rangle _0
    - \left\langle x,y \right\rangle^2_0
  )
    &= 2r + 2 v(\Nm u) + v(\beta \bar\beta \varpi) \\
    &= 2r + 2v(e) + v(b) + v(c).
\end{align*}
Notice that the last valuation is odd.
Therefore we can always extract $v(\left\langle y,y \right\rangle _0)$ by writing
\begin{align*}
  v\left(\left\langle x,x \right\rangle _0\right) + v\left(\left\langle y,y \right\rangle _0\right)
  &= \min \left( v\left( \left\langle x,x \right\rangle _0 \left\langle y,y \right\rangle _0 - \left\langle x,y \right\rangle^2_0 \right),
    v(\left\langle x,y \right\rangle^2_0) \right) \\
  \implies  v\left(\left\langle y,y \right\rangle _0\right)
  &= \min(2r + 2v(e) + v(b) + v(c), 2r + 2v(d-a) + 2v(e)) - 2v(e) \\
  &= 2r + \min(v(b) + v(c), 2v(d-a)).
\end{align*}
Hence, we have
\begin{align*}
  &\phantom= \min\left(
    v(\left\langle x,x \right\rangle _0),
    v(\left\langle x,y \right\rangle _0),
    v(\left\langle y,y \right\rangle _0),
  \right) \\
  &= \min\left( 2v(e), r+v(d-a)+v(e), 2r+\min(v(b)+v(c), 2v(\varsigma)) \right) \\
  &= \min\left( 2v(e), r+v(d-a)+v(e), 2r+v(b)+v(c), 2r+2v(\varsigma) \right) \\
  &= \min\left( 2v(e), v(b)+v(c)+2r, 2v(d-a)+2r \right)
\end{align*}
where we can drop $r+v(d-a)+v(e)$ from the minimum because it equals
$\frac{2v(e) + (2v(d-a)+2r)}{2}$.

Now recall the right-hand side of \Cref{prop:GK}, that is
\[
  \begin{cases}
    \sum_{j=0}^{\frac{n_1-1}{2}} (n_1+n_2-4j) \cdot q^j & \text{if } n_1 \equiv 1 \pmod 2 \\
    \frac{n_2-n_1+1}{2} q^{n_1/2} + \sum_{j=0}^{n_1/2-1} (n_1+n_2-4j) \cdot q^j & \text{if } n_1 \equiv 0 \pmod 2.
  \end{cases}
\]
for $0 \le n_1 \le n_2$.
Then apply \Cref{prop:GK} to obtain that
\[
  \left< \ZO4(1), \;
    \ZO4(\bar u \sqrt{\eps} u), \;
    \ZO4((g \cdot \varpi^r)_-) \right>_{\MM_2 \times \MM_2}
\]
is equal to the above formula applied at
\begin{align*}
  n_1 &\coloneqq \min(2v(e), v(b)+v(c)+2r, 2v(d-a)+2r) \\
  n_2 &\coloneqq 2v(e) + v(b) + v(c) + 2r - n_1.
\end{align*}
Note that
\[ n_1 + n_2 = 2v(e) + v(b) + v(c) + 2r. \]
For brevity, we henceforth introduce the symbol $\GK$ for the sum above;
hence we have
\begin{equation}
  \begin{aligned}
    &\GK(r, v(b), v(c), v(e), v(d-a)) \\
    &= \left< \ZO4(1), \; \ZO4(\bar u \sqrt{\eps} u), \; \ZO4((g \cdot \varpi^r)_-) \right>_{\MM_2 \times \MM_2}.
  \end{aligned}
  \label{eq:GKdef}
\end{equation}

In that case we also have
\begin{align*}
  &\GK(r, v(b), v(c), v(e)-1, v(d-a)) \\
  &= \left< \ZO4(1), \;
    \ZO4\left(\frac{\bar u \sqrt{\eps} u}{\varpi}\right), \;
    \ZO4((g \cdot \varpi^r)_-) \right>_{\MM_2 \times \MM_2}.
\end{align*}
Subtracting the two gives
\begin{equation}
  \begin{aligned}
    \Int^\circ((g,u), \fr) &=
      \left< \ZO4(1), \;
        \ZO4(\bar u \sqrt{\eps} u)^\circ, \;
        \ZO4((g \cdot \varpi^r)_-) \right>_{\MM_2 \times \MM_2} \\
    &= \GK(r, v(b), v(c), v(e), v(d-a)) \\
    &\qquad- \GK(r, v(b), v(c), v(e)-1, v(d-a)).
  \end{aligned}
  \label{eq:int_subtract}
\end{equation}

\subsection{Base change}
\label{sec:finale_base_change}
The base change for $n=2$ was already calculated in \cite{ref:AFLspherical}
and we simply recall the result here.

As in \Cref{ch:jiao}, we define
\begin{align*}
  \fr &\coloneqq \mathbf{1}_{\varpi^{-r} \Mat_2(\OO_E) \cap \U(\VV_n^+)} \in \HH(\U(\VV_n^+)) \\
  \mathbf{1}_{K, r} &\coloneqq \fr - \mathbf{1}_{K, \le (r-1)} \\
  &= \mathbf{1}_{\varpi^{-r} \Mat_2(\OO_E) \cap \U(\VV_n^+)} \in \HH(\U(\VV_n^+)) \\
\end{align*}

\begin{lemma}
  [{\cite[Lemma 7.1.1]{ref:AFLspherical}}]
  \label{lem:finale_base_change}
  For $n = 2$ and $r \ge 1$ we have
  \begin{align*}
    \BC_S^{\eta}(
      \mathbf{1}_{K'_{S, \le r}}
      + \mathbf{1}_{K'_{S, \le (r-1)}}
      )
      &= (-1)^r \mathbf{1}_{K, r} \\
      &= (-1)^r(\mathbf{1}_{K, \le r} - \mathbf{1}_{K, \le (r-1)}).
  \end{align*}
\end{lemma}
\begin{proof}
  This follows directly from \cite[Equation (7.1.9)]{ref:AFLspherical}.
\end{proof}

\subsection{Transfer factor}
As stated in \eqref{eq:semi_lie_transfer}, the transfer factor is
\[ \omega\guv = (-1)^{v(c)+1}. \]

\subsection{Comparison to orbital formula}
In what follows, we introduce the shorthand
\[ \partial\Orb(r, v(b), v(c), v(e), v(d-a))
  \coloneqq \partial \Orb(\guv, \mathbf{1}_{K'_{S, \le r}}). \]
The main claim is that the following formula holds:
\begin{theorem}
  [$\GK$ is a difference of two orbitals]
  \label{thm:miracle}
  We have
  \begin{align*}
    \frac{(-1)^{r+v(c)}}{\log q}\Big(&\partial\Orb(r, v(b), v(c), v(e), v(d-a)) \\
    &+ \partial\Orb(r, v(b), v(c), v(e)-1, v(d-a))\Big) \\
    &= \GK(r, v(b), v(c), v(e), v(d-a)).
  \end{align*}
\end{theorem}

We continue to use the notation $N$ and $\varkappa$ from \Cref{ch:orbitalFJ2} defined by
\begin{align*}
  N &\coloneqq \min \left( v(e),
    \tfrac{v(b)+v(c)-1}{2} + r, v(d-a) + r \right) \\
  \varkappa &\coloneqq v(e) - v(d-a) - r \ge 0
\end{align*}
and prove \Cref{thm:miracle} by exhausting the
cases based on which value of $N$ is smallest.

\subsubsection{Proof of \Cref{thm:miracle}
  when $v(e) \le \tfrac{v(b)+v(c)-1}{2} + r$
  and $v(e) \le v(d-a) + r$}
In the Gross-Keating formula we have simply
\begin{align*}
  n_1 &= 2v(e) \\
  n_2 &= v(b) + v(c) + 2r.
\end{align*}
Hence, we have
\begin{align*}
  \GK(r, v(b), v(c), v(e), v(d-a))
  &= \frac{-2v(e)+v(b)+v(c)+2r+1}{2} q^{v(e)} \\
  &\qquad+ \sum_{j=0}^{v(e)-1} (2v(e)+v(b)+v(c)+2r-4j) \cdot q^j.
\end{align*}
On the orbital side, we refer to \Cref{cor:semi_lie_derivative_single}
and compare it to the single instance of Gross-Keating above.
The exponent of $j$ runs up to $v(e)$ in one case and $v(e)-1$ in the second;
that is we need
\begin{align*}
  &\phantom= \sum_{j=0}^{v(e)} \left( \frac{2v(e)+v(b)+v(c)+1}{2} + r - 2j \right) \cdot q^j \\
  &\qquad+ \sum_{j=0}^{v(e)-1} \left( \frac{2(v(e)-1)+v(b)+v(c)+1}{2} + r - 2j \right) \cdot q^j \\
  &= \frac{-2v(e)+v(b)+v(c)+2r+1}{2} q^{v(e)}
  + \sum_{j=0}^{v(e)-1} (2v(e)+v(b)+v(c)+2r-4j) \cdot q^j
\end{align*}
which is obvious.

\subsubsection{Proof of \Cref{thm:miracle}
  when $\tfrac{v(b)+v(c)-1}{2} + r < v(e)$
  and $v(b)+v(c) < 2v(d-a)$}

Set $N = \frac{v(b)+v(c)-1}{2} + r$.
In the Gross-Keating formula we have simply
\begin{align*}
  n_1 &= 2N+1 \\
  n_2 &= 2v(e).
\end{align*}
Hence
\[ \GK(r, v(b), v(c), v(e), v(d-a))
  = \sum_{j=0}^N (2v(e)+v(b)+v(c)+2r-4j) \cdot q^j. \]
We compare this to \Cref{cor:semi_lie_derivative_single}; we check
\begin{align*}
  &\phantom= \sum_{j=0}^{N} \left( \frac{2v(e)+v(b)+v(c)+1}{2} + r - 2j \right) \cdot q^j \\
  &\qquad+ \sum_{j=0}^{N} \left( \frac{2(v(e)-1)+v(b)+v(c)+1}{2} + r - 2j \right) \cdot q^j \\
  &= \sum_{j=0}^N (2v(e)+v(b)+v(c)+2r-4j) \cdot q^j
\end{align*}
which is clear.

\subsubsection{Proof of \Cref{thm:miracle}
  when $v(d-a)+r < v(e)$ and $2v(d-a) < v(b)+v(c)$}
Hence $N = v(d-a) + r$ and $\varkappa \coloneqq v(e) - (v(d-a)+r) > 0$.
In the Gross-Keating side formula, we now have
\begin{align*}
  n_1 &\coloneqq 2v(d-a) + 2r = 2N \\
  n_2 &= 2v(e) + v(b) + v(c) - 2v(d-a).
\end{align*}
Hence
\begin{align*}
  \GK(r, v(b), v(c), v(e), v(d-a))
  &= \frac{2v(e)+v(b)+v(c)-4v(d-a)-2r+1}{2} q^{N} \\
    &\qquad+ \sum_{j=0}^{N-1} (2v(e)+v(b)+v(c)+2r-4j) \cdot q^j.
\end{align*}
This time, the relevant combination of \Cref{cor:semi_lie_derivative_single} is
\begin{align*}
  &\phantom= \sum_{j=0}^N q^j
  \cdot \left( \frac{2v(e)+v(b)+v(c)+1}{2} + r - 2j \right) \\
  &\qquad + q^{N} \cdot
  \begin{cases}
    -\frac{\varkappa}{2} & \text{if }\varkappa \equiv 0 \pmod 2 \\
    \frac{\varkappa}{2} - \left( v(e)+\frac{v(b)+v(c)}{2}-2v(d-a)-r \right)
    & \text{if }\varkappa \equiv 1 \pmod 2
  \end{cases} \\
  &\qquad + \sum_{j=0}^N q^j
  \cdot \left( \frac{2(v(e)-1)+v(b)+v(c)+1}{2} + r - 2j \right) \\
  &\qquad + q^{N} \cdot
  \begin{cases}
    -\frac{\varkappa-1}{2} & \text{if }\varkappa-1 \equiv 0 \pmod 2 \\
    \frac{\varkappa-1}{2} - \left( (v(e)-1)+\frac{v(b)+v(c)}{2}-2v(d-a)-r \right)
    & \text{if }\varkappa-1 \equiv 1 \pmod 2 \\
  \end{cases} \\
  &= \sum_{j=0}^N q^j
  \cdot \left( 2v(e)+v(b)+v(c) + 2r - 4j \right) \cdot q^j \\
  &\qquad + q^{N} \cdot
  \begin{cases}
    -\frac{\varkappa}{2} + \frac{\varkappa-1}{2} - \left( (v(e)-1) +\frac{v(b)+v(c)}{2}-2v(d-a)-r \right)
    & \text{if }\varkappa \equiv 0 \pmod 2 \\
    -\frac{\varkappa-1}{2} + \frac{\varkappa}{2} - \left( v(e)+\frac{v(b)+v(c)}{2}-2v(d-a)-r \right)
    & \text{if }\varkappa \equiv 1 \pmod 2 \\
  \end{cases} \\
  &= \sum_{j=0}^N q^j
  \cdot \left( 2v(e)+v(b)+v(c) + 2r - 4j \right) \cdot q^j \\
  &\qquad + q^{N} \cdot
    \left( \half - \left( v(e)+\frac{v(b)+v(c)}{2}-2v(d-a)-r \right) \right) \\
  &= \sum_{j=0}^{N-1} q^j
  \cdot \left( 2v(e)+v(b)+v(c) + 2r - 4j \right) \cdot q^j \\
  &\qquad + q^{N} \cdot
    \left( \left( 2v(e)+v(b)+v(c) + 2r - 4N \right) +
    \half - \left( v(e)+\frac{v(b)+v(c)}{2}-2v(d-a)-r \right) \right).
\end{align*}
The coefficient of $q^N$ is given by
\begin{align*}
  &\phantom= \left( 2v(e)+v(b)+v(c) + 2r - 4(v(d-a)-r) \right) +
  \half - \left( v(e)+\frac{v(b)+v(c)}{2}-2v(d-a)-r \right) \\
  &= \frac{2v(e)+v(b)+v(c)-4v(d-a)-2r+1}{2}
\end{align*}
which matches the one from Gross-Keating.
Hence \Cref{thm:miracle} is completely proved.

\subsection{Conclusion (proof of \Cref{thm:semi_lie_n_equals_2})}
From \Cref{thm:miracle} we have
\begin{align*}
  \GK(r, v(b), v(c), v(e), v(d-a))
  &= \frac{(-1)^{r+v(c)}}{\log q} \Big(
      \partial \Orb(r, v(b), v(c), v(e), v(d-a)) \\
      &\qquad + \partial \Orb(r, v(b), v(c), v(e)-1, v(d-a))
    \Big) \\
  \GK(r, v(b), v(c), v(e)-1, v(d-a))
  &= \frac{(-1)^{r+v(c)}}{\log q} \Big(
      \partial \Orb(r, v(b), v(c), v(e)-1, v(d-a)) \\
      &\qquad + \partial \Orb(r, v(b), v(c), v(e)-2, v(d-a))
    \Big)
\end{align*}
so subtraction (and recalling \eqref{eq:int_subtract}) gives
\begin{equation}
  \begin{aligned}
    \Int^\circ((g,u), \fr)
    &= \GK(r, v(b), v(c), v(e), v(d-a)) \\
    &\qquad- \GK(r, v(b), v(c), v(e)-1, v(d-a)) \\
    &= \frac{(-1)^{r+v(c)}}{\log q} \Big(
        \partial \Orb(r, v(b), v(c), v(e), v(d-a)) \\
        &\qquad- \partial \Orb(r, v(b), v(c), v(e)-2, v(d-a))
      \Big).
  \end{aligned}
  \label{eq:descent_by_two}
\end{equation}
We now show that \eqref{eq:descent_by_two} implies \Cref{thm:semi_lie_n_equals_2}.
Because $r = 0$ is known already, it suffices to verify for $r > 0$.

Suppose we sum \eqref{eq:descent_by_two} with $u$ replaced by
$u/\varpi^i$ for $i = 0, 1, \dots$.
The left-hand side equals $\Int((g,u), \fr)$ by \eqref{eq:int_to_int_circ}.
On the right-hand side this has the effect of decreasing $v(e)$ by $2$ since $e = \Nm u$.
Hence the sum of the right-hand sides telescopes and gives us the identity
\begin{equation}
  \Int((g,u), \mathbf{1}_{K, \le r})
  = \frac{(-1)^{v(c)+r}}{\log q}
    \partial\Orb(\guv, \mathbf{1}_{K'_{S, \le r}}).
  \label{eq:match_penultimate}
\end{equation}
Subtracting the same equation from itself with $r$ replaced by $r-1$ gives
\begin{align*}
  &\Int((g,u), \mathbf{1}_{K, \le r} - \mathbf{1}_{K, \le (r-1)}) \\
  &= \frac{(-1)^{v(c)+r}}{\log q}
    \partial \Orb(\guv, \mathbf{1}_{K'_{S, \le r}} + \mathbf{1}_{K'_{S, \le (r-1)}}).
\end{align*}
which, since $(-1)^{v(c)} = -\omega\guv$, becomes
\begin{equation}
  \begin{aligned}
    &\Int((g,u), (-1)^r(\mathbf{1}_{K, \le r} - \mathbf{1}_{K, \le (r-1)})) \\
    &= \frac{-\omega\guv}{\log q}
    \partial \Orb(\guv, \mathbf{1}_{K'_{S, \le r}} + \mathbf{1}_{K'_{S, \le (r-1)}}).
  \end{aligned}
  \label{eq:match_final}
\end{equation}
But \Cref{lem:finale_base_change} says that
$(-1)^r(\mathbf{1}_{K, \le r} - \mathbf{1}_{K, \le (r-1)}) \in \HH(\U(\VV_n^+))$
matches $\mathbf{1}_{K'_{S, \le r}} + \mathbf{1}_{K'_{S, \le (r-1)}} \in \HH(S_2(F))$
for any $r \ge 0$.
And hence from the $r = 0$ case we can inductively conclude
\Cref{thm:semi_lie_n_equals_2} for $r > 0$, completing the proof.

\subsection{A particularly clean formula for a certain intersection number}
We mention in particular that the value of
\[ \Int^\circ((g,u), 1_{K,r})
  = \Big\langle \mathbb{T}_{\mathbf{1}_K \otimes \mathbf{1}_{K, r}}
    \Delta_{\ZD(u)^\circ}, \Gamma_g \Big\rangle_{\RZ_{2,2}} \]
(note the change from $\fr$ to $\mathbf{1}_{K, \le r}$ here)
has a particularly clean formula that seems worth mentioning.
We phrase this entirely based on the quantities in the geometric side to keep in self-contained.

\begin{theorem}
  [$\Int^\circ((g,u), 1_{K,r})$]
  \label{thm:clean_intersection}
  Let $r \ge 1$ and $v(\Nm u) > 0$ for $u \in \VV_2^-$, and let
  \[ g = \lambda^{-1}
    \begin{pmatrix} \alpha & \bar\beta \varpi \\ \beta & \bar\alpha \end{pmatrix}
    \in \U(\VV_2^-) \]
  where $v(\lambda) = 0$.
  Then
  \[ \Big\langle \mathbb{T}_{\mathbf{1}_K \otimes \mathbf{1}_{K, r}}
    \Delta_{\ZD(u)^\circ}, \Gamma_g \Big\rangle_{\RZ_{2,2}} \]
  is equal to
  \[
    \begin{cases}
      (C+1) q^{N} + (C+2) q^{N-1}
        & \text{if } v(\Nm u)-r = v(\alpha - \bar\alpha) \le v(\beta) \\
      2q^N & \text{if } v(\beta) + r < \min(v(\Nm u), v(\alpha - \bar\alpha) + r) \\
      q^N + q^{N-1} & \text{otherwise}
    \end{cases}
  \]
  where
  \[ N = \min(v(\Nm u), v(\beta) + r, v(\alpha-\bar\alpha) + r) \]
  and we write
  \[ C = v(\beta) - v(\alpha - \bar\alpha) \ge 0 \]
  in the first case.
\end{theorem}

\begin{proof}
  Recall that
  \begin{align*}
    &
    \GK(r, v(b), v(c), v(e), v(d-a))
    \\
    &= \frac{(-1)^{r+v(c)}}{\log q} \Big(
      \partial \Orb(r, v(b), v(c), v(e), v(d-a)) \\
      &\qquad+ \partial \Orb(r, v(b), v(c), v(e)-1, v(d-a))
      \Big) \\
    &
    \GK(r-1, v(b), v(c), v(e), v(d-a))
    \\
    &= \frac{(-1)^{r-1+v(c)}}{\log q} \Big(
      \partial \Orb(r-1, v(b), v(c), v(e), v(d-a)) \\
      &\qquad- \partial \Orb(r-1, v(b), v(c), v(e)-1, v(d-a))
      \Big)
  \end{align*}
  when we subtract we obtain that
  \begin{align*}
    \Int^\circ((g,u), \mathbf{1}_{K, r})
    &= \frac{(-1)^{r+v(c)}}{\log q} \Big(
      \partial \Orb(r, v(b), v(c), v(e), v(d-a)) \\
      &\qquad+ \partial \Orb(r-1, v(b), v(c), v(e), v(d-a)) \\
      &\qquad- \partial \Orb(r, v(b), v(c), v(e)-2, v(d-a)) \\
      &\qquad+ \partial \Orb(r-1, v(b), v(c), v(e)-2, v(d-a))
    \Big).
  \end{align*}
  Gathering the first two terms lets us apply the simpler formula \Cref{cor:semi_lie_combo} twice;
  doing so gives
  \begin{align*}
    \Int^\circ((g,u), \mathbf{1}_{K, r})
    &= \left( (q^N + q^{N-1} + \dots + 1) + C q^N + C' q^{N-1} \right) \\
    &- \left( (q^{N^\flat} + q^{N^\flat-1} + \dots + 1) + C^\flat q^{N^\flat} + (C')^\flat q^{N^\flat-1} \right)
  \end{align*}
  where $N$, $C$, $C'$ are is in \Cref{cor:semi_lie_combo},
  and $N^\flat$, $C^\flat$, $(C')^\flat$ are the same quantities
  with $v(e)$ replaced by $v(e)-2$.
  Let $\varkappa$ and $\varkappa^\flat = \varkappa - 2$ be also as in \Cref{cor:semi_lie_combo}.

  We consider cases now.
  \begin{itemize}
    \ii Suppose first $v(e) \le \frac{v(b)+v(c)-1}{2}+r$ and $v(e) < v(d-a)+r$.
    Then $N = v(e)$ and $N^\flat = v(e) - 2$ and $C = C' = (C^\flat) = (C')^\flat = 0$,
    Hence in this case we have
    \[ \Int^\circ((g,u), \mathbf{1}_{K, r}) = q^{N} + q^{N-1}. \]

    \ii Next suppose $2v(d-a) > v(b) + v(c)$ and consider cases on $v(e)$.
    We only need to consider $v(e) > \frac{v(b)+v(c)-1}{2} + r$.
    \begin{itemize}
      \ii If $v(e) = \frac{v(b)+v(c)-1}{2} + r + 1$
      then we have
      \[ N = \frac{v(b)+v(c)-1}{2} + r, \qquad N^\flat = \frac{v(b)+v(c)-1}{2} + r - 1 \]
      and
      $C = 1$, $C^\flat = 0$, and $C' = (C')^\flat = 0$.
      Consequently we get
      \[ \Int^\circ((g,u), \mathbf{1}_{K, r}) = 2 q^N. \]

      \ii Once $v(e) \ge \frac{v(b)+v(c)-1}{2} + r + 2$
      we always have $N = N^\flat = \frac{v(b)+v(c)-1}{2} + r$,
      \[ C - C^\flat = (v(e)-N)-((v(e)-2)-N) = 2 \]
      and $C' = (C')^\flat = 0$.
      Hence in this case we have
      \[ \Int^\circ((g,u), \mathbf{1}_{K, r}) = 2 q^N \]
      as well.
    \end{itemize}

    \ii Finally suppose $2v(d-a) < v(b) + v(c)$ and consider cases on $v(e)$.
    We only need to consider $v(e) \ge v(d-a) + r$.
    \begin{itemize}
      \ii If $v(e) = v(d-a) + r$,
      then \[ N = v(d-a) + r, \qquad N^\flat = v(d-a) + r - 2. \]
      In this case $\varkappa = 0$ (and $\varkappa^\flat=-2$).
      So $C^\flat = (C')^\flat = 0$ but we have larger terms
      \begin{align*}
        C &= \frac{v(b) + v(c)- 2v(d-a) - 1}{2} \\
        C' &= \frac{v(b) + v(c)- 2v(d-a) + 1}{2}.
      \end{align*}
      Hence, we get an exceptional case
      \begin{align*}
        \Int^\circ((g,u), \mathbf{1}_{K, r})
        &= \frac{v(b) + v(c) - 2v(d-a) + 1}{2} q^{N} \\
        &\qquad+ \frac{v(b) + v(c) - 2v(d-a) + 3}{2} q^{N-1}
      \end{align*}

      \ii If $v(e) = v(d-a) + r + 1$,
      then we have
      \[ N = v(d-a) + r, \qquad N^\flat = v(d-a) + r - 1. \]
      In this case $\varkappa = 1$ (and $\varkappa^\flat=-1$) so we have
      $C = 0$, $C' = 1$, $C^\flat = 0 = (C')^\flat = 0$.
      Consequently we get
      \[ \Int^\circ((g,u), \mathbf{1}_{K, r}) = q^{N} + q^{N-1}. \]

      \ii Once $v(e) \ge v(d-a) + r + 2$,
      we always have $N = N^\flat = v(d-a) + r$ and
      \[ C - C^\flat = (C') - (C'^\flat) = 1 \]
      regardless of the parity of $\varkappa$.
      Hence in this case we also get
      \[ \Int^\circ((g,u), \mathbf{1}_{K, r}) = q^{N} + q^{N-1}. \]
    \end{itemize}
  \end{itemize}
  Hence, in summary we get that
  \begin{equation}
  \begin{aligned}
    &\Int^\circ((g,u), \mathbf{1}_{K, r}) \\
    &= \begin{cases}
      (C+1) q^{N} + (C+2) q^{N-1}
        & \text{if } v(e)-r = v(d-a) \le \frac{v(b)+v(c)-1}{2} \\
      2q^N & \text{if } \frac{v(b)+v(c)-1}{2} + r < \min(v(e), v(d-a)+r) \\
      q^N + q^{N-1} & \text{otherwise}.
    \end{cases}
  \end{aligned}
  \label{eq:int_circle_orbital_param}
  \end{equation}
  where
  \[ C = \frac{v(b) + v(c)- 2v(d-a) - 1}{2} \ge 0 \]
  in the first case.
  Then \eqref{eq:int_circle_orbital_param} translates via
  \Cref{lem:finale_match} into the desired claim.
\end{proof}

\printbibliography[title=References]

\end{document}